\documentclass[a4paper, 11pt]{article}
\usepackage [a4paper,left=2.5cm,bottom=2.5cm,right=2.5cm,top=2.5cm]{geometry}
\usepackage[english]{babel}
\usepackage{amssymb,amsfonts}
\usepackage{amsmath,amsthm}
\usepackage{stmaryrd}
\usepackage{ulem}
\usepackage{subcaption}
\usepackage{tabularx,array}
\usepackage{graphicx}
\usepackage{dsfont}
\usepackage{enumitem} 
\usepackage{comment}
\usepackage{mathtools}
\usepackage{xcolor}
\usepackage[colorlinks,linkcolor=blue,citecolor=blue,urlcolor=blue]{hyperref}
\usepackage{enumitem}
\usepackage[capitalise]{cleveref}

\newtheorem{remark}{Remark}
\newtheorem{theorem}{Theorem}
\newtheorem{lemma}{Lemma}
\newtheorem{corollary}{Corollary}
\newtheorem{proposition}{Proposition}
\newtheorem{definition}{Definition}

\newcommand{\keywords}[1]{\par\addvspace\baselineskip
\noindent\enspace\ignorespaces#1}
\newcommand{\ind}[1]{\mathds{1}_{\{#1\}}}

\newcommand{\bqn}{\begin{equation}}
\newcommand{\eqn}{\end{equation}}
\newcommand{\bag}{\begin{align}}
\newcommand{\eag}{\end{align}}

\newcommand{\Cloc}{\ensuremath{\mathcal{C}_{\mathrm{loc}}}}

\newcommand{\vertiii}[1]{{\left\vert\kern-0.25ex\left\vert\kern-0.25ex\left\vert #1 \right\vert\kern-0.25ex\right\vert\kern-0.25ex\right\vert}}

\newcommand{\C}{\mathbb C}
\newcommand{\ES}{\mathbb E}
\newcommand{\PE}{\mathbb P}

\newcommand{\R}{\mathbb R}
\newcommand{\ER}{\mathbb R}

\newcommand{\frC}{\mathfrak{C}}

\newcommand{\lp}{\left(}
\newcommand{\rp}{\right)}

\def\R{\mathbb{R}}
\newcommand{\E}{\mathbb{E}}

\newcommand{\ar}{\textcolor{black}}

\newcommand{\tcr}{\textcolor{black}}
\newcommand{\TCR}{\textcolor{black}}
\newcommand{\fab}{\textcolor{black}}
\newcommand{\mes}{\mu}

\newcommand{\tcb}{\textcolor{black}}
\newlist{sk}{enumerate}{1}
\setlist[sk]{label=$\mathbf{(S)}_{\kappa,R,\lambda}$:, ref=$\mathbf{(S)}_{\kappa,R,\lambda}$, wide, labelwidth=!, labelindent=0pt}

\newcommand{\mesinv}{\Pi}

\begin{document}

\title{Fast convergence rates for estimating the stationary density in SDEs driven by a fractional Brownian motion with semi-contractive drift}

\author{Chiara Amorino\thanks{Universitat Pompeu Fabra and Barcelona Graduate School of Economics, Department of Economics and Business, Ram\'on Trias Fargas 25-27, 08005, Barcelona, Spain.}, Eulalia Nualart\thanks{Universitat Pompeu Fabra and Barcelona Graduate School of Economics, Department of Economics and Business, Ram\'on Trias Fargas 25-27, 08005
Barcelona, Spain. EN acknowledges support from the Spanish MINECO grant PGC2018-101643-B-I00 .}, Fabien Panloup \thanks{LAREMA, Facult\'e des Sciences, 2 Boulevard Lavoisier, Universit\'e d'Angers, 49045
Angers, France} and Julian Sieber \thanks{Department of Mathematics, Imperial College London, 180 Queen's Gate, London
SW7 2RH, United Kingdom}} 

\maketitle

\begin{abstract}
\TCR{We study the estimation of the invariant density of additive fractional stochastic differential equations with Hurst parameter $H \in (0,1)$. We first focus on continuous observations and develop a kernel-based estimator achieving faster convergence rates than previously available. This result stems from a martingale decomposition combined with new bounds on the (conditional) convergence in total variation to equilibrium of fractional SDEs. For $H<1/2$, we further refine the rates based on recent bounds on the marginal density.  
We then extend the methodology to discrete observations, showing that the same convergence rates can be attained. Moreover, we establish concentration inequalities for the estimator and introduce a data-driven bandwidth selection procedure that adapts to unknown smoothness. Numerical experiments for the fractional Ornstein–Uhlenbeck process illustrate the estimator's practical performance.}
Finally, our results weaken the usual convexity assumptions on the drift component, allowing us to consider settings where strong convexity only holds outside a compact set. \\
\\
MSC class: primary 60G22, 60H10; secondary 62M09.
\end{abstract}


\section{Introduction}
Fractional Brownian motion (fBm) is a stochastic process that has attracted considerable attention in statistics due to its ability to capture long-range dependence and self-similarity. Its flexibility makes it a valuable modeling tool across a wide range of disciplines. \TCR{In geophysics and hydrology, fBm has been used to model phenomena such as earthquake occurrences, river flows, and precipitation patterns (see, e.g., \cite{Earthquakes, Hurst}), where long-range dependence provides a realistic description of these processes. In economics and finance, it has been applied to model and forecast stock prices, exchange rates, and commodity prices, with implications for risk management, option pricing, and the study of long memory and rough volatility (see, e.g., \cite{3rep, 4rep} for long-memory volatility and \cite{2rep, 6rep, Wan} for rough volatility). Further applications include modeling traffic patterns in high-speed data networks (\cite{Cag04}), medical processes (\cite{Lai00}), and turbulence phenomena (\cite{5rep}).} The breadth of these applications continues to motivate research on statistical inference for fBm, which remains of both theoretical and practical importance.

In this paper, we consider a time horizon $T>0$ and the stochastic differential equation (SDE) with values in $\R^d$ given by
\begin{equation}\label{eq:sde}
dX_t = \sigma\, dB_t + b(X_t)\, dt, \qquad t \in [0,T],
\end{equation}
where $X_0 \in \mathbb{R}^d$ is a given (random) initial condition, $b:\R^d \to \R^d$ is a measurable function, and $\sigma$ is a nondegenerate $d \times d$ matrix throughout the paper. 

The process $B = (B_t^1, \ldots, B_t^d)_{t \ge 0}$ is a $d$-dimensional fBm with Hurst parameter $H \in (0,1)$, defined on a complete probability space $(\Omega, \mathcal{F}, \mathbb{P})$. Its law is characterized by the covariance function
\begin{equation*}\label{eq:cov-fbm}
\mathbb{E}\big[B_t^i B_s^j\big] = \frac{1}{2} \left( |t|^{2H} + |s|^{2H} - |t-s|^{2H} \right) \delta_{i,j},
\qquad s,t \ge 0.
\end{equation*}
Even in this non-Markovian setting, where the Hurst exponent $H \neq 1/2$, the process $X$ can be embedded into an infinite-dimensional Markov structure, as shown in~\cite{Hairer} and further developed in \cref{s: model and main}. This embedding ensures the existence of a unique invariant distribution for $X$ with density $\pi: \mathbb{R}^d \to \mathbb{R}$, under standard regularity conditions on $b$ and $\sigma$ (see \cref{prop:li_sieber} and references therein).

We study the nonparametric estimation of the invariant density $\pi$ from the continuous observation of $(X_t)_{t \in [0,T]}$ as $T \to \infty$. Throughout, $X$ is assumed to evolve in its stationary regime—a standard assumption in the analysis of ergodic systems, reflecting that many processes reach equilibrium before observation.

\TCR{The estimation of invariant densities and drift coefficients in stochastic differential equations is a classical topic in statistics, with substantial theoretical and applied interest. Recent work has focused on models driven by fBm, which introduce additional challenges beyond the Markovian framework. Over the past decade, considerable progress has been made on parameter estimation for such processes; see, e.g., \cite{1rep, 7rep, 8rep}. Advances include drift estimation in additive fBm models from continuous~\cite{27HR} and discrete data~\cite{22HR}, estimation of diffusion coefficients~\cite{2HR} and the Hurst parameter~\cite{11HR}, and joint estimation of all parameters~\cite{HR}. In contrast, the nonparametric estimation of the drift or invariant density has received comparatively little attention.}

In the diffusion framework with standard Brownian noise, invariant density estimation has received substantial attention for its numerical and theoretical significance. On the computational side, efficient algorithms for approximating invariant densities, such as Markov Chain Monte Carlo methods, are well established (see, e.g., \cite{LamPag}). Theoretically, the analysis of invariant distributions provides key insights into the stability of stochastic differential systems (see \cite{Ban}) and serves as a fundamental tool for nonparametric drift estimation (see, e.g., \cite{Sch13}).  

In the classical one-dimensional diffusion setting, invariant density estimators can achieve parametric convergence rates under standard nonparametric assumptions (see Chapter~4.2 in \cite{Kut}). This result relies on properties of the diffusion local time, which restrict the method to $d=1$. In this context, \cite{Schmisser} investigates the discrete-time case and constructs penalized least squares estimators for successive derivatives of the stationary density, obtaining for $j=0$ the same convergence rate as in~\cite{ComMer2, ComMer}.  

In higher dimensions, the asymptotic statistical equivalence for drift estimation in multidimensional diffusions is established in \cite{DR}. Their results imply fast pointwise convergence rates for invariant density estimators: if the density belongs to an isotropic H\"older class $\mathcal{H}(\beta,L)$, the rate is $T^{-\frac{2\beta}{2\beta + d - 2}}$. The proof relies on variance bounds for additive diffusion functionals, derived from spectral gap inequalities and transition density estimates. Extensions to anisotropic settings were later obtained in \cite{Strauch}, and the optimality of this rate was established in~\cite{Minimax}.

As discussed above, only a few works address the nonparametric estimation of the drift coefficient and invariant density for fractional SDEs. In~\cite{MisPra, Sau}, the authors study consistency and convergence rates of nonparametric drift estimators, while~\cite{ComMar} analyzes Nadaraya–Watson-type estimators for the case $H > 1/2$.

More closely related to our objectives,~\cite{adaptive} investigates the anisotropic estimation of the invariant density associated with~\eqref{eq:sde} from discrete observations. Under a strong convexity assumption, the authors propose a data-driven method achieving, in the isotropic setting, the rate $T^{-(2 - \max(2H, 1)) \frac{\beta}{2(1 + \beta) + 2d}}$. Their analysis relies on novel concentration inequalities for the stationary process, controlling deviations between a functional and its mean through a decomposition of conditioned paths—an approach revisited in the first part of this paper.

\subsection{Our contribution}

Our approach estimates the invariant density $\pi$ using the kernel estimator $\hat{\pi}_{h,T}(x)$, whose mean squared error decomposes into bias and variance components. We first derive an upper bound for the stochastic term via a martingale decomposition (see \cref{th: upper bound}), yielding a faster convergence rate than in~\cite{adaptive}, as shown in \cref{th: rate start} and discussed in Remark~\ref{rk: cont}.
This improvement stems from the observation that, in our framework, the concentration inequality must be applied directly to the kernel estimator, whose Lipschitz norm is substantially larger than its supremum norm. Our main result is obtained by establishing new bounds for the martingale decomposition, replacing the Lipschitz norm in~\cite{adaptive} with the supremum norm, better suited to our setting.
With this adjustment, we show in Theorem \ref{th: rate start} that for any \(\varepsilon > 0\) sufficiently small, there exist constants \(c\) and \(c_\varepsilon\) and a suitable bandwidth \(\TCR{h = (h_1, \dots, h_d)}\) for the kernel estimator, depending on \(T\), \(d\), and the smoothness parameters such that for all $x \in \R^d$,
\begin{equation}{\label{eq: rate intro}}
\mathbb{E}[|\hat{\pi}_{h,T}(x) - \pi (x)|^2] \le 
\begin{cases} 
c \left(\frac{1}{T}\right)^{\frac{\TCR{\bar{\beta}}}{\TCR{\bar{\beta}} + d}} & \text{if $H<1/2$},\\[2mm]
c_\varepsilon \left(\frac{1}{T}\right)^{\frac{2\TCR{\bar{\beta}}(1 - H)}{\TCR{\bar{\beta}} + d} - \varepsilon} & \text{if $H>1/2$},
\end{cases}
\end{equation}
\TCR{where $\bar{\beta}$ denotes the harmonic mean of the smoothness parameters, i.e., 
$\frac{1}{\bar{\beta}} := \frac{1}{d} \sum_{l = 1}^d \frac{1}{\beta_l}$.}

Then, we focus on improving the convergence rate in the case $H < \frac{1}{2}$. This refinement is based on novel bounds for the density of the semigroup, specifically developed for the non-Markovian setting (\tcb{see \cite{LPS22}}). These bounds exploit the fact that the $L^1$ norm of the kernel density estimator is uniformly bounded, independently of the bandwidth $h$.
As a result, we obtain in Theorem \ref{th: rate plus} convergence rates that improve upon those in \eqref{eq: rate intro}, expressed as
\[
\mathbb{E}[|\hat{\pi}_{h,T}(x) - \pi (x)|^2] \le 
\ar{c \left(\frac{1}{T}\right)^{ \min \left( \frac{2\TCR{\bar{\beta}}}{2\TCR{\bar{\beta}} + \alpha_{d,H}}, \frac{2 \TCR{\bar{\beta}} (1 - H)}{\TCR{\bar{\beta}} + d} - \varepsilon\right)}},
\]
\ar{where $\alpha_{d,H} := \max \left(2 d - \frac{1}{H}, \frac{4d}{5 - 2H}\right)$.}

Although these rates surpass those in existing literature, they remain slower than the convergence observed in classical SDEs with $H = \frac{1}{2}$. \TCR{Further discussion of this discrepancy is provided in \cref{rk: sde}.}

\TCR{After that, we turn to a more practically relevant framework in which only discrete observations of the process are available. We show in Theorem~\ref{th: rate start discrete} that the discrete version of our estimator attains the same convergence rate as in the continuous case, with \( T = n \Delta_n \), where \( n \) denotes the number of observations and \( \Delta_n \) the discretization step. Moreover, achieving this rate requires choosing the bandwidth in a rate-optimal way, which in turn depends on the (unknown) smoothness of the invariant density. This motivates the introduction of a data-driven procedure for bandwidth selection. In particular, we establish a concentration inequality (see Proposition~\ref{th: upper bounddiscrete} for the discrete case and Theorem~\ref{prop:concentrationdisccontinu} for the continuous one), which ensures that the estimator obtained via this adaptive procedure achieves the same rate as in the previous setting, as proved in Theorem~\ref{th: rate adaptive} and Corollary~\ref{cor: rate adaptive final}. The good performance of the estimator is further illustrated in Section \ref{section:NS} through numerical experiments for the fractional Ornstein–Uhlenbeck model in dimension 3.
}

Furthermore, \tcb{another key contribution of this paper is that all results are stated under relatively relaxed convexity assumptions}. Specifically, we do not require global strong convexity; instead, we assume that the drift is \TCR{contractive outside a compact set}. In the case where $b$ derives from a potential, i.e., $b = - \nabla V$ for some $V: \mathbb{R}^d \rightarrow \mathbb{R}$, this condition implies that $V$ is uniformly and strictly convex outside a compact set. Within the compact set, however, $V$ may contain multiple wells, allowing the drift to exhibit local repulsion.

\TCR{The paper is organized as follows. In \cref{s: model and main}, we introduce the assumptions on the drift, present the framework, and state the main results. \cref{section:NS} presents the numerical results and \cref{s: preliminary bounds} establishes several preliminary estimates that are essential for bounding the stochastic term via a martingale decomposition and for controlling the semigroup density, as developed in \cref{s: mtg decomposition}. The proofs of the main statistical results are given in \cref{s: proof stat}, while technical arguments are deferred to the Appendix. In particular, \cref{sec:martproof} contains the proof of the variance bound based on the martingale decomposition, \cref{s: CI} provides the proof of the concentration inequality, which is then used in \cref{s: proof adaptive} to establish the adaptive procedure. Further technical proofs are collected in \cref{s: tech}.
}

\vskip 12pt

{\textbf{Notations.}} {We mainly use standard notations: $|\cdot|$ denotes the Euclidean norm (note that the dimension may vary though), $\langle\cdot,\cdot\rangle$ is the Euclidean scalar product, and ${\cal P}({\cal X})$ is the set of Borel probability measures on a Polish space ${\cal X}$. The law of a random variable $X$ is abbreviated by ${\cal L}(X)=\mathbb{P}(X\in\cdot)$.}

{For $p \ge 1$ and $\mu,\,\nu\in{\cal P}({\cal X})$, the Wasserstein distance is defined by
\begin{equation} \label{wass}
\mathbb{W}_{{p}}(\mu, \nu) = {\Big(\inf \int d(x,y)^p \alpha(dx,dy) \Big)^\frac{1}{p}},
\end{equation}
where the infimum is taken over all the probability measures $\alpha$ on ${\cal X} \times {\cal X}$, with marginals $\mu$ and $\nu$. 
We denote by $\left \| \mu \right \|_{\rm TV} := \sup_{A} |\mu(A)|$ the total variation norm.}

The space of the \TCR{anisotropic} $\beta$-H{\"o}lder functions with related constants bounded by $L$ is denoted by $\mathcal{H}_d (\beta, L)$ (see \cref{def:holderspace} for a precise definition), \tcb{whereas $\Cloc^{H-}(\R_+,\R^d)$ denotes the space of functions $f:\ER_+\rightarrow\ER^d$ which are locally $\alpha$-H\"older for any $\alpha<H$.}

For $a\ge0$ and $b\ge0$, one writes $a \land b$ for $\min (a,b)$,  $a \lor b$ for $\max(a,b)$ and $a\lesssim b$ if there exists a (positive) constant $c$ such $a\le c b$. One also writes $a\lesssim_\varepsilon b$
when one needs to emphasize that the constant $c$ depends on $\varepsilon$. 

\section{Model assumptions and main results}{\label{s: model and main}}

In the context of the non-Markovian dynamics described by \eqref{eq:sde}, a proper definition of the invariant distribution for the process $X$ requires embedding the dynamics into an infinite-dimensional Markovian framework. Specifically, the resulting Markov process, referred to as the Stochastic Dynamical System, can be represented as a mapping on $\mathbb{R}^d \times \mathcal{W}$, where $\mathcal{W}$ denotes a suitable space of H\"older functions from $(-\infty, 0]$ to $\mathbb{R}^d$ equipped with the Wiener measure, as detailed below.

Before giving the full embedding, we state the assumption on the drift coefficient:
\begin{sk}\hypertarget{sk}{}
\item\label{hyp1} The function $b:\ER^d \rightarrow \ER^d$ is Lipschitz continuous, with
\[
[b]_{\text{Lip}} := \sup_{x \neq y} \frac{|b(x)-b(y)|}{|x-y|},
\]
and satisfies
\begin{equation}\label{eq:conditionb}
\langle b(x)-b(y), x-y \rangle \le  
\begin{cases}
-\kappa |x-y|^2, & |x|, |y| \geq R, \\
\lambda |x-y|^2, & \text{otherwise},
\end{cases}
\end{equation}
for some constants $\lambda, R \ge 0$ and $\kappa > 0$.
\end{sk}

This hypothesis is relatively weak compared with assumptions in previous works (e.g., \cite{adaptive} for the fBm case or \cite{DR, Strauch} for classical SDEs), where global contractivity is assumed, \TCR{i.e., \eqref{eq:conditionb} holds with $R=0$}. Note that $\lambda$ may be smaller than $[b]_{\text{Lip}}$, so its specification is not redundant.  
When $b = -\nabla V$ for a potential $V$, the condition requires $V$ to be at most $\lambda$-concave for $|x| \le R$ and $\kappa$-convex for $|x| > R$. If $V \in \mathcal{C}^2(\mathbb{R}^d)$, this translates to $\nabla^2 V \le \lambda$ and $\nabla^2 V \ge -\kappa$ on the respective domains. Thus, universal strong convexity (\TCR{i.e., $b = -\nabla V$ satisfies \eqref{eq:conditionb} with $R=0$}) is not required: our results remain valid when the drift is strongly convex only outside a compact set, allowing for mildly non-convex regions.

It is well-known that \tcb{either \eqref{eq:conditionb} or the Lipschitz continuity of $b$} ensures existence and uniqueness of the solution to \eqref{eq:sde} \TCR{(see \cref{prop:li_sieber})}.
%For $H>\frac12$, the integral with respect to fBm can be interpreted pathwise as a Riemann-Stieltjes integral. For $H<\frac12$, rough path theory is required.

As shown in \cite{Hairer}, the system (\ref{eq:sde})  can be endowed with a  Markovian structure as follows.
First consider the Mandelbrot\TCR{-Van Ness} representation of the fBm. Let $W=(W_t)_{t \in \R}$ be a two-sided Brownian motion on $\R^d$. Then, \TCR{a $d$-dimensional fBm $(B_t)_{t\ge0}$ with Hurst parameter $H \in (0,1)$ admits the following moving-average representation: 
\begin{equation} \label{Man}
B_t=c_1(H) \left(\int_{-\infty}^0 (t-s)^{H -\frac12}-(-s)^{H-\frac{1}{2}} dW_s+\int_0^t (t-s)^{H-\frac{1}{2}} ds \right), \quad t \ge 0,
\end{equation}
where $c_1(H)=\frac{\sqrt{\Gamma(2H+1) \sin(\pi H)}}{\Gamma(H+\frac12)}$.
}

\TCR{Then, taking advantage of this moving-average representation, \cite{Hairer} showed that the process $Z_t:=(X_t,(W_{s+t}-W_t)_{s\le 0})$ has a Feller Markov (homogeneous) structure when the Brownian motion is realized on an appropriate H\"older space ${\cal W}$. Precisely, denoting by $\mathcal{C}^{\infty}_0(\R_{-})$ the space of $\mathcal{C}^{\infty}$ functions $w:\R_{-}\rightarrow \R^d$, with compact support  satisfying $w(0)=0$, $\mathcal{W}$ is the (Polish) closure of $\mathcal{C}^{\infty}_0(\R_{-})$ for the norm
$$
\Vert w \Vert_{\mathcal{W}}=\sup_{s,t \in \R_{-}} \frac{\vert w(t)-w(s) \vert}{\vert t-s\vert^{\frac{1-H}{2}}(1+\vert t\vert+\vert s \vert)^{1/2} }.
$$}
\TCR{Now, let us denote by $(\mathcal{Q}_t(x,w))_{t \geq 0, (x,w) \in \R^d \times \mathcal{W}}$ the Feller Markov semigroup associated with $(Z_t)_{t\ge0}$ (see \cite[Lemma 2.12]{Hairer} for details).}  A probability measure $\nu_0$ on $\R^d \times \mathcal{W}$ is  called  a generalized initial condition if  the  projection on the second coordinate is $\mathbb{P}_{-}$ when $\mathbb{P}_{-}$ denotes the distribution on $\mathbb{R}_{-}$. 
A  generalized initial condition   $\Pi$ is said to be an invariant distribution for $(X_t)_{t \geq 0}$ if for every $t \geq 0$, $\mathcal{Q}_t \mesinv=\mesinv$. We say that uniqueness of the invariant
distribution holds if the stationary regime, that is, the distribution $\overline{\mathcal{Q}}\mesinv$ of the whole process $(X_t)_{t \geq 0}$ with initial (invariant) distribution $\mesinv$, is unique.
We  denote by $\overline{\mesinv}$ the  first marginal of $\mesinv$, that is, $$\overline{\mesinv}(dx)=\int_{\mathcal{W}} \mesinv(dx, dw).$$
Let $(\mathcal{F}_t)_{t \in \mathbb{R}}$ be the natural filtration associated to the two-sided Brownian motion $(W_t)_{t \in \R}$ induced by the Mandelbrot-Van Ness representation (see \eqref{Man}).

\begin{proposition}\label{prop:li_sieber}
Assume  \ref{hyp1}. Then: 
\begin{enumerate}
\item Existence and uniqueness hold for the invariant distribution $\mesinv$.
\item There exist $\fab{\Lambda}>0$ and some constant $c,C>0$ depending only on $\kappa$, $R$ and $[b]_{Lip}$, which are such that if $\lambda\in[0,\fab{\Lambda}]$, we have for any $t \geq 0$ and any
generalized initial condition $\nu_0$, 
$$
\Vert \mathcal{L}(X_{t}^{\nu_0}) -\overline{\mesinv} \Vert_{{\rm TV}} \leq  \TCR{C} e^{-ct} {{\mathbb W}_1}(\mesinv, \bar{\nu}_0),
$$
\TCR{where $\mathbb{W}_1$ is the Wasserstein distance defined by (\ref{wass})} and
$\bar{\nu}_0$ denotes the first marginal of  $\nu_0$.

\item The marginal distribution $\overline{\Pi}$ admits a density $\pi$ with respect to the  Lebesgue measure  on $\R^d$. Furthermore, $\pi$ has the following properties:
{
\begin{enumerate}
\item There are positive constants $c_1,c_2,C_1,C_2>0$ such that
  \begin{equation}\label{eq:gaussian_bounds}
    C_1 e^{-c_1|y|^2}\leq \pi(y) \leq C_2 e^{-c_2|y|^2} \quad\forall y\in\R^d.
  \end{equation}
\item If $b$ is ${\cal C}^k$ with bounded derivatives and if furthermore $D^k b$  is $\alpha$-H\"older with $\alpha>1-(2H)^{-1}$ when $H>1/2$, then  $\pi$ is ${\cal C}^k$ with bounded derivatives.
\end{enumerate}
}
\end{enumerate}
\end{proposition}

The  first point of the above theorem follows from \cite[Theorem 1]{Hairer}. The second one comes from \cite[Theorem 1.3]{LS22a}. Finally, the last statement is a consequence of \cite[Theorem 1.1]{LPS22}. In the sequel, the notation $\pi$ will also be used for the marginal distribution.

\begin{remark}{\label{rk: beginning}} 
The exponential convergence holds if $\lambda_0$ is small enough. {As mentioned before, it is always true when $b$ is the gradient of a  convex ${\cal C}^2$-function $V$ which is uniformly strongly convex outside a compact set. The fact that $\lambda_0$ may be positive means that it can extend to ``nice'' non-convex settings. In the general non-convex case, $i.e.$ without constraints on $\lambda_0$}, \cite[Theorem 1]{Hairer} yields convergence to equilibrium but with a fractional rate of the order $t^{-\gamma}$ with $\gamma\in(0, \max_{\alpha<H} \alpha(1-2\alpha))$.
\end{remark}

We aim at estimating the invariant density $\pi$ associated to the stationary process $X$ solution to \eqref{eq:sde} via the kernel estimator and assuming that the continuous record $(X_t)_{t \in [0, T]}$ is available. We also assume that the density $\pi$ we want to estimate belongs to the \TCR{anisotropic} H{\"o}lder class $\mathcal{H}_d (\beta, L)$ defined as follows.
\begin{definition}\label{def:holderspace}
Let \TCR{$\beta = (\beta_1, ... , \beta_d)$ and ${L} = ({L}_1, ... , {L}_d)$, with $ \beta_i \ge 1$ and $L_i > 0$, for any $i \in \{1, ... , d \}$.} A function $g : \mathbb{R}^d \rightarrow \mathbb{R}$ is said to belong to the \TCR{anisotropic} H{\"o}lder class of functions $\mathcal{H}_d (\beta, L)$ if, for all $i \in \{1,\ldots,d\}$, $k = 0,1,\ldots, \lfloor \TCR{\beta_i} \rfloor$ and $t \in \mathbb{R}$,
\begin{equation*} 
\left \| D^{(k)}_i g \right \|_\infty \le \TCR{L_i}  \quad \text{ and } \quad \left \| D_i^{(\lfloor \TCR{\beta_i} \rfloor)}g(. + te_i) - D_i^{(\lfloor \TCR{\beta_i} \rfloor)}g(.) \right \|_\infty \le L |t|^{\TCR{\beta_i - \lfloor \beta_i \rfloor}},
\end{equation*}
where $D^{(k)}_i$ denotes the $k$th order partial derivative of $g$ w.r.t. the $i$th  component, 
$\lfloor \TCR{\beta_i} \rfloor$ is the integer part of $\TCR{\beta_i}$, and $e_1,\ldots,e_d$ is the canonical  basis of $\R^d$.
\end{definition}

 In  words, a function $g$ belongs to the class $\mathcal{H}_d (\beta, L)$ if all the partial derivatives of $g$ up to order $\lfloor \TCR{\beta_i} \rfloor$ are bounded and the 
 $\lfloor \TCR{\beta_i} \rfloor$th partial derivative is H\"older continuous of order $\TCR{\beta_i-\lfloor \beta_i \rfloor}$.
\begin{remark}
   As shown \TCR{in point 3.b) of} \cref{prop:li_sieber}, the smoothness of $\pi$ is directly determined by that of $b$. Moreover, a careful inspection of the proof of \cite[Theorem 1.1]{LPS22} shows that this result extends to the H\"older setting.
\end{remark}
    
 In our context, it is natural to assume that the invariant density belongs to a H\"older class as described above. Further examples of nonparametric estimation over H\"older classes can be found, for instance, in \cite{Chapitre4, Hop_Hof, Lep_Spo, Tsy}.

To achieve \TCR{our goal of estimating the invariant density}, we introduce the kernel density estimator, based on the kernel function $K: \mathbb{R} \rightarrow \mathbb{R}$ satisfying 
\begin{equation*}
\int_\mathbb{R} K(x) dx = 1, \quad \left \| K \right \|_\infty < \infty, \quad \mbox{supp}(K) \subset [-1, 1], \quad \int_\mathbb{R} K(x) x^i dx = 0,
\label{eq: properties K}
\end{equation*}
for all $i \in \left \{ 1, ... , M-1 \right \}$ with $M \ge \TCR{\max_i \beta_i}$. Such functions are referred to as kernel functions of order $M$. For a multi-index $\TCR{h = (h_1, ..., h_d)}$ which denotes the bandwidth, \TCR{with $h_l > 0$ for any $l \in \{1, ..., d \}$}, let $\mathbb{K}_h$ be defined by:
$$\mathbb{K}_h(z)=\frac{1}{\TCR{\prod_{l = 1}^d h_l}} \prod_{i=1}^d K\left(\frac{z_i}{\TCR{h_i}}\right), \quad z\in\ER^d.$$

The kernel density estimator of $\pi$ at $x= (x_1, ... , x_d) \in \mathbb{R}^d$ is given by \begin{equation}
\hat{\pi}_{h,T}(x) = \frac{1}{T} \int_0^T \mathbb{K}_h(x - X_u) du,
\label{eq: def estimator}
\end{equation}
for a given small bandwidth $h$. In particular, we assume $\TCR{h_l} < 1$ for any $\TCR{l \in \{1, ... , d \}}$. 

Kernel estimators have proven to be powerful tools in a variety of settings. For instance, \cite{Ban, Bosq9} use them to estimate the marginal density of a continuous-time process, while \cite{Chapitre4, Optimal} apply them in a jump-diffusion context, with \cite{Optimal} specifically addressing volatility estimation. More generally, as noted in the introduction, the literature on using kernel estimators to study convergence rates for invariant density estimation in diffusion processes is extensive; see, e.g., \cite{Minimax, ComMar, DR, Strauch}.

The suitability of kernel estimators for invariant density estimation is closely linked to ergodicity. Under appropriate conditions, the empirical measure $T^{-1}\int_0^T \delta_{X_u} \, du$ converges to $\pi$ as $T \to \infty$, implying that $\hat{\pi}_{h,T}(x)$ converges almost surely to $\E[\mathbb{K}_h(x - X_0)] = (\mathbb{K}_h \ast \pi)(x)$, which in turn converges to $\pi(x)$ as $h \to 0$. Our goal is to quantify these convergence properties under the assumptions of our framework.

%Let us also define the harmonic mean of the smoothness over the different directions. This is given by $\bar{\beta}$, which is such that $\frac{1}{ \bar{\beta}} := \frac{1}{d} \sum_{i \ge 1} \frac{1}{\beta_i}$. \\

 Propositions \ref{th: upper bound} and \ref{th:bound2} provide upper bounds for the stochastic component of the estimator $\hat{\pi}_{h,T}$. These bounds are then employed in Theorems \ref{th: rate start} and \ref{th: rate plus}, which constitute our main statistical results, establishing the convergence rates for the estimation of $\pi$.

 %\TCR{The proofs of Theorems \ref{th: upper bound} and \ref{th:bound2}, and those of Theorems \ref{th: rate start} and \ref{th: rate plus} can be found in Sections \ref{s: mtg decomposition} and \ref{s: proof stat}, respectively.}

%\textcolor{orange}{Throughout our results, we consistently assume that the kernel density estimator \(\hat{\pi}_{h,T}\) that we utilize has been constructed using a kernel of order \(M \ge \beta\), as specified in \eqref{eq: properties K}.}

%\tcb{Fabien: In fact, I am not completely satisfied by this assumption for all the results since in Theorem 1 and Theorem 4, it does not appear (and in a large part of the proof) whereas the other assumption appears everywhere. I definitely think that it is better to write assume that \(\pi\) belongs to \(\mathcal{H}_d(\beta, \mathcal{L})\) when necessary.} \ar{Ch: Ok, I modified it.}

\begin{proposition}\label{th: upper bound}
There exists $\Lambda>0$ such that if  \ref{hyp1} holds with $\lambda\le \Lambda$, then, for all $\varepsilon>0$ sufficiently small, there exist positive constants $c$ and $c_\varepsilon$ such that \TCR{for all $x\in\ER^d$,}
\begin{equation*}
{\rm Var}(\hat{\pi}_{h,T}(x)) \le \begin{cases} {c}{T^{-1}{(\TCR{\prod_{l = 1}^d h_l})}^{-2}} & \textnormal{if $H<1/2$},\\
\\
 {c_\varepsilon}{T^{2H-2+\varepsilon} {(\TCR{\prod_{l = 1}^d h_l})}^{-2}} &\textnormal{if $H>1/2$}.
\end{cases}
%c_\epsilon \left(\frac{1}{T}\right)^{\frac{2 \bar{\beta}(1 - H)}{  2 \bar{\beta} (1 - H) + d} + \epsilon}.
%\label{eq: conv rate}
\end{equation*}
%where $\bar{h}_d=\prod_{i=1}^d h_i.$
\end{proposition}
\fab{\cref{th: upper bound} is a direct consequence of \cref{prop: mixing} (see \cref{s: mtg decomposition}) with \( F = \mathbb{K}_h \). We refer the reader to \cref{rem:optimality} for a discussion of optimality. 
We now state our first main result, whose proof is postponed to \cref{sec:prooftheo2}. 
Recall that our kernel estimator is assumed to be of order $M \ge \max_i \beta_i$; otherwise all our theorems below remain valid after replacing $\beta_l$ by $\beta_l \wedge M$.}

%From a bias-variance decomposition and the bound on the variance gathered in \cref{th: upper bound} we obtain our first main result. 

\begin{theorem} \label{th: rate start}
\ar{Assume that \(\pi\) belongs to the anisotropic H\"older class \(\mathcal{H}_d(\beta, L)\).} Then there exists $\Lambda>0$ such that, if \ref{hyp1} holds with $\lambda\le \Lambda$, then, for all  $\varepsilon>0$ sufficiently small, there \TCR{exist} positive constants $c$ and $c_\varepsilon$ such that \TCR{for all $x \in \R^d$,}
\begin{equation*}
\mathbb{E}[|\hat{\pi}_{h,T}(x) - \pi (x)|^2] \le 
\begin{cases} {c}\left({(\TCR{\prod_{l = 1}^d h_l})}^{-2}T^{-1} \TCR{+} \TCR{\sum_{l = 1}^d h_l^{2\beta_l}}\right) & \textnormal{if $H<1/2$},\\
\\
c_\varepsilon \left({(\TCR{\prod_{l = 1}^d h_l})}^{-2}T^{2H-2+\varepsilon} \TCR{+} \TCR{\sum_{l = 1}^d h_l^{2\beta_l}}\right) &\textnormal{if $H>1/2$}.
\end{cases}
\end{equation*}
Thus, the rate optimal\footnote{\fab{By optimal, we mean the choice which minimizes the error in terms of $T$ (see \cref{sec:prooftheo2} for details)}.} choices \TCR{$h_l(T)=T^{-\frac{\bar{\beta}}{2\beta_l(\bar{\beta}+d)}}$ for $H<1/2$ and $h_l(T)=T^{-\frac{\bar{\beta}(1-H)}{\beta_l(\bar{\beta}+d)} -\varepsilon}$ for $H>1/2$, for any $l \in \{1, ... , d \}$,} lead us to
\begin{equation*}
\mathbb{E}[|\hat{\pi}_{h,T}(x) - \pi (x)|^2] \le 
\begin{cases} c(\frac{1}{T})^{\frac{\TCR{\bar{\beta}}}{\TCR{\bar{\beta}} + d}} & \textnormal{if $H<1/2$},\\
\\
c_\varepsilon (\frac{1}{T})^{\frac{2(1 - H)\TCR{\bar{\beta}}}{\TCR{\bar{\beta}} + d} - \varepsilon}  &\textnormal{if $H>1/2$},
\end{cases}
\end{equation*}
\TCR{where we recall that $\bar{\beta}$ denotes the harmonic mean of the smoothness parameters $\beta_1, ... , \beta_d$.}

\end{theorem}
\begin{remark}
A reader familiar with fractional Brownian motion might find it noteworthy that, contrary to the conventional expectation that cases with $H < \frac{1}{2}$ are particularly challenging, we obtain a better convergence rate for smaller $H$. \fab{This is primarily because, in this additive setting, the error induced by long-range dependence (which increases with $H$) has a greater impact than the roughness of the paths (which decreases with $H$).}
\end{remark}

\begin{remark}{\label{rk: cont}}
It is worth noting that, in the result above, the two limits as $H$ approaches $\frac{1}{2}$ from the left and the right coincide. However, as discussed in the introduction, for classical SDEs (where $H = \frac{1}{2}$), a faster convergence rate can be achieved. Further details are provided in Remark \ref{rk: sde}.  \end{remark}

\begin{remark} \label{rem:lb}
\TCR{In the parametric framework, consider the one-dimensional fractional Ornstein--Uhlenbeck (fOU) process $dX_t = dB_t - \theta X_t\,dt$ where $\theta > 0.$
Then, \cite{Bercu2011} shows that the maximum likelihood estimator (MLE) achieves the squared parametric rate \( T^{-1} \) for \( H > 1/2 \), and claims without proof that this rate also holds for \( H < 1/2 \). In the latter case, the least squares estimator attains the squared parametric rate \( T^{-1} \), as proven in \cite{HNZ}.
 More generally, when $H>1/2$, \cite{LNT19} shows that the MLE is asymptotically efficient in the sense of the Minimax Theorem (\cite[Theorem 1.6]{LNT19}) with the same rate.  
 In our non-parametric and multidimensional framework,  consider the isotropic fOU process, which is known to have an invariant Gaussian density for all $H \in (0,1)$. Hence, the density is infinitely smooth ($\beta = \infty$). In this case, our nonparametric estimator operates in the variance-dominated regime: choosing the bandwidth as $h_{l} = T^{-\delta}$ with an arbitrarily small $\delta > 0$ in Theorem \ref{th: rate start} makes the bias negligible, so the mean squared error (MSE) scales as 
\begin{equation*}
\mathrm{MSE}(x)\;\lesssim\;
\begin{cases}
T^{-1+2d\delta}, & H<\tfrac12,\\[2pt]
T^{\,2H-2+2d\delta+\epsilon}, & H>\tfrac12,
\end{cases}
\end{equation*}
up to arbitrarily small $\epsilon, \delta>0$. 
In general, a standard approach to prove minimax lower bounds in the nonparametric setting (see, e.g., \cite[Section~2]{Tsy}) relies on constructing two invariant densities $\pi_0$ and $\pi_1$ satisfying: (i) membership in the considered H\"older class, (ii) separation at some point $x^\ast$, and (iii) mutual absolute continuity of the corresponding stationary laws $\mathbb{P}_0$ and $\mathbb{P}_1$ of the processes
\[
dX_t = b_0(X_t)dt + \sigma dB_t
\quad \text{and} \quad
dX_t = b_1(X_t)dt + \sigma dB_t.
\]
In the standard Brownian case ($H=1/2$), condition (iii) follows from Girsanov's theorem by controlling the $L^2$ distance between $b_0$ and $b_1$. In the case that $\beta=\infty$, choosing $b_1(x)=b_0(x)+\tau^{-1/2} \delta(x)$, where $\delta$ is a smooth and bounded function with bounded derivatives, Girsanov's theorem for fBm (see \cite[Section 3]{LPS22} and \cite[Section 2.1.1]{LNT19}) suggests that if 
\[
\tau^2 \asymp
\begin{cases}
T, & H<\tfrac12,\\
T^{\,2-2H}, & H>\tfrac12,
\end{cases}
\]
then condition (iii) will hold. This suggests that, in this case, the optimal squared rate is $T^{2-2H}$ when $H>1/2$ and
$T^{-1}$ when $H<1/2$. However, the lack of an explicit correspondence $b\mapsto\pi_b$ implies that proving (ii) would require new analytical tools.
This correspondence is explicit in the case $H=1/2$ via the generator of the diffusion process and allows to obtain minimax lower bounds as shown for instance in \cite{Minimax, Optimal}. In particular, the optimal squared rate in the isotropic case is $T^{-1}$ when $d=1$ (as in the parametric case), $\frac{\log(T)}{T}$ when $d=2$, and $T^{-\frac{2\beta}{2\beta+d-2}}$ when $d \geq 3$. Extending those lower bounds to the fractional case remains an open problem, even for the fOU case.} \end{remark}

It is natural to consider how the convergence rates above compare with those in \cite{adaptive}, where the authors study a similar problem: estimating the density of an fBm in an anisotropic setting from discrete observations under more restrictive assumptions on the drift. 

For a clear comparison, we consider the isotropic case, where the smoothness is uniform across all directions. In this scenario, the methodology in \cite{adaptive} yields a convergence rate of $T^{- (2 - \max(2H, 1)) \frac{\beta}{2(1 + \beta) + 2d}}$, whereas the rate in \cref{th: rate start} can be expressed as $T^{- (2 - \max(2H, 1)) \frac{\beta}{\beta+d}}$. Since $2(1 + \beta + d) > \beta + d$, our rate clearly improves upon that in \cite{adaptive}.
The key to this improvement is replacing the Lipschitz norm in the martingale decomposition of \cite{adaptive} with the supremum norm (see \cref{prop: mixing}), which is more suitable for our setting. Notably, $\left \| \mathbb{K}_h \right \|_{Lip}$ scales like $h^{-1} \left \| \mathbb{K}_h \right \|_{\infty}$.

Moreover, for $H<1/2$, the convergence rate can be further refined using an additional bound on the density. This refinement exploits the fact that $\left \| \mathbb{K}_h \right \|_1 = 1$, whereas $\left \| \mathbb{K}_h \right \|_{\infty} \sim \TCR{(\prod_{l = 1}^d h_l)^{-1}}$. Techniques of this type, common in classical SDEs (see, e.g., \cite{asynch, DR, Strauch}), have been adapted to the non-Markovian setting, with recent developments in \cite{LPS22} and further refinements in Lemma \ref{lem:espcondkh}. These allow for an improved bound on the stochastic term for $H < \frac{1}{2}$ (\fab{see \cref{sec:prooftheo33} for the proof of the following result}):

\begin{proposition}\label{th:bound2}
Let $H<1/2$. There exists $\Lambda>0$ such that if  \ref{hyp1} holds with $\lambda\le \Lambda$, then, for all $\varepsilon>0$ sufficiently small, a  positive constant $c_\varepsilon$ exists such that \TCR{for all $x \in \R^d$},
\begin{equation*}
{\rm Var}(\hat{\pi}_{h,T}(x)) \le \TCR{\frac{(\prod_{l = 1}^d h_l)^{-2}}{T} 
\max\left(c\, (\prod_{l = 1}^d h_l)^{\frac{1}{dH}}, c\, (\prod_{l = 1}^d h_l)^{\frac{2(3-2H)}{5-2H}}, c_\varepsilon\, {T^{2H-1+\varepsilon}}\right) }.
%\frac{ h^\frac{2d(1-4H + \varepsilon)}{3-2H}}{T^{\frac{1-2H}{3-2H} -\varepsilon}}\right).
\end{equation*}
\end{proposition}
Here, we choose to include the factor $\TCR{(\prod_{l = 1}^d h_l)^{-2}} T^{-1}$ to facilitate comparison with the previous bound. {In particular, up to constants $c$ and $c_\varepsilon$, this bound is always more favorable than the previous one for any $H < 1/2$ and $\varepsilon < 1 - 2H$.}

\ar{By replacing the bound in \cref{th: upper bound} with that in \cref{th:bound2}, we obtain the following convergence rate. Recall from the introduction that $\alpha_{d,H} = 2 d - \frac{1}{H} \lor \frac{4d}{5 - 2H}$.}

\begin{theorem}\label{th: rate plus}
Let $H<1/2$ \ar{and $\pi \in \mathcal{H}_d(\beta, L)$.} There exists $\Lambda>0$ such that if  \ref{hyp1} holds with $\lambda\le \Lambda$, then, for all $\varepsilon>0$ sufficiently small, a  positive constant $c_\varepsilon$ exists such that \TCR{for all $x \in \R^d$},
\begin{equation*}
\mathbb{E}[|\hat{\pi}_{h(T),T}(x) - \pi (x)|^2] \le \TCR{c_\varepsilon \left(\frac{1}{T}\right)^{\frac{2\bar{\beta}}{2\bar{\beta} + \alpha_{d,H}} \land \frac{2 \bar{\beta} (1 - H)}{\bar{\beta} + d} - \varepsilon}},
\end{equation*}
\tcr{
where $\alpha_{d,H}=(2d -  \frac{1}{H}) \lor \frac{4d}{5 - 2H}$ and \TCR{$h_l(T)=T^{-a_l}$ with \(a_l := \frac{\bar{\beta}}{\beta_l} \min\left(\frac{1}{2\bar{\beta} + \alpha_{d,H}}, \frac{1 - H - \varepsilon}{\bar{\beta} + d}\right)\).}}
\end{theorem}
\fab{The proof of \cref{th: rate plus} is achieved in \cref{sec:prooftheo44}.}
\begin{remark}
{The results derived here crucially depend on the interaction between \(\bar{\beta}\), \(d\), and \(H\). Notably, when \(H\) is close to \(\frac{1}{2}\), we find that \(\frac{2 \bar{\beta}}{2 \bar{\beta} + \alpha_{d,H}} = \frac{\bar{\beta}}{\bar{\beta} + d - 1}\), which is always negligible compared to the second term in the minimum above, leading to the same rate as in \cref{th: rate start}. Conversely, for small \(H\), the first term in the definition of \(\alpha_{d,H}\) becomes irrelevant. Particularly, when \(H\) is near \(0\), the convergence rates obtained in \cref{th: rate plus} improve to \((\frac{1}{T})^{\frac{5 \bar{\beta}}{5 \bar{\beta} + 2d} \land \frac{2 \bar{\beta}}{\bar{\beta} + d } - \varepsilon}\), which is clearly significantly better than the rate in \cref{th: rate start}.}
\end{remark}

%\begin{corollary}{\label{cor: final rate}}
%Let $H<1/2$. There is a $\Lambda>0$ such that if  \ref{hyp1} holds with $\lambda\le \Lambda$, then, for all $\varepsilon>0$ a  constant $c_\varepsilon$ exists such that the following bounds hold true. 
%\begin{itemize}
 %   \item If $d \ge \frac{5 - 2H}{\ar{2}H(3 - 2H)}$, $\beta \ge \ar{\frac{1}{H(1 - 2H)}- d + \varepsilon} $, then 
  %  $$\mathbb{E}[|\hat{\pi}_{h,T}(x) - \pi (x)|^2] \le c\, (\frac{1}{T})^{\frac{\ar{2}\beta}{\ar{2(\beta + d)} - \frac{1}{H}}}.$$
   % \item If $d \le \frac{5 - 2H}{\ar{2}H(3 - 2H)}$ and $\beta \ge \ar{\frac{d (1 + 2H)}{(5 - 2H)(1-2H)} + \varepsilon}$, then 
   % $$\mathbb{E}[|\hat{\pi}_{h,T}(x) - \pi (x)|^2] \le c_\varepsilon\, (\frac{1}{T})^{\frac{4\beta (1 - H)}{\beta (3 - 2H) + 2d (1 + H)}- \varepsilon}.$$
    % \item If $d \ge \frac{5 - 2H}{\ar{2}H(3 - 2H)}$ and $\beta \le \ar{\frac{1}{H(1 - 2H)}- d + \varepsilon}$ or  $d \le \frac{5 - 2H}{\ar{2}H(3 - 2H)}$ and $\beta \le \ar{\frac{d (1 + 2H)}{(5 - 2H)(1-2H)} + \varepsilon}$, then 
    %$$\mathbb{E}[|\hat{\pi}_{h,T}(x) - \pi (x)|^2] \le c\,(\frac{1}{T})^{\frac{\ar{2}\beta (1 - H)}{\ar{\beta + d}}- \varepsilon}.$$
    
%\end{itemize}
%\end{corollary}

\begin{remark}{\label{rk: sde}}
\fab{Even in the refinement stated above, the convergence rate we obtained is slower compared to that of classical SDEs. This is mainly due to the decay of the covariance, which is exponential for classical SDEs but only {fractional in this setting} (on this topic, see also \cref{rem:optimality}). To control the decay of the memory, one important point of the proof is to show that the $TV$-distance between two solutions starting from initial conditions $(x,w)$ and $(y,\tilde{w})$ decreases like $t^{H-1+\varepsilon}$. This in turn implies that the $TV$-distance between the distribution of $X_{s+t}$ conditionally to the past before $X_s$ and the invariant distribution decreases like $t^{H-1+\varepsilon}$. This property may appear surprising because of the result we recalled in \cref{prop:li_sieber}, giving exponential convergence to equilibrium (see also \cite{PanRic}). Unfortunately, the lack of Markovian property induces that the convergence to equilibrium of distributions and conditional distributions may be completely different. With simpler terms, one may have exponential ergodicity in a setting where the covariance decreases with a fractional rate.}
\end{remark}

\subsection{Discrete observations and adaptive procedure}{\label{s:adaptive}}

\TCR{In the previous section, we stated our results in the continuous-time setting, without taking into account that the smoothness parameter \( \beta \) of the invariant density is unknown. In this section, we aim to derive more realistic results based on discrete-time observations and to address the adaptive procedure in the spirit of~\cite{GL}.
} 

\TCR{
We start by introducing the observation scheme. We assume that we have $n$ discrete observations of the process $X$ collected at equally spaced sampling times 
$
0 = t_0 < t_1 < \cdots < t_n = T,
$
where the discretization step $\Delta_n$ is a non-increasing positive sequence such that $n \Delta_n = T \to \infty$. Note that $\Delta_n$ depends on $n$, but we will simply write $\Delta$ to lighten the notation. We then define the estimator based on the discrete observations of the process as
\begin{equation} \label{pi_dicrete}
\check{\pi}_{h,n}(x) = \frac{1}{n} \sum_{i=0}^{n-1} \mathbb{K}_h(x - X_{i\Delta}),
\end{equation}
which is the discrete analogue of~\eqref{eq: def estimator}. We begin by stating a discrete-time counterpart of \cref{th: upper bound}, where we also provide a concentration inequality that will be instrumental for the adaptive procedure.
\begin{proposition} \label{th: upper bounddiscrete}
Assume that $\Delta\in[1/n,1]$. Then, there exists $\Lambda>0$ such that if  \ref{hyp1} holds with $\lambda\le \Lambda$, then, for all  $\varepsilon>0$ sufficiently small, there exist positive constants $c$ and $c_\varepsilon$ such that for all $x\in\ER^d$,
\begin{equation}\label{eq:MSEdiscr}
\mathrm{Var}(\check{\pi}_{h,n}(x)) \le 
\begin{cases}
c (n\Delta)^{-1} (\prod_{l = 1}^d h_l)^{-2} & \textnormal{if } H<\tfrac{1}{2},\\ 
c_\varepsilon (n\Delta)^{2H-2+\varepsilon} (\prod_{l = 1}^d h_l)^{-2} & \textnormal{if } H>\tfrac{1}{2}.
\end{cases}
\end{equation}
Furthermore, the following concentration inequality holds for all $r>0$
$$\PE\left(\left|\check{\pi}_{h,n}(x)-\E[\check{\pi}_{h,n}(x)])\right|>r\right)\le \exp\left(-\frac{r^2 (n\Delta)^{1\wedge (2-2H+\varepsilon)}}{C_\varepsilon \left \| \mathbb{K}_h \right \|^2_\infty} \right),$$
for some positive constant $C_{\epsilon}$.
\end{proposition}}

\fab{The concentration inequality of Proposition~\ref{th: upper bounddiscrete} is stated together with its continuous-time counterpart in \cref{prop:concentrationdisccontinu}. Both cases are proved at the same time introducing a reference measure (see \eqref{eq:generalfuncrefmes}) which can be the Lebesgue measure in the continuous-time setting and a counting one in the discrete-time setting. This reference measure can also be used to deduce \eqref{eq:MSEdiscr} from an adaptation of  the proof of \cref{th: rate plus}. The details of the proof of \eqref{eq:MSEdiscr}  are left to the reader.
\begin{theorem} \label{th: rate start discrete}
Assume that \(\pi\) belongs to the H\"older class \(\mathcal{H}_d(\beta, L)\) and that $\Delta\in[1/n,1]$. Then, there exists $\Lambda>0$ such that, if \ref{hyp1} holds with $\lambda\le \Lambda$, then, for all  $\varepsilon>0$ sufficiently small, there exist positive constants $c$ and $c_\varepsilon$ such that for all $x \in \R^d$,
\begin{equation}
\mathbb{E}[|\check{\pi}_{h,n}(x) - \pi (x)|^2] \le 
\begin{cases} {c}\left((\prod_{l = 1}^d h_l)^{-2}(n \Delta)^{-1} {+} \sum_{l = 1}^d h_l^{2\beta_l}\right) & \textnormal{if $H<1/2$},\\
\\
c_\varepsilon \left((\prod_{l = 1}^d h_l)^{-2}(n \Delta)^{2H-2+\varepsilon} {+} \sum_{l = 1}^d h_l^{2\beta_l}\right) &\textnormal{if $H>1/2$}.
\end{cases}
\end{equation}
Thus, the rate optimal choices $h_l=(n \Delta)^{-\frac{\bar{\beta}}{2 \beta_l(\bar{\beta}+d)}}$ for $H<1/2$ and $h_l=(n \Delta)^{-\frac{\bar{\beta}(1-H)}{\beta_l(\bar{\beta}+d)} -\varepsilon}$ for $H>1/2$, for any $l \in \{1, ... , d \}$, lead us to
\begin{equation*}
\mathbb{E}[|\check{\pi}_{h,n}(x) - \pi (x)|^2] \le 
\begin{cases} c(\frac{1}{n \Delta})^{\frac{\bar{\beta}}{\bar{\beta} + d}} & \textnormal{if $H<1/2$},\\
\\
c_\varepsilon (\frac{1}{n \Delta})^{\frac{2(1 - H)\bar{\beta}}{\bar{\beta} + d} - \varepsilon}  &\textnormal{if $H>1/2$}.
\end{cases}
\end{equation*}
\end{theorem}}

\TCR{
It is clear from the results above that the variance, and therefore the rate-optimal choice of bandwidth, depend on the (unknown) smoothness of the invariant density. This motivates the introduction of a data-driven bandwidth selection procedure. With this goal in mind, we define a quantity that heuristically represents the bias term, together with a penalty term whose order matches that of the variance. Following the Goldenshluger–Lepski approach, the selected bandwidth is the one that minimizes the sum of these two quantities.}

\TCR{
To formalize this reasoning, let us introduce the set of candidate bandwidths given by
$$
\mathcal{H}_n \subset \left\{ h = (h_1, \dots, h_d) \in (0, 1]^d : \prod_{l = 1}^d h_l \le \left( \frac{1}{\log (n \Delta)} \right)^{\frac{1}{4} + a}, \, a > 0 \right\}.
$$
We assume that the growth of $|\mathcal{H}_n|$ is at most polynomial in $T = n \Delta$, i.e., there exists a constant $c > 0$ such that
\begin{equation}{\label{eq: Hn pol}}
 |\mathcal{H}_n| \le c (n \Delta)^c.   
\end{equation}
Such an assumption is classical in the analysis of adaptive procedures; see, for instance, Section~3.2 in \cite{Minimax}.  
An example of a set satisfying this condition is
\begin{equation}\label{eq: bandwidths pol}
\mathcal{H}_n := \left\{ h = (h_1, \dots, h_d) \in (0, 1]^d : \forall i \in \{1, \dots, d\}, \, h_i = \frac{1}{z_i}, \, z_i \in \mathbb{N}:\prod_{l = 1}^d \frac{1}{z_l} \le \left( \frac{1}{\log (n \Delta)} \right)^{\frac{1}{4} + a} \right\}.
\end{equation}
Corresponding to the candidate bandwidths, we define the set of candidate estimators as
$$
\mathcal{F}(\mathcal{H}_n) := \left\{ \check{\pi}_{h,n}(x) = \frac{1}{n} \sum_{i = 0}^{n - 1} \mathbb{K}_h(X_{i \Delta} - x) : \, x \in \R^d, \, h \in \mathcal{H}_n \right\}.
$$
Our goal is to select an estimator from the family $\mathcal{F}(\mathcal{H}_n)$ in a fully data-driven way, based solely on the discrete observations of the process~$X$.  
To this end, for any $h, \eta \in \mathcal{H}_n$ and any $x \in \R^d$, we set
$$
\mathbb{K}_h \star \mathbb{K}_\eta (x) := \prod_{j = 1}^d (K_{h_j} \star K_{\eta_j})(x_j) = \prod_{j = 1}^d \int_{\R} K_{h_j}(u - x_j) K_{\eta_j}(u) \, du.
$$
We then introduce the associated kernel density estimator defined by
$$
\check{\pi}_{(h, \eta), n}(x) := \frac{1}{n} \sum_{i = 0}^{n - 1} (\mathbb{K}_h \star \mathbb{K}_\eta)(X_{i \Delta} - x).
$$
Since the convolution is commutative, we have $\check{\pi}_{(h, \eta), n} = \check{\pi}_{(\eta, h), n}$. Next, we define the penalty function $V_n(h)$ as an estimate of the order of the squared variance:
$$
V_n(h) := \frac{1}{n \Delta \prod_{l = 1}^d h_l^2} \min \big(1, (n \Delta)^{2H - 1 + \epsilon}\big).
$$
We then compare the differences between $\check{\pi}_{(h, \eta), n}$ and $\check{\pi}_{\eta, n}$, which heuristically correspond to an estimate of the squared bias:
\begin{equation}\label{eq:Anhx}
A_n(h, x) := \sup_{\eta \in \mathcal{H}_n} \big( |\check{\pi}_{(h, \eta), n}(x) - \check{\pi}_{\eta, n}(x)|^2 - V_n(\eta) \big)_+,
\end{equation}
where $(y)_+ := \max(0, y)$ denotes the positive part of $y$. The selected bandwidth is then defined as
$$
\tilde{h}(x) := \arg \min_{h \in \mathcal{H}_n} \big( A_n(h, x) + V_n(h) \big),
$$
and the final plugin estimator is the one in $\mathcal{F}(\mathcal{H}_n)$ associated with $\tilde{h}(x)$, namely $\check{\pi}_{\tilde{h}(x), n}(x).$}

\TCR{
Before presenting the results on $\check{\pi}_{\tilde{h}(x), n}(x)$, let us introduce some notation.  
In what follows, it will be useful to consider
$\pi_h := \mathbb{K}_h \star \pi,$
that is, the smoothed version of $\pi$, corresponding to the bias component estimated by $\check{\pi}_{h, n}$, since
$\E[\check{\pi}_{h, n}(x)] = \pi_h(x)$.  
Similarly, we define
$\pi_{(h, \eta)} := \mathbb{K}_h \star \mathbb{K}_\eta \star \pi.$
Moreover, we introduce the (pointwise) bias as
$$
B_n(h, x) := \big| \E[\check{\pi}_{h, n}(x)] - \pi(x) \big| = \big| \pi_h(x) - \pi(x) \big|,
$$
and its uniform version as
$$
B_n(h) := \sup_{x \in \R^d} B_n(h, x).
$$
With this notation in mind, we are now ready to state our first result on the adaptive procedure.
\begin{theorem} \label{th: rate adaptive}
Assume that \(\pi\) belongs to the H\"older class \(\mathcal{H}_d(\beta, L)\) and that $\Delta\in[1/n,1]$. Then, there exist $\Lambda, n_0 >0$ such that, if \ref{hyp1} holds with $\lambda\le \Lambda$, then, for all $n \ge n_0$ and for all $x \in \R^d$,
\begin{equation}{\label{eq: result adap}}
\mathbb{E}[|\check{\pi}_{\tilde{h}(x),n}(x) - \pi (x)|^2] \lesssim \inf_{h \in \mathcal{H}_n}(B_n(h) + V_n (h)) + (n \Delta)^{c_1} e^{-c_2 (\log (n \Delta))^{c_3}}
\end{equation}
for positive constants $c_1, c_2$ and for $c_3 > 1$.
\end{theorem}
The proofs of \cref{th: rate adaptive} and Corollary~\ref{cor: rate adaptive final} (stated below) can be found in \cref{s: proof adaptive}.  
We now conclude by establishing that the data-driven estimator attains the same convergence rates as in \cref{th: rate start discrete}.  
In other words, we aim to recover the rate corresponding to the optimal bandwidth choice, under the additional requirement that this bandwidth belongs to the set of candidate values.  
To this end, we fix the collection of candidate bandwidths \( \mathcal{H}_n \) as in \eqref{eq: bandwidths pol}.
\begin{corollary}{\label{cor: rate adaptive final}}
Assume that \(\pi\) belongs to the H\"older class \(\mathcal{H}_d(\beta, L)\) and that $\Delta\in[1/n,1]$. Then, there exist $\Lambda, n_0 >0$ such that, if \ref{hyp1} holds with $\lambda\le \Lambda$, then, for all $n \ge n_0$ and for all $x \in \R^d$,
\begin{equation}
\mathbb{E}[|\check{\pi}_{\tilde{h}(x),n}(x) - \pi (x)|^2] \lesssim
\begin{cases} (\frac{1}{n \Delta})^{\frac{\bar{\beta}}{\bar{\beta} + d}} + (n \Delta)^{c_1} e^{-c_2 (\log (n \Delta))^{c_3}}  & \textnormal{if $H<1/2$},\\
\\
c_\varepsilon (\frac{1}{n \Delta})^{\frac{2(1 - H)\bar{\beta}}{\bar{\beta} + d} - \varepsilon} + (n \Delta)^{c_1}e^{-c_2 (\log (n \Delta))^{c_3}}  &\textnormal{if $H>1/2$},
\end{cases}
\end{equation}
for positive constants $c_1, c_2$ and $c_3 > 1$.
\end{corollary}
\begin{remark}
It is clear that the exponential term in the convergence rates above is negligible compared to the first, polynomial term in \( n \Delta \). Consequently, the estimator constructed using the data-driven bandwidth selection over the set of candidate bandwidths \( \mathcal{H}_n \), as defined in \eqref{eq: bandwidths pol}, achieves the same convergence rates as those established in \cref{th: rate start discrete}.
\end{remark}}

\section{\TCR{Numerical examples}} \label{section:NS}

\TCR{The aim of this section is to illustrate our theoretical results through simulated examples. More precisely, we investigate whether the non-asymptotic bounds on the distance between the kernel estimator and the true density remain valid in practice. To this end, we consider a setting in which the invariant distribution is explicitly known, namely the Ornstein--Uhlenbeck process
\begin{equation} \label{m1}
dX_t = dB_t - X_t\,dt .
\end{equation}
The (marginal) invariant density $\pi$ is the Gaussian distribution $\mathcal N(0,\sigma_H^2 I_d)$, where
\[
\sigma_H^2 = \tfrac{1}{2}\,\Gamma(2H+1) = H\,\Gamma(2H),
\]
and $I_d$ denotes the $d \times d$ identity matrix. Moreover, the drift coefficient satisfies the global contractivity condition with $\kappa = 1$.}

\TCR{The process is discretized by the Euler scheme with time step $\Delta = 0.02$ and the driving fBm is built
via the classical Davies–Harte circulant embedding algorithm for fractional Gaussian noise
increments, see \cite{DaviesHarte1987}. 
To ensure approximate stationarity of the simulated trajectory and avoid dependence on the initial condition $X_0=0$,
we discard the first $10$ time units of simulation (warm-up period). This classical burn-in or thermalization step is standard in ergodic Monte Carlo simulations of Ornstein–Uhlenbeck
and Langevin-type dynamics.}

\TCR{We investigate two regimes: $H = 0.3$ (short memory) and $H = 0.7$ (long memory). In both cases, we first compare the kernel estimator with the true density using a Gaussian kernel (see \cref{fig:compdensity11}). In a second step (see \cref{fig:MSE22}), we compute a Monte Carlo approximation of the mapping
\[
n \longmapsto \mathbb{E}\big[ \lvert \check{\pi}_{h,n}(x) - \pi(x) \rvert^2 \big],
\]
where we recall that
\[
\check{\pi}_{h,n}(x)
= \frac{1}{n} \sum_{i=0}^{n-1} \mathbb{K}_h(x - X_{i\Delta}).
\]
We now provide further details. For \cref{fig:compdensity11}, the computations are performed in dimension $d = 3$, using a Gaussian kernel, with a total time horizon $T = 10^4$. We display the projection onto the first coordinate. For the sake of simplicity, we assume that the process is observed at each discretization time. Although computations could also be carried out to investigate the behavior under sparser observations, our primary objective here is to focus on the comparison with the theoretical rates.}

\begin{figure}[htbp]
    \centering
    \begin{subfigure}[b]{0.48\textwidth}
        \includegraphics[width=\textwidth]{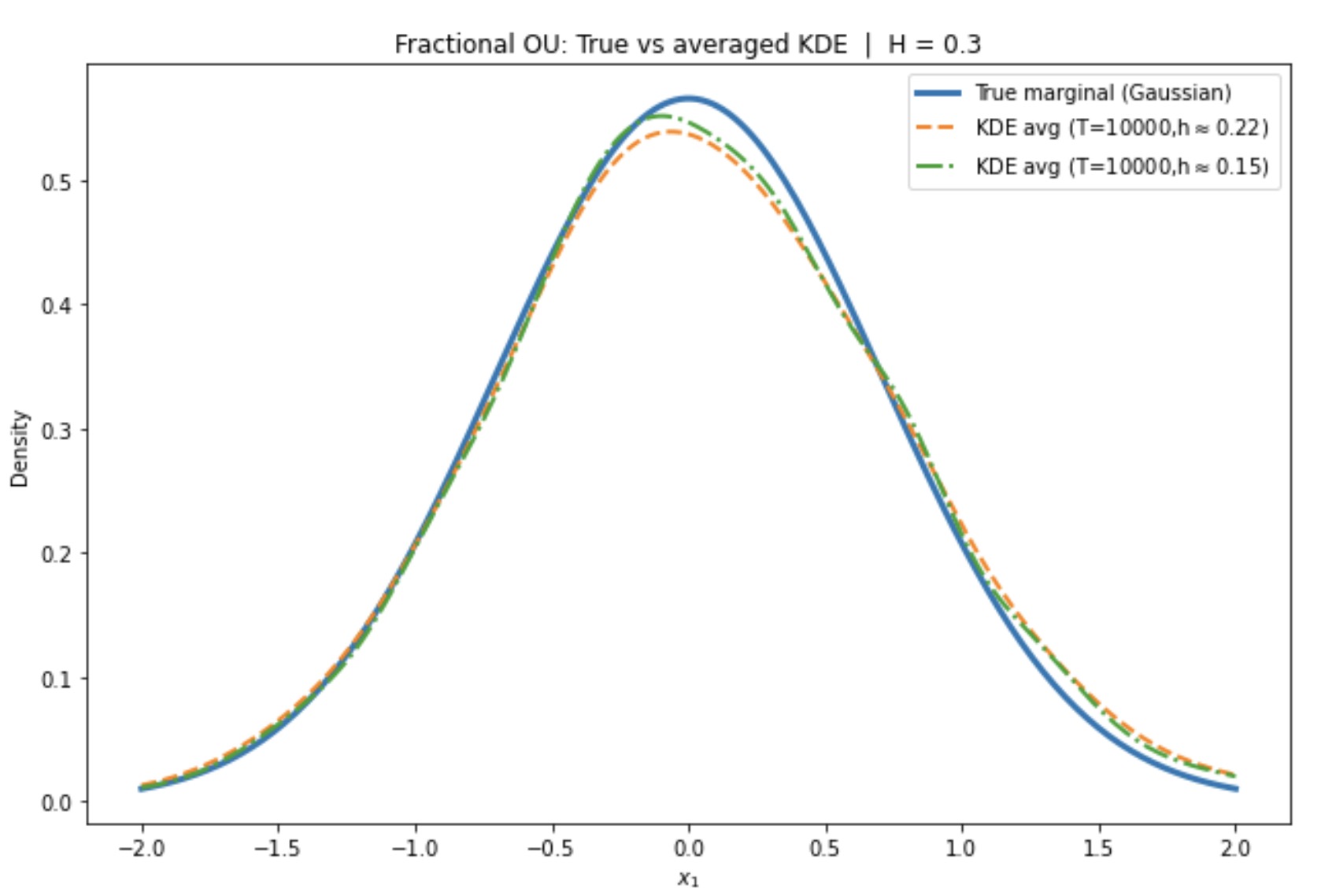}
        \caption{H=0.3, Gaussian kernel}
        %\label{fig:figure1}
    \end{subfigure}
    \hfill
    \begin{subfigure}[b]{0.48\textwidth}
        \includegraphics[width=\textwidth]{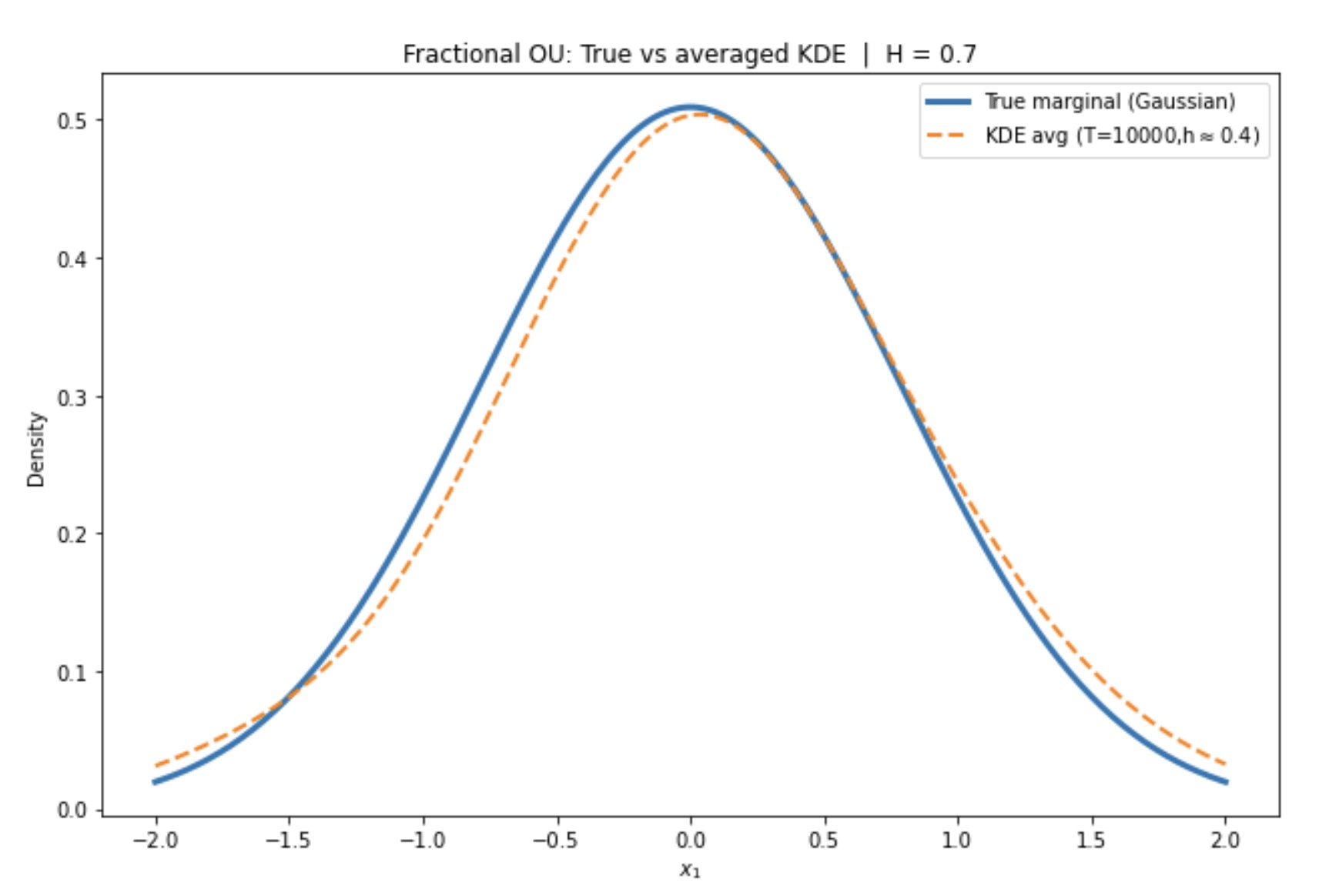}
        \caption{H=0.7, Gaussian kernel}
        %\label{fig:figure2}
    \end{subfigure}
    \caption{Comparison between true and estimated densities. Left: $H=0.3$ with $h$ of Theo 1 and Theo 2. Right $H=0.7$.}
    \label{fig:compdensity11}
\end{figure}
\TCR{For this reason, in \cref{fig:MSE22} we explicitly plot a Monte Carlo approximation of the mapping
\[
\bigl\{ \bigl( \log t,\; \log\bigl( \mathbb{E}\big[ \lvert \check{\pi}_{h,\lfloor t/\Delta \rfloor}(x) - \pi(x) \rvert^2 \big] \bigr) \bigr) : t \in [T_0, T] \bigr\}.
\]
Recall that, for a function of the form $f(t) = C t^{-\alpha}$, such a representation yields an affine function with slope $-\alpha$. In this part, we assume that $d = 1$ and consider two kernels: the Gaussian kernel and the Legendre kernel of order $M = 5$\footnote{From a theoretical perspective, it is preferable to choose $M$ large. However, this typically has practical drawbacks, since large values of $M$ require polynomials of higher degree and may lead to large coefficients, which can negatively affect numerical performance.}. We recall that the Legendre kernel is defined by
\[
K(u) = \sum_{m=0}^M \varphi_m(0)\,\varphi_m(u)\,\mathbf{1}_{\{|u| \le 1\}},
\]
where $(\varphi_m)_{m \ge 0}$ denotes the orthogonal basis of Legendre polynomials (see \cite{Tsy} for details). In each figure, the curves are plotted for $T_0 = 10^3$ and $T = 10^4$, for $x \in \{-1, -\tfrac{1}{2}, 0, \tfrac{1}{2}, 1\}$, and the mean is taken using $100$ Monte Carlo simulations. In \cref{fig:MSE22}, the bandwidth $h$ is chosen according to the recommendation of \cref{th: rate start discrete}. If the numerical results are consistent with the theoretical predictions, \cref{th: rate start discrete} implies that the observed slope should be smaller than
\[
\frac{M}{M+1} \quad \text{if } H < \tfrac{1}{2},
\qquad \text{and} \qquad
\frac{2(1-H)M}{M+1} - \varepsilon \quad \text{if } H > \tfrac{1}{2}.
\]
These theoretical slopes are represented by the blue line in the figure. We observe that the empirical rates closely match the theoretical ones when $H < \tfrac{1}{2}$ and are slightly better when $H > \tfrac{1}{2}$, for both the Gaussian and the Legendre kernels.}

\begin{figure}[htbp]
    \centering
    \begin{subfigure}[b]{0.48\textwidth}
        \includegraphics[width=\textwidth]{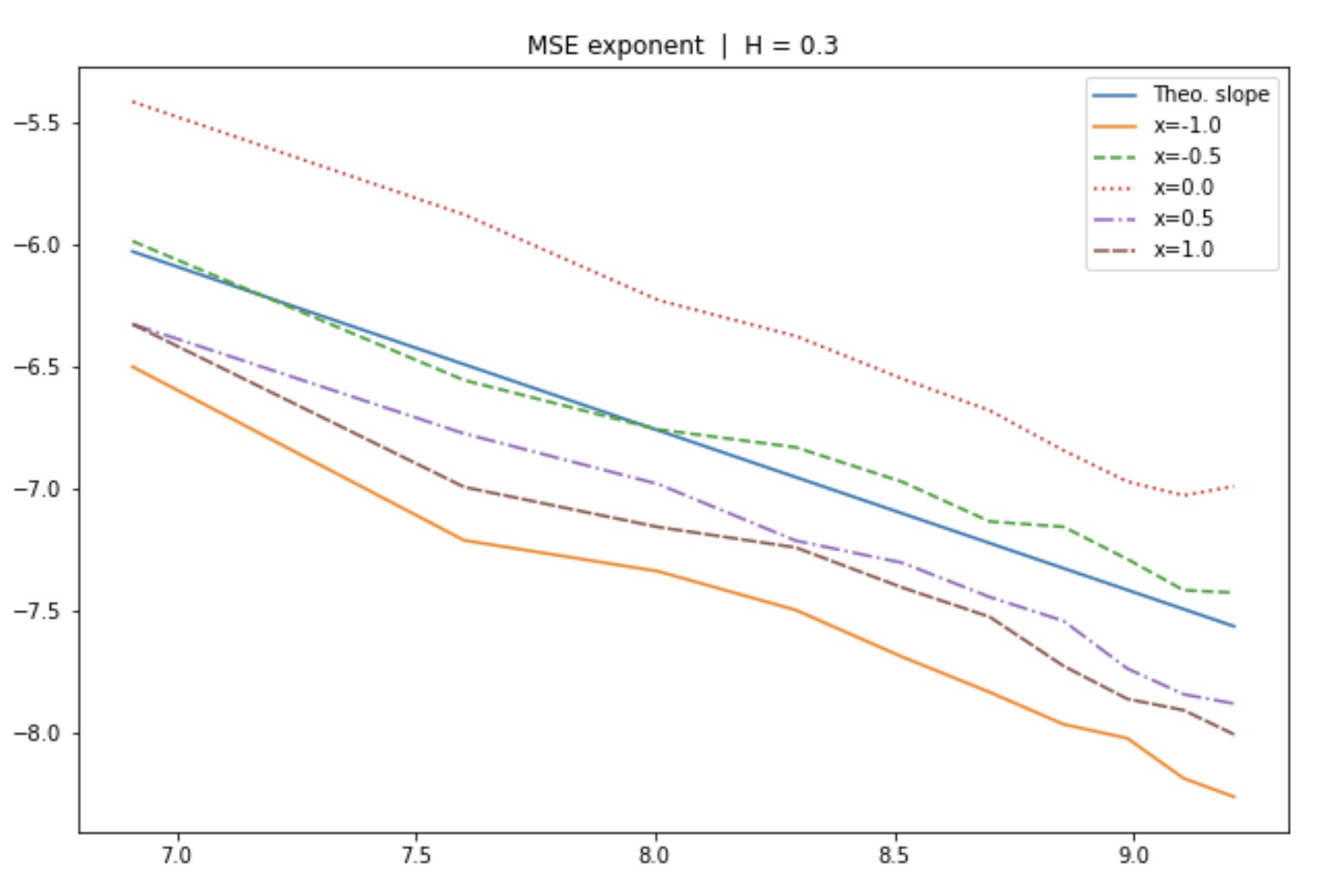}
        %\caption{Première figure}
        %\label{fig:figure3}
    \end{subfigure}
    \hfill
    \begin{subfigure}[b]{0.48\textwidth}
        \includegraphics[width=\textwidth]{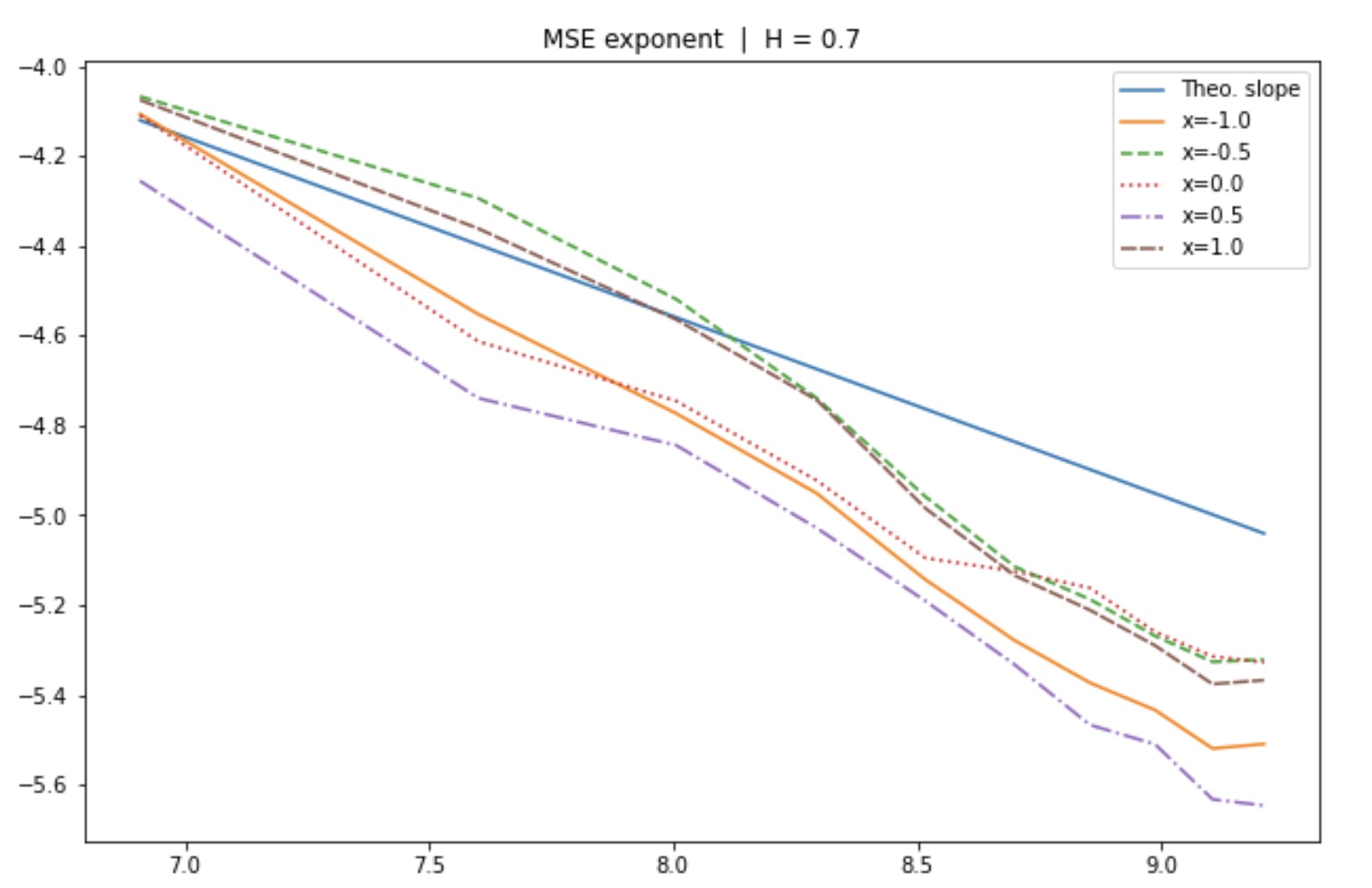}
        %\caption{Deuxième figure}
        %\label{fig:figure4}
    \end{subfigure}

    \vspace{0.5cm} % Espace vertical entre les lignes

    \begin{subfigure}[b]{0.48\textwidth}
        \includegraphics[width=\textwidth]{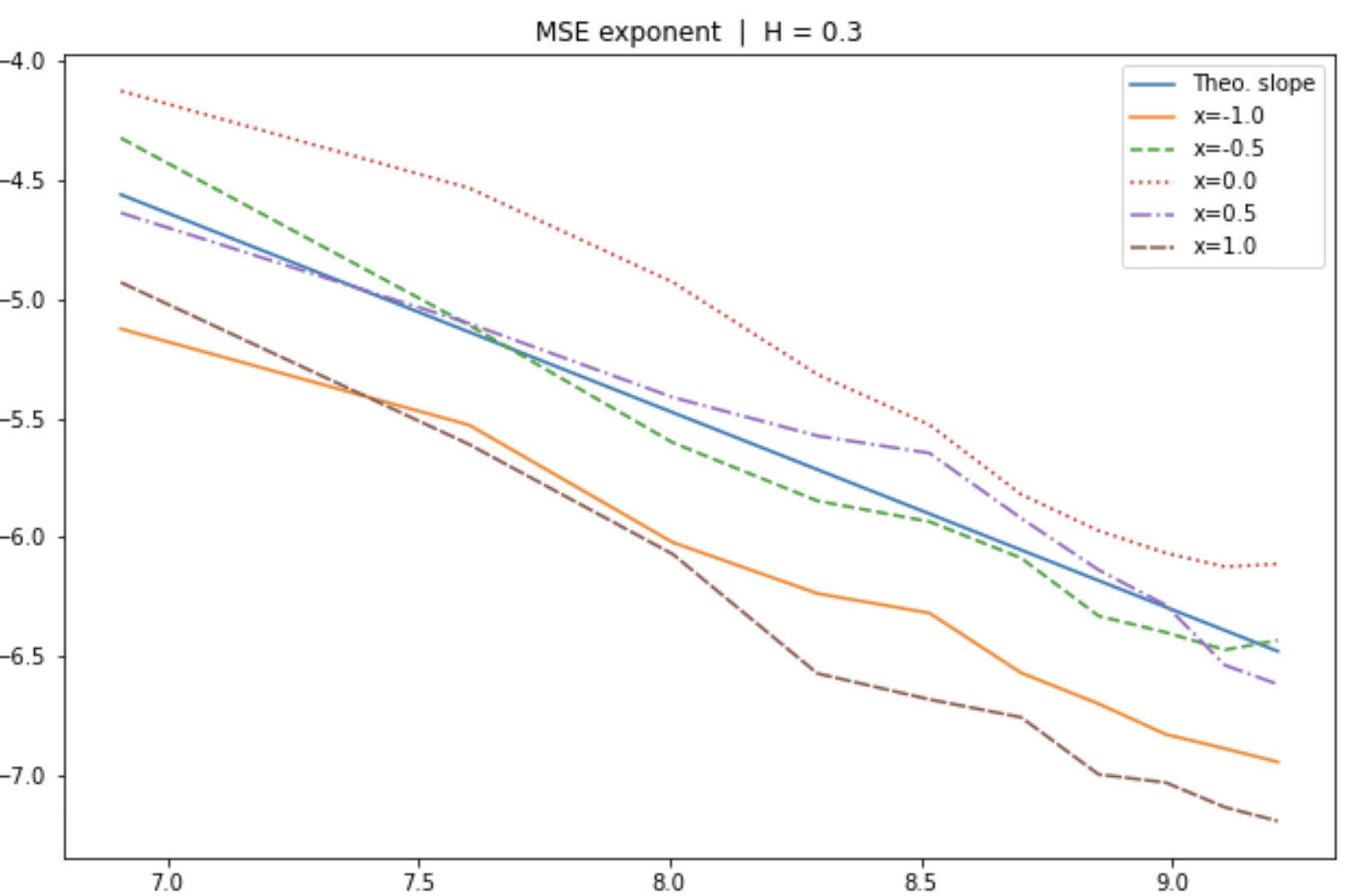}
        %\caption{Troisième figure}
        %\label{fig:figure5}
    \end{subfigure}
    \hfill
    \begin{subfigure}[b]{0.48\textwidth}
        \includegraphics[width=\textwidth]{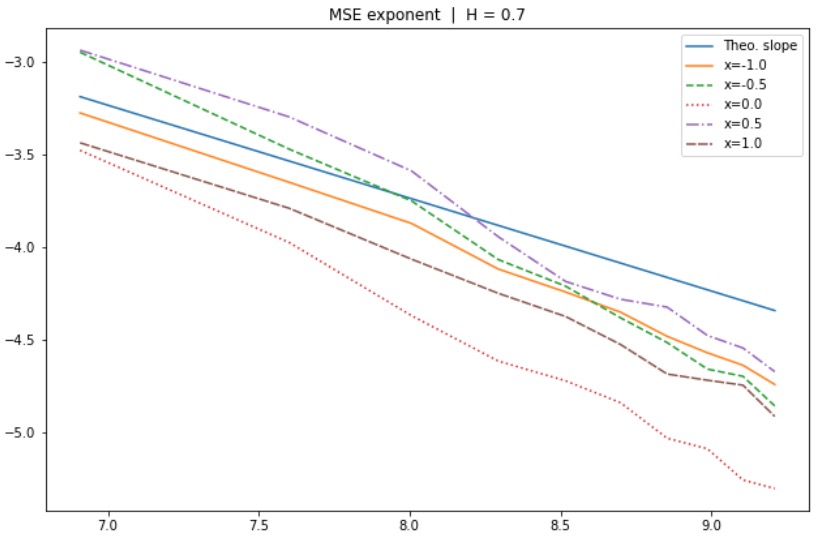}
        %\caption{Quatrième figure}
        %\label{fig:figure6}
    \end{subfigure}
    \caption{Evolution of the MSE in logarithmic scale. Blue line: theoretical exponent of Th. \ref{th: rate start discrete}, dotted lines: MSE. Top: Gaussian kernel ($M=2$). Bottom: Legendre kernel with $M=5$.}
    \label{fig:MSE22}
\end{figure}

\begin{remark}
\TCR{An example of a drift satisfying the weaker contractivity condition \textnormal{(\ref{eq:conditionb})} but not the global one can be obtained by replacing the linear drift 
in the three-dimensional fOU equation \eqref{m1} with a state-dependent radial drift $-\phi(\|X_t\|)\, X_t$, where $\phi$ is a smooth function of the distance from the origin. The drift is weaker near the origin and increases with distance, ensuring contractivity outside a compact set while failing globally. The invariant density is not explicit, and its smoothness depends on $\phi$: if $\phi$ is only once differentiable, the density inherits limited regularity and the bias of kernel estimators matters; if $\phi$ is twice differentiable or smoother, the density is sufficiently regular, the bias is negligible, and the estimator operates in a variance-dominated regime as in the fOU case. Thus, the choice of smoothness in the bandwidth schedule reflects the regularity of $\phi$, and while our theoretical results apply directly to this model, performing numerical simulations would require approximating the invariant density, which is not explicit and would be computationally challenging.}
\end{remark}

\section{Preliminary bounds}{\label{s: preliminary bounds}}

For $\ell\in\Cloc^{H-}(\R_+,\R^d)$, we will denote by $\Phi_t(\ell)$ a  solution to
\begin{equation}\label{eq:innovation_sde}
  \Phi_t(\ell)=\ell(t)+\int_0^t b\big(\Phi_s(\ell)\big)\,ds+\sigma\tilde{B}_t,\qquad t\geq 0,
\end{equation}
where $(\tilde{B}_t)_{t\ge0}$ denotes a Liouville process:
\begin{equation}\label{def:liouville}
\tilde{B}_t=\int_0^t(t-u)^{H-\frac{1}{2}} dW_s.
\end{equation}
By \cite[Lemma 3.6]{LPS22}, such an equation has a weak solution as soon as $b$ is Borel measurable with at most linear growth when $H<1/2$ or $b$ is $\alpha$-H\"older with $\alpha>1-(2H)^{-1}$, when $H>1/2$.

 Let us also introduce the following notations related to the Mandelbrot-Van Ness representation (\ref{Man}). We have:
\begin{align*}
  B_{t+h}-B_t&=c_1(H)\int_{-\infty}^t\Big((t+h-u)^{H-\frac12}-(t-u)^{H-\frac12}\Big)\,dW_u \nonumber \\
  &+c_1(H) \int_t^{t+h} (t+h-u)^{H-\frac12}\,dW_u=: \bar{B}_h^t + \tilde{B}_h^t. \label{eq:innovation_process}
\end{align*} 
For any $t\ge0$, $\bar{B}^t$ belongs to $\Cloc^{H-}(\ER_+,\ER^d)$. Furthermore, for any $t\ge0$,
{$\tilde{B}^t\overset{d}{=}\tilde{B}^0=:\tilde{B}$}. 
Note also that the solution $(X_t)_{t\ge0}$ to the fractional SDE satisfies for any $0\le s\le t$,
\begin{equation}\label{eq:semimarkov}
    X_{t+s}=\Phi_{t}(X_s+\sigma \bar{B}^s).
\end{equation}

\subsection{$L^2$-bounds between paths with different pasts}
We start by proving some $L^2$-estimates for paths with different pasts. Then, we will use them to deduce some bounds in total variation. 
\begin{lemma}[{\cite[Lemma 3.12]{LS22a}}]\label{lem:li_sieber}
    Under \ref{hyp1}, for each $R>0$ and each $\eta\in(0,\frac12)$, there is a constant $a_{\eta, R}>0$ such that the following holds. For each $\ell\in\mathcal{C}(\R_+,\R^d)$ there is an event $A_\ell$ with $\mathbb{P}(A_\ell)\geq a_{\eta,R}$ such that
    \begin{equation*}
        \int_0^1\mathbf{1}_{\{t:\,|\Phi_t(\ell)(\omega)|>R\}}(s)\,ds>\eta\qquad\forall\,\omega\in A_\ell.
    \end{equation*}
\end{lemma}

\begin{proposition}\label{lem:stability_innovation}
There exists $\Lambda>0$ such that if  \ref{hyp1} holds with $\lambda\le \Lambda$,  then, some constants $c,C>0$ and $\rho\in(0,1)$ exist such that, for each $\ell_1,\ell_2\in\mathcal{C}^1(\R_+,\R^d)$, it holds that for all $\tau\ge 1$,
    \begin{equation*}
        \ES[|\Phi_{\tau}(\ell_1) - \Phi_{\tau}(\ell_2)|^2]\leq C\left(e^{-c\tau}\ES[|\Phi_{1}(\ell_1) - \Phi_{1}(\ell_2)|^2] +\sum_{k=1}^{\lfloor\tau\rfloor}\rho^{\tau-k}\|{\dot{\ell}_1-\dot{\ell}_2}\|_{\infty;[k,(k+1)]}^2\right).
    \end{equation*}
\end{proposition}

\begin{proof}
    First notice that we have
    \begin{align*}
        \ES[|\Phi_{\tau}(\ell_1) &- \Phi_{\tau}(\ell_2)|^2]=\ES[\ES[|\Phi_{\tau}(\ell_1) - \Phi_{\tau}(\ell_2)|^2|\mathcal{F}_{\tau -1}]]\\
        &=\ES[|\Phi_1(x+\theta_{\tau-1}(\ell_1)+\varsigma) - \Phi_1(y+\theta_{\tau-1}(\ell_2)+\varsigma)|^2]|_{x=\Phi_{\tau-1}(\ell_1),y=\Phi_{\tau-1}(\ell_2),\varsigma=\bar{B}^{t}_{\cdot+t}},
    \end{align*}
    where $\theta_{\tau-1}(\ell_1)(t)=\ell_1(\tau - 1 + t) - \ell_1(\tau-1)$ is the Wiener shift. {By \ref{hyp1}} and the elementary inequality $\langle a,b\rangle\le 2^{-1}(\frac{|a|^2}{\epsilon}+\epsilon|b|^2)$, for each $\epsilon>0$,
    \begin{align*}
        \frac{d}{dt}|\Phi_{t}(\ell_1)-\Phi_{t}(\ell_2)|^2&=2\langle\dot{\ell}_1(t)-\dot{\ell}_2(t)+b(\Phi_t(\ell_1))-b(\Phi_t(\ell_2)),\Phi_{t}(\ell_1)-\Phi_{t}(\ell_2)\rangle \\
        &\leq\frac{|\dot{\ell}_1(t)-\dot{\ell}_2(t)|^2}{\epsilon}+(2\lambda + \epsilon)|\Phi_{t}(\ell_1)-\Phi_{t}(\ell_2)|^2.
    \end{align*}
    For $0\leq t_1\leq t_2\leq 1$, this implies the pathwise estimate
    \begin{align}
        |\Phi_{t_2}(\ell_1)-\Phi_{t_2}(\ell_2)|^2&\leq|\Phi_{t_1}(\ell_1)-\Phi_{t_1}(\ell_2)|^2 e^{(2\lambda+\epsilon)(t_2-t_1)}\nonumber\\
        &+\epsilon^{-1}\int_{t_1}^{t_2}e^{(2\lambda+\epsilon)(t_2 - s)}|\dot{\ell}_1(s)-\dot{\ell}_2(s)|^2\,ds \nonumber \\
        &\leq |\Phi_{t_1}(\ell_1)-\Phi_{t_1}(\ell_2)|^2 e^{(2\lambda+\epsilon)(t_2-t_1)}+C\epsilon^{-1}\|\dot{\ell}_1-\dot{\ell}_2\|_\infty^2 (t_2-t_1).\label{eq:taufloortau}
    \end{align}
    If the interval is such that $\Phi(\ell_1)\restriction_{[t_1,t_2]}$ lies within the contractive region of $b$ (typically for an $R$ larger than the contractive radius of the drift and a smaller $\kappa$), then we obtain the following estimate
    \begin{align*}
        |\Phi_{t_2}(\ell_1)-\Phi_{t_2}(\ell_2)|^2&\leq|\Phi_{t_1}(\ell_1)-\Phi_{t_1}(\ell_2)|^2 e^{(\epsilon-\kappa)(t_2-t_1)}+c \epsilon^{-1}\|\dot{\ell}_1-\dot{\ell}_2\|_\infty^2 (t_2-t_1).
    \end{align*}
    Combining these two estimates with \cref{lem:li_sieber} as done in \cite{LS22a}, this gives that, for an $\epsilon>0$ sufficiently small, 
    \begin{equation*}
        \ES[|\Phi_{\tau}(\ell_1) - \Phi_{\tau}(\ell_2)|^2]\leq\rho\ES[|\Phi_{\tau-1}(\ell_1) - \Phi_{\tau-1}(\ell_2)|^2]+C\|\dot{\ell}_1-\dot{\ell}_2\|_{\infty;[\tau-1,\tau]}^2.
    \end{equation*}
 {Iterating this bound achieves the proof when $\tau$ is an integer. To extend it to any $\tau\ge1$, one can use the inequality \eqref{eq:taufloortau} with $t_1=\lfloor \tau \rfloor $ and $t_2=\tau$, where the value of the constant $C$ may potentially change.} 
\end{proof}
In the next corollary, we want to deduce some $L^2$-bounds when the process $\ell$ corresponds to the memory of the Brownian motion before time $0$.  To this end, we recall that the linear operator
$$\omega \mapsto \int_{-\infty}^0 ( (t-s)^{H-\frac{1}{2}}-(-s)^{H-\frac{1}{2}} )dw_s$$
is (clearly) well-defined on ${\cal C}_0^\infty(\ER^{-})$ and continuously extends to $\mathcal{W}$.

%
%By an integration by parts, one can check that for every continuous path on $(-\infty,0]$ such that
%$$\frC(w,\epsilon):=\sup_{s\le 0}  \frac{|w_s|}{(1\vee|s|)^{\frac{1+\varepsilon}{2}}}<+\infty\quad  \forall\,\varepsilon>0,$$
%we have:
%$$\int_{-\infty}^0 (t-s)^{H-\frac{1}{2}}-(-s)^{H-\frac{1}{2}} dw_s=\left(H-\frac{1}{2}\right)\int_{-\infty}^0 [(t-s)^{H-\frac{3}{2}}-(-s)^{H-\frac{3}{2}}\ind{s<-1}] w_s ds-\int_{-1}^0(-s)^{H-\frac{1}{2}} dw_s.$$
%\textcolor{blue}{There is a problem here since the above integral does not converge near $0$. But, there is a point which is not clear to me.
%If I write
%$$\ell_t=\int_{-\infty}^0 (t-s)^{H-\frac{1}{2}}-(-s)^{H-\frac{1}{2}} dw_s$$
%and I formally derivate, I get
%$$\dot{\ell}_t=\pm(3/2-H)\int_{-\infty}^0 (t-s)^{H-\frac{3}{2}} dw_s$$
%which in turn by an integration by parts leads to
%$$\dot{\ell}_t=\pm(3/2-H)$$
%}
%\fab{F: Here, we should perhaps provide the integration by parts before the corollary ??}
\begin{corollary}\label{cor:L2bound} Let ${\cal C}_\varepsilon$ denote the space of continuous functions $w$ on $(-\infty,0]$ with $w_0=0$ and 
\begin{equation}\label{eq:mathfrakCesp}
\frC(w,\epsilon):=\sup_{s\le 0}  \frac{|w_s|}{(1\vee|s|)^{\frac{1+\varepsilon}{2}}}<+\infty\quad  \forall\,\varepsilon>0.
\end{equation}
 For $t>0$, set
%$$\ell_t(x,w)=x+\left(H-\frac{1}{2}\right)\int_{-\infty}^0 ((t-s)^{H-\frac{3}{2}}-(-s)^{H-\frac{3}{2}}) w_s ds.$$
$$\ell_t(x,w)=x+c_1(H)\int_{-\infty}^0 \left((t-s)^{H-\frac{1}{2}}-(-s)^{H-\frac{1}{2}} \right) dw_s.$$
 Then, there exists $c>0$ such that for any $\varepsilon\in(0,1)$, the following bound holds for every $w$, $\tilde{w}\in{\cal C}_\varepsilon$, $x,y\in\ER^d$ and $\tau\ge 1$:
    \begin{equation*}
        \ES[|\Phi_{\tau}(\ell(x,w)\TCR{)} - \Phi_{\tau}(\ell(y,\tilde{w}))|^2]\lesssim_{\varepsilon}  e^{-c\tau}|x-y|^2+ \TCR{(\frC(w-\tilde{w},\epsilon))^2}\tau^{2H-2+\varepsilon}.
        \end{equation*}
        Furthermore, if $w$ and $\tilde{w}$ are such that $w_t=\tilde{w}_t$ on $(-\infty,-1]$, then
         \begin{equation*}
        \ES[|\Phi_{\tau}(\ell(x,w)\TCR{)} - \Phi_{\tau}(\ell(y,\tilde{w}))|^2]\lesssim e^{-c\tau}|x-y|^2+ \TCR{ \|w-\tilde{w}\|^2_{\infty,[-1,0]} }\tau^{2H-5}.
       %  \ES[|\Phi_{\tau}(\ell(x,w)\TCR{)} - \Phi_{\tau}(\ell(y,\tilde{w}))|^2]\lesssim \|w-\tilde{w}\|^2_{\infty,[-1,0]} \left(e^{-c\tau}|x-y|^2+ \tau^{2H-5}\right).
  \end{equation*}
%\sum_{k=0}^{[\tau]}\rho^k\|\tcr{\dot{\ell}_1-\dot{\ell}_2}\|_{\infty;[k,k+1]}^2\right)\qquad\forall\,\tau\geq 2.
    %\end{equation*}
 \end{corollary}
 \begin{proof} For all $t>0$,
 %$$\dot{\ell}_t(x,w)=\left(H-\frac{1}{2}\right)\left(\frac{3}{2}-H\right)\int_{-\infty}^0 (t-s)^{H-\frac{5}{2}} w_s ds.$$
 $$\dot{\ell}_t(x,w)=c_1(H)\left(H-\frac{1}{2}\right)\left(\frac{3}{2}-H\right)\int_{-\infty}^0 (t-s)^{H-\frac{5}{2}} w_s ds.$$
 For every $t\ge1$ and $\varepsilon>0$,
 \begin{equation}{\label{eq: 1}}
|\dot{\ell}_t(x,w)-\dot{\ell}_t(y,\tilde{w})|\lesssim c_\varepsilon t^{{H-1+\frac{\varepsilon}{2}}}\quad\textnormal{with} \quad c_\varepsilon=\frC(w-\tilde{w},\epsilon).
 \end{equation}
By \cref{lem:stability_innovation}, we get for all $\tau\ge 2$,
\begin{equation*}
        \ES[|\Phi_{\tau}(\ell(x,w) - \Phi_{\tau}(\ell(y,\tilde{w}))|^2]\lesssim e^{-c\tau}\ES[|\Phi_{1}(\TCR{\ell(x,w)}) - \Phi_{1}(\TCR{\ell(y,\tilde{w})})|^2] +c_\varepsilon^2 \sum_{k=1}^{\lfloor\tau\rfloor}\rho^{\tau-k} k^{2H-2+\varepsilon}.
 \end{equation*}
 On the one hand, elementary computations show that
 $$\sum_{k=1}^{[\tau]}\rho^{\tau-k} k^{2H-2+\varepsilon}\lesssim \tau^{2H-2+\varepsilon}.$$
On the other hand, one checks  that $\ell_0(x,w)-\ell_0(y,\tilde{w})=x-y$ and 
 $$  \|\dot{\ell}(x,w)-\dot{\ell}(y,\tilde{w})\|_{\infty,[0,1]}\le c_\varepsilon.$$
 Thus, by \TCR{the same argument as in \eqref{eq:taufloortau}},
$$ \ES[|\Phi_{1}(\TCR{\ell(x,w)}) - \Phi_{1}(\TCR{\ell(y,\tilde{w})})|^2]\lesssim  \TCR{|x-y|^2+c_\varepsilon^2}.$$
The first statement follows.\\

\noindent  For the second statement, one remarks that when $w_t=\tilde{w}_t$ on $(-\infty,-1]$, then for every $t\ge1$
\begin{equation}{\label{eq: 2}}
|\dot{\ell}_t(x,w)-\dot{\ell}_t(y,\tilde{w})|\lesssim \|w-\tilde{w}\|_{\infty,[-1,0]}\int_{-1}^0 (t-s)^{H-\frac{5}{2}} ds\lesssim \|w-\tilde{w}\|_{\infty,[-1,0]}  t^{H-\frac{5}{2}}.
\end{equation}
 Then, the same argument as before leads to the result.
\end{proof} 
 \subsection{$TV$-bounds between paths with different pasts}\label{sec:TVbounds}
Now, we want to deduce some total variation bounds from the previous $L^2$-estimates. The idea is to try to stick the paths at time $\tau$ if \ar{the paths} are close at time $\tau-1$; the closeness being measured with the help of the $L^2$-estimates.
\begin{proposition}\label{prop:TVTV} 
 With the notation of \cref{cor:L2bound}, there exists $c>0$ such that for any $\varepsilon\in(0,1)$, the following bound holds for every $w$, $\tilde{w}\in{\cal C}_\varepsilon$ and $x,y\in\ER^d$:
    \begin{equation*}
        \|{\cal L}(\Phi_{\tau}(\ell(x,w)) - {\cal L}(\Phi_{\tau}(\ell(y,\tilde{w})))\|_{{\rm TV}}\lesssim_{\varepsilon}    e^{-c\tau}|x-y|+\TCR{\frC(w-\tilde{w},\epsilon)} \TCR{(1\vee\tau)}^{H-1+\frac{\varepsilon}{2}}.
        \end{equation*}
        Furthermore, if $w$ and $\tilde{w}$ are such that $w_t=\tilde{w}_t$ on $(-\infty,-1]$, then
         \begin{equation*}
        \|{\cal L}(\Phi_{\tau}(\ell(x,w)) - {\cal L}(\Phi_{\tau}(\ell(y,\tilde{w})))\|_{{\rm TV}}\lesssim e^{-c\tau}|x-y|+ \TCR{ \|w-\tilde{w}\|_{\infty,[-1,0]}}\TCR{(1\vee\tau)}^{H-\frac{5}{2}}.
        \end{equation*}
        %%% version avant revision
% \begin{proposition}\label{prop:TVTV} 
% With the notation of \cref{cor:L2bound}, there exists $c>0$ such that for any $\varepsilon\in(0,1)$, the following bound holds for every $w$, $\tilde{w}\in{\cal C}_\varepsilon$ and $x,y\in\ER^d$:
%    \begin{equation*}
%        \|{\cal L}(\Phi_{\tau}(\ell(x,w)) - {\cal L}(\Phi_{\tau}(\ell(y,\tilde{w})))\|_{{\rm TV}}\lesssim_{\varepsilon}   \frC(w-\tilde{w},\epsilon) \left(e^{-c\tau}|x-y|+\tau^{H-1+\frac{\varepsilon}{2}}\right).
%        \end{equation*}
%        Furthermore, if $w$ and $\tilde{w}$ are such that $w_t=\tilde{w}_t$ on $(-\infty,-1]$, then
%         \begin{equation*}
%        \|{\cal L}(\Phi_{\tau}(\ell(x,w)) - {\cal L}(\Phi_{\tau}(\ell(y,\tilde{w})))\|_{{\rm TV}}\lesssim \|w-\tilde{w}\|_{\infty,[-1,0]} \left(e^{-c\tau}|x-y|+ \tau^{H-\frac{5}{2}}\right).
%        \end{equation*}

%\sum_{k=0}^{[\tau]}\rho^k\|\tcr{\dot{\ell}_1-\dot{\ell}_2}\|_{\infty;[k,k+1]}^2\right)\qquad\forall\,\tau\geq 2.
    %\end{equation*}
 \end{proposition}
\fab{The proof of \cref{prop:TVTV} is postponed to \cref{s: tech}.}

\section{The martingale decomposition}\label{s: mtg decomposition}
The bounds established in \cref{sec:TVbounds} will now be used to control the variance of $\frac{1}{T} \int_0^T F(X_s) ds$ for a given functional $F$. Here, as in \cite{adaptive}, we use a martingale decomposition of this random variable to eliminate covariance terms. \fab{More precisely, setting  $F_T=\int_0^T F(X_s)ds$, we write
%\tcr{This is not correct when $D\neq 0$ !}
\begin{equation}\label{eq:martdecomp429}
 F_T-\ES[F_T]=F_T-\ES[F_T|{\cal F}_{\lfloor T\rfloor}]+\sum_{k=1}^{\lfloor T\rfloor} \ES[F_T|{\cal F}_k]-\ES[F_T|{\cal F}_{k-1}]+\ES[F_T|{\cal F}_{0}]-\ES[F_T].
 \end{equation}
Using that all these terms are (by construction) centered and orthogonal and studying their variance separately, we get the following proposition (whose proof is postponed to \cref{sec:martproof}):}
\begin{proposition}\label{prop: mixing}
    Assume \ref{hyp1} and let $F\in L^\infty(\R^d)$. Then,
    %\begin{itemize}
  \begin{equation*}
       {\rm Var}\lp\frac{1}{T} \int_0^T F(X_s) ds \rp  \le \|F\|_\infty^2   \begin{cases} c T^{-1}  & \textnormal{if $H<1/2$}\\
        c_\varepsilon T^{2H-2+\varepsilon}&\textnormal{if $H>1/2$},
       \end{cases}
    \end{equation*}
    where $\varepsilon\in(0,1)$ and $c$, $c_\varepsilon$ denote some constants independent of $T$ and $F$.
\end{proposition}
\begin{remark}\label{rem:optimality}
\TCR{When $(X_t)_{t\ge0}$ is a fractional Ornstein-Uhlenbeck process and $F(x)=x$, one can check that the rates in $T$ can be retrieved with another approach using directly the covariance. Actually, in this case (see  \cite[Theorem 2.3]{cheridito}),
\begin{equation}\label{eq:covxsxtrem}
 {\rm Cov}(X_t,X_{t+s})=H(2H-1)  s^{2H-2}+O(s^{2H-4}), \quad s\ge1,
 \end{equation}
%show that the rates in $T$ 
%it is well-known (see \textit{e.g.} \cite{cheridito}) that 
%\begin{equation}\label{eq:covxsxtrem}
% {\rm Cov}(X_s,X_t)\approx 1\wedge (t-s)^{2H-2}.
% \end{equation}
%Then, classical arguments based on  Wick's probability Theorem and Hermite expansion allow to show that for a general function $F$, ${\rm Cov}(F(X_s),F(X_t))$ can be written as a series in powers of the covariance of $X_s$ and $X_t$. Hence, this suggests that the estimate \eqref{eq:covxsxtrem} holds for a general function $F$\footnote{A precise result would require technicalities that we do not want to add here. This remark must thus be read at the ``heuristic level''.}
%$${\rm Cov}(F(X_s),F(X_t))\approx 1\wedge (t-s)^{2H-2}.$$
which allows to compute the variance with the decomposition:
$${\rm Var}\lp\frac{1}{T} \int_0^T F(X_s) ds \right) =\frac{1}{T^2}\int_0^T \int_0^T {\rm Cov}(F(X_s),F(X_u)) du ds.$$
This leads to
$${\rm Var}\lp\frac{1}{T} \int_0^T X_s ds\right)= O(T^{-(1\vee (2-2H)}).$$
Note that this ``double-integral'' decomposition is not usable in the general case but that the martingale decomposition combined with the ``conditional contraction bounds'' of \cref{prop:TVTV} allows to retrieve the same rates (for any bounded measurable function) up to an $\varepsilon$ when $H>1/2$. 
}

%\noindent \textcolor{blue}{Finally, remark even if we use a ``double-integral'' decomposition of the covariance to discuss about lower-bounds, we can absolutely not use such arguments in the general case. To be more precise
% the ``conditional bound'' of \cref{prop:TVTV}, which is a fundamental ingredient of the proof of \cref{prop: mixing}, allows to retrieve this almost optimal bound through the martingale decomposition but does not provide sufficiently precise estimates of the covariance.

\end{remark}

{\cref{prop: mixing} holds for any bounded measurable functional $F$. Now, we want to use some specific properties of the kernel. In particular,} a classical method to improve the bounds and especially the dependence in \TCR{$h_1, ... , h_d$} is to take advantage of the fact that 
\begin{equation*}{\label{eq: 20,25}}
\int \mathbb{K}_h(x-y)\lambda_d(dy)=1.
\end{equation*}
To be exploited, this property must be combined with some bounds on the density of the semi-group. \tcr{In the specific setting of \eqref{eq:innovation_sde} (\textit{i.e.} of conditional fractional SDEs)}, some results have been recently obtained in \cite{LPS22}.  

First, by \cite[Lemma 3.6]{LPS22}, \eqref{eq:innovation_sde} has a weak solution as soon as $b$ is Borel measurable with at most linear growth when $H<1/2$ or $b$ is $\alpha$-H\"older with $\alpha>1-(2H)^{-1}$ when $H>1/2$. When $\sigma$ is invertible, this equation admits a density with respect to the Lebesgue measure on $\ER^d$, \textit{i.e.}, for any $t>0$, for any Borel set $A$ of $\ER^d$,
$$\PE(\Phi_t(\ell)\in A)=\int_A p_t(\ell;y)\lambda_d(dy).$$
\tcr{Furthermore, by \cite[Proposition 4.7]{LPS22},  $p_t(\ell;y)$ admits Gaussian upper-bounds (see \eqref{eq:upper_bound_conditional_density}) that we are going to exploit to obtain the next lemma.}

%The space of locally H\"older continuous functions $f:\R_+\to\R^d$ of order $\gamma\in(0,1]$ is denoted by $\Cloc^\gamma(\R_+,\R^d)$ and we set $\Cloc^{\gamma-}(\R_+,\R^d)\define\bigcap_{\gamma^\prime\in(0,\gamma)}\Cloc^{\gamma^\prime}(\R_+,\R^d)$.
%For $\ell\in\Cloc^{H-}(\ER_+,\ER^d)$, we denote by $\Phi_t(\ell)$ a  solution to
%\begin{equation}\label{eq:innovation_sde}
%  \Phi_t(\ell)=\ell(t)+\int_0^t b\big(\Phi_s(\ell)\big)\,ds+\sigma\tilde{B}_t,\qquad t\geq 0,
%\end{equation}
%where $(\tilde{B}_t)_{t\ge0}$ denotes a Liouville process:
%$$\tilde{B}_t=\int_0^t(t-u)^{H-\frac{1}{2}} dW_s.$$

%Let us also introduce the following notations related to the Mandelbrot-Van Ness representation of the fractional Brownian motion. We have:
%\begin{align}
%  B_{t+h}-B_t&=\alpha_H\int_{-\infty}^t\Big((t+h-u)^{H-\frac12}-(t-u)^{H-\frac12}\Big)\,dW_u\nonumber\\
%  &+\alpha_H\int_t^{t+h} (t+h-u)^{H-\frac12}\,dW_u \nonumber\\
%  &\define \bar{B}_h^t + \tilde{B}_h^t. \label{eq:innovation_process}
%\end{align} 
%For any $t\ge0$, $\bar{B}^t$ belongs to $\Cloc^{H-}(\ER_+,\ER^d)$. Furthermore, for any $t\ge0$,
%$\tilde{B}^t\overset{d}{=}\tilde{B}$. Hence, the solution $(X_t)_{t\ge0}$ to the fractional SDE satisfies for any $0\le s\le t$,
\begin{lemma}\label{lem:espcondkh}  For any $(s,t)\in\ER_+^2$ such that $0\le s<t$,
$$\TCR{|}\ES[\mathbb{K}_h(x-X_t)|{\cal F}_s]\TCR{|}\le \frac{\Xi_{s,t}}{(t-s)^{dH}\vee 1},$$
where $\Xi_{s,t}$ is a ${\cal F}_s$-random variable such that for any $p\ge1$,
$$\sup_{s,t\ge0, 0\le s\le t}\ES[|\Xi_{s,t}|^p]<+\infty.$$
%whose distribution does not depend on $s$ and has finite moments of any order.
\end{lemma}
\begin{proof}
By \eqref{eq:semimarkov}, for any $0\le s<t$, 
\begin{equation*}    
X_{t}=\Phi_{t-s}(X_s+\sigma \bar{B}^s).
\end{equation*}
Thus,
\begin{equation}{\label{eq: 20.5}}
 \ES[\mathbb{K}_h(x-X_t)|{\cal F}_s]=\int \mathbb{K}_h(x-y) p_{t-s}(X_s+\sigma \bar{B}^s;y)\lambda_d(dy).   
\end{equation}
By \cite[Proposition 4.7]{LPS22}, there exist $T_0>0$ and constants $c, C,\eta>0$ such that, for each $\ell\in\Cloc^{H-}\big(\R_+,\R^d\big)$, the density of $\Phi_\tau(\ell)$ admits the upper bound 
    \begin{equation}\label{eq:upper_bound_conditional_density}
      p_\tau(\ell;y)\leq \frac{e^{C(1+\tau^\eta\vertiii{\ell}_{\C^\gamma}^2)}}{\tau^{dH}}\exp\bigg(-\frac{c\big|y-\ell(\tau)\big|^2}{\tau^{2H}}\bigg) \qquad\forall\,y\in\R^d,
    \end{equation}
    for all $\tau\in(0, T_0]$ and $\gamma\in(0,H)$ with 
    $$\vertiii{\ell}_{\C^\gamma}=\|\ell\|_{\infty,[0,T_0]}+\sup_{0\le s<t\le T_0}\frac{|\ell(t)-\ell(s)|}{|t-s|^\gamma}.$$
Thus, if $0\le t-s\le T_0$, the above bound leads to 
\begin{equation}{\label{eq: bound density}}
 \sup_{y\in\ER^d} p_{t-s}(X_s+\sigma \bar{B}^s;y)\le \frac{1}{(t-s)^{dH}}e^{C\left(1+2(t-s)^\eta\left(|X_s|^2+\bar{\sigma}\vertiii{\bar{B}^s}_{\C^\gamma}^2\right)\right)},   
\end{equation}
with $\bar{\sigma}={\rm sup}_{|x|\le 1} |\sigma x|$. Hence, 
%$$
%\ES[\mathbb{K}_h(x-X_t)|{\cal F}_s]\le \frac{1}{(t-s)^{dH}}\ES\left[e^{C\left(1+2(t-s)^\eta\left(|X_s|^2+\bar{\sigma}\vertiii{\bar{B}^s}_{\C^\gamma}^2\right)\right)}\right].
%$$
from \eqref{eq: 20.5} and \eqref{eq: bound density} we obtain
\begin{align*}
\TCR{|} {\ES[\mathbb{K}_h(x-X_t)|{\cal F}_s]} \TCR{|}& {\le \int_{\R^d} \TCR{|}\mathbb{K}_h(x-y)\TCR{|} \frac{1}{(t-s)^{dH}}e^{C\left(1+2(t-s)^\eta\left(|X_s|^2+\bar{\sigma}\vertiii{\bar{B}^s}_{\C^\gamma}^2\right)\right)}\lambda_d(dy)}\\
 & {\le  \frac{\TCR{\|K\|_\infty}}{(t-s)^{dH}}e^{C\left(1+2T_0^\eta\left(|X_s|^2+\bar{\sigma}\vertiii{\bar{B}^s}_{\C^\gamma}^2\right)\right)} =: \frac{\Xi_{s}^{(1)}}{(t-s)^{dH}}},   
\end{align*}
\TCR{where, in the last line, we used that,
$$\int |\mathbb{K}_h(x-y)|\lambda_d(dy)= \int |K(x-y)|\lambda(dy)\le \|K\|_\infty.$$}
\tcr{Moreover, $\Xi_{s}^{(1)}$ is an $\mathcal{F}_s$-random variable whose distribution does not depend on $s$ and has finite moments of any order for $T_0$ small enough (as shown below). Indeed,} noting that
$$\bar{B}^s\overset{d}{=}\bar{B}^0,$$
we deduce from Fernique's Theorem that, for any $\gamma\in(0,H)$, there exists $\lambda_1>0$ such that
$$\ES[e^{\lambda_1\bar{\sigma}\vertiii{\bar{B}^s}_{\C^\gamma}^2}]<+\infty.$$
Furthermore, by \cref{prop:li_sieber} and the fact that the process is assumed to be stationary, there also exists $\lambda_2>0$
such that 
$$ \ES[e^{\lambda_2 |X_u|^2}]=\int e^{\lambda_2 |x|^2}\pi(dx)<+\infty.$$
Thus, at the price of replacing $T_0$ by $\tilde{T}_0=\sup\{u\ge0: 4Cu^\eta\le \lambda_1\quad\textnormal{and} \quad 4C\bar{\sigma} u^\eta\le \lambda_2\}$, we deduce from Cauchy-Schwarz inequality that for any $0\le s\le t$ such that $t-s\le T_0$,
$$
\ES[\mathbb{K}_h(x-X_t)|{\cal F}_s]\le \frac{\Xi^{(1)}_s}{(t-s)^{dH}}.
$$
where $\Xi^{(1)}_s$ satisfies the announced properties. When $t-s\ge T_0$, we can apply the tower property for conditional expectations:
$${ \ES[\ES[\mathbb{K}_h(x-X_t)|{\cal F}_{t-T_0}]|{\cal F}_s]\le \frac{C}{T_0^{dH}} \ES[\Xi^{(1)}_{t- T_0}|{\cal F}_s] =: \frac{\Xi^{(2)}_{s,t}}{T_0^{dH}}}.$$
{This is again $\mathcal{F}_s$-measurable by construction and it has bounded moments of any order. Indeed, for any $p$, }
$${\E[|\Xi^{(2)}_{s,t}|^p] = \E[|\ES[\Xi^{(1)}_{t- T_0}|{\cal F}_s]|^p] \le c \E[\ES[|\Xi^{(1)}_{t- T_0}|^p|{\cal F}_s]] = c \ES[|\Xi^{(1)}_{t- T_0}|^p]}$$
{which involves that the moments of $\Xi^{(2)}_{s,t}$ are uniformly bounded since the distribution of $\Xi^{(1)}_s$ does not depend on $s$ and have moment of any order}. Finally, the result is true by setting $\Xi_{s,t} := \Xi_s^{(1)} + \Xi_{s,t}^{(2)}$. This concludes the proof.
\end{proof}

\subsection{Proof of \cref{th:bound2}}\label{sec:prooftheo33}
We start with the  decomposition \eqref{eq: var split} of the proof of \TCR{\cref{prop: mixing}}, but we choose to decompose each term as follows. 
When $F=\mathbb{K}_h$,
$$\ES[F_T|{\cal F}_0]-\ES_{\tcr{\Pi}}[F_T]=\int_0^T \left(\ES[\mathbb{K}_h(x-X_t)|{\cal F}_0]-\pi(\mathbb{K}_h)\right) dt.$$
We have three different bounds for $|\ES[\mathbb{K}_h(x-X_t)|{\cal F}_0]-\pi(\mathbb{K}_h)|$. Using that $\pi(\mathbb{K}_h)\le \|\pi\|_\infty\lesssim 1$ (by \cref{lem:stationary_density_bounded} \tcr{in \cref{s: tech}}), a first trivial bound is:
$$|\ES[\mathbb{K}_h(x-X_t)|{\cal F}_0]-\pi(\mathbb{K}_h)|\le \TCR{(\prod_{l = 1}^d h_l)^{-1}} \|K\|_\infty\TCR{+\|\pi\|_\infty\lesssim (\prod_{l = 1}^d h_l)^{-1}}.$$
The second and third bounds come respectively from \cref{lem:espcondkh}  and \eqref{TVlongtimes;;}. These three bounds lead to 
\bqn
|\ES[\mathbb{K}_h(x-X_t)|{\cal F}_0]-\pi(\mathbb{K}_h)|\lesssim \min\left\{ \TCR{(\prod_{l = 1}^d h_l)^{-1}}, \frac{\Xi_{0,t}}{t^{dH}\vee 1}, 
\TCR{(\prod_{l = 1}^d h_l)^{-1}} \hat{\Xi}\left(\check{\Xi} e^{-ct}+(1+t)^{H-1+\frac{\varepsilon}{2}}\right)\right\}.
\eqn 
with $\tcr{\Xi_{0,t}}$ is defined in \cref{lem:espcondkh}, $\hat{\Xi}=\int \frC(W^{-}-{w^{-}},\epsilon)\PE_{W^{-}}(dw^{-})$ and $\check{\Xi}=\int |X_0-y|\pi(dy)$.
We thus use the first bound on $[0, \TCR{(\prod_{l = 1}^d h_l)^{\frac{1}{dH}}}]$, the second one on $[\TCR{(\prod_{l = 1}^d h_l)^{\frac{1}{dH}}}, t_0]$ where $t_0\in[1,T]$ will be calibrated later, and the third one on $[t_0, T]$. We have
\begin{align*}
\left|\int_0^T \ES[\mathbb{K}_h(x-X_t)|{\cal F}_0]-\pi(\mathbb{K}_h) dt\right|& \lesssim 
\tcr{\Xi_{0,t}}\left(\TCR{(\prod_{l = 1}^d h_l)^{-1 + \frac{1}{dH}}}+  t_0\right)\\
&+\int_{t_0}^T \TCR{(\prod_{l = 1}^d h_l)^{-1}}\hat{\Xi}\left(\check{\Xi} e^{-ct}+(1+t)^{H-1+\frac{\varepsilon}{2}}\right) dt.
\end{align*}
{One checks that
\begin{align*}
\ES\left(\int_{t_0}^T \TCR{(\prod_{l = 1}^d h_l)^{-1}} \hat{\Xi}\left(\check{\Xi} e^{-ct}+(1+t)^{H-1+\frac{\varepsilon}{2}}\right) dt\right)^2&\lesssim
 \TCR{(\prod_{l = 1}^d h_l)^{-2}}\left(e^{-2c t_0}+T^{2H+\varepsilon}\right)\\
&\lesssim  \TCR{(\prod_{l = 1}^d h_l)^{-2}} T^{2H+2\varepsilon}.
\end{align*}
Moreover, by \cref{lem:espcondkh}, $\tcr{\Xi_{0,t}}$ has uniformly bounded moments. Hence,
\begin{align*}
\ES\left(\frac{1}{T}\left(\ES[F_T|{\cal F}_0]-\ES_{\Pi}[F_T]\right)\right)^2&\lesssim \frac{\TCR{(\prod_{l = 1}^d h_l)^{-2 + \frac{2}{dH}}}}{T^2}+\frac{ t_0^2}{T^2}+\frac{ \TCR{(\prod_{l = 1}^d h_l)^{-2}} T^{2H+2\varepsilon}}{\ar{T^{2}}}.
%&\le \frac{h^{-d+\frac{1}{H}}}{T^2}+ \frac{h^{-\frac{2d}{1-H-\frac{\varepsilon}{2}}+\frac{d}{2}}}{T^{1-\frac{1}{2-2H-\varepsilon}}}.
\end{align*}
Since the bound does not depend on $t_0$, we choose $t_0=1$ and get
%An optimization of $t_0$ leads to $t_0=(h^{-2d} T)^{\frac{1}{3-2H-\varepsilon}}\wedge T$, and then to,
\begin{equation}\label{eq:bornehhTT}
\ES\left(\frac{1}{T}\left(\ES[F_T|{\cal F}_0]-\ES_{\tcr{\Pi}}[F_T]\right)\right)^2\lesssim_{\varepsilon} \frac{\TCR{(\prod_{l = 1}^d h_l)^{-2}}}{\ar{T^{2}}}\min\left\{\TCR{(\prod_{l = 1}^d h_l)^{\frac{2}{Hd}}},T^{2H+\varepsilon}\right\}.
%\frac{h^{-2d+\frac{2}{H}}}{T^2}+\frac{ h^{-2d}}{T} \times \frac{ h^\frac{2d(1-4H+\varepsilon)}{3-2H}}{T^{\frac{1-2H-\varepsilon}{3-2H-\varepsilon}}}.
\end{equation}
Note that, without loss of generality, we replaced $\varepsilon$ by $\varepsilon/2$.
}
%Then, by Jensen inequality, for any $\varepsilon\in(0,2-2H)$,
%\tcr{Ancienne version ici}
%\begin{align*}
%\ES\left(\int_{t_0}^Th^{-d} \hat{\Xi}\left(\check{\Xi} e^{-ct}+(1+t)^{H-1+\frac{\varepsilon}{2}}\right) dt\right)^2&\lesssim
%(T-t_0) h^{-2d}\left(e^{-2c t_0}+(1+t_0)^{2H-1+\varepsilon}\right)\\
%&\lesssim T h^{-2d} t_0^{2H-1+\varepsilon}.
%\end{align*}

%\tcr{Moreover, by \cref{lem:espcondkh}, $\tcr{\Xi_{0,t}}$ has uniformly bounded moments. Hence,}
%\begin{align*}
%\ES\left(\frac{1}{T}\left(\ES[F_T|{\cal F}_0]-\ES_{\Pi}[F_T]\right)\right)^2&\lesssim \frac{h^{-2d+\frac{2}{H}}}{T^2}+\frac{ t_0^2}{T^2}+\frac{ h^{-2d} t_0^{2H-1+\varepsilon}}{T}.
%&\le \frac{h^{-d+\frac{1}{H}}}{T^2}+ \frac{h^{-\frac{2d}{1-H-\frac{\varepsilon}{2}}+\frac{d}{2}}}{T^{1-\frac{1}{2-2H-\varepsilon}}}.
%\end{align*}
%An optimization of $t_0$ leads to $t_0=(h^{-2d} T)^{\frac{1}{3-2H-\varepsilon}}\wedge T$, and then to,
%\begin{equation}\label{eq:bornehhTTbis}
%\ES\left(\frac{1}{T}\left(\ES[F_T|{\cal F}_0]-\ES_{\tcr{\Pi}}[F_T]\right)\right)^2\lesssim \frac{h^{-2d+\frac{2}{H}}}{T^2}+\frac{ h^{-2d}}{T} \times \frac{ h^\frac{2d(1-4H+\varepsilon)}{3-2H}}{T^{\frac{1-2H-\varepsilon}{3-2H-\varepsilon}}}.
% t_0^{2H-1+\varepsilon}}{T}
%&\le \frac{h^{-d+\frac{1}{H}}}{T^2}+ \frac{h^{-\frac{2d}{1-H-\frac{\varepsilon}{2}}+\frac{d}{2}}}{T^{1-\frac{1}{2-2H-\varepsilon}}}.
%\end{equation}
%\tcr{Fin ancienne version}

Let us now consider the other terms of \eqref{eq: var split}. First, for every $k\in \{ 1, ... , T\}$, 
\begin{equation}\label{eq:part24409}
\ES[F_T|{\cal F}_k]-\ES[F_T|{\cal F}_{k-1}]=\int_{k-1}^T \left(\ES[\mathbb{K}_h(x-X_t)|{\cal F}_k]-\ES[\mathbb{K}_h(x-X_t)|{\cal F}_{k-1}]\right) dt.
\end{equation}
\tcr{We split the integral into two parts. For $t\in(k,T)$,
\begin{align*}
&|\ES[\mathbb{K}_h(x-X_t)|{\cal F}_k]-\ES[\mathbb{K}_h(x-X_t)|{\cal F}_{k-1}]|\\
&\quad \lesssim \min\left\{ \TCR{(\prod_{l = 1}^d h_l)^{-1}}, \tcr{\frac{\Xi_{k,t}+\Xi_{k-1,t}}{(t-k)^{dH}\vee 1}}, 
\TCR{(\prod_{l = 1}^d h_l)^{-1}} \tilde{\Xi}_k\left(\bar{\Xi}_k e^{-c(t-k)}+(t-k)^{H-\frac{5}{2}}\right)\right\}.
\end{align*}
with $\Xi_{s,t}$ is defined in \cref{lem:espcondkh} and $\tilde{\Xi}_k$ and 
$\bar{\Xi}_k$ have uniformly bounded moments (using a  similar control as in \eqref{eq:aveck??}). We thus use the first bound on $[k, k+ \TCR{(\prod_{l = 1}^d h_l)^{\frac{1}{dH}}}]$, the second one on $[k+\TCR{(\prod_{l = 1}^d h_l)^{\frac{1}{dH}}}, k+{t}_{k}]$ where $t_{k}\in[h^{\frac{1}{H}},T-k]$ will be calibrated below. We have
\begin{align*}
\int_{k}^T&\left|\ES[\mathbb{K}_h(x-X_t)|{\cal F}_k]-\ES[\mathbb{K}_h(x-X_t)|{\cal F}_{k-1}] \right| dt\\
& \lesssim 
(\Xi_{k,t}+\Xi_{k-1,t})(\TCR{(\prod_{l = 1}^d h_l)^{-1 + \frac{1}{dH}}}+   t_k)+\int_{t_k}^{T-k} \TCR{(\prod_{l = 1}^d h_l)^{-1}}\hat{\Xi}_k\left(\check{\Xi}_k e^{-ct}+t^{H-\frac{5}{2}}\right) dt.
\end{align*}
Thus, 
\begin{align}
\ES&\left(\int_{k}^T \left(\ES[\mathbb{K}_h(x-X_t)|{\cal F}_k]-\ES[\mathbb{K}_h(x-X_t)|{\cal F}_{k-1}]\right) dt\right)^2\nonumber\\
&\lesssim {\TCR{(\prod_{l = 1}^d h_l)^{-2 + \frac{2}{dH}}}}+
{ t_k^2}+{ \TCR{(\prod_{l = 1}^d h_l)^{-2}} t_k^{2H-3}}\lesssim 
\TCR{(\prod_{l = 1}^d h_l)^{-2 + \frac{2}{dH}}}+\TCR{(\prod_{l = 1}^d h_l)^{-\frac{4}{5 - 2H}}},\label{eqa:freekkm1}
\end{align}
where the last inequality follows by optimizing  over $t_k$ leading to the choice $t_k=\TCR{(\prod_{l = 1}^d h_l)^{-\frac{2}{5 - 2H}}}\wedge T$.
%and then to:
%$$\ES\left(\frac{1}{T}\left(\ES[F_T|{\cal F}_k]-\ES[F_T|{\cal F}_{k-1}]\right)\right)^2\lesssim 
%\frac{h^{-2d+\frac{2}{H}}}{T^2}+ \frac{h^{-\frac{4d}{5-2H}}}{T^2}.$$
%Thus, 
%\begin{align*}
%\sum_{k=1}^T \ES\left(\frac{1}{T}\left(\ES[F_T|{\cal F}_k]-\ES[F_T|{\cal F}_{k-1}]\right)\right)^2&\lesssim 
}

\tcr{Let us finally consider the part $\int_{k-1}^k$ of \eqref{eq:part24409}. For $k\ge1$, we have
$$ \ES\left(\int_{k-1}^k \left(\mathbb{K}_h(x-X_t)-\ES[\mathbb{K}_h(x-X_t)|{\cal F}_{k-1}]\right) dt\right)^2=
2\int_{k-1}^k\int_{s}^k C_k(s,u) du,
$$
where $C_k$ denotes the covariance-type function defined by:
$$C_k(s,u)=\ES\left[\left(\mathbb{K}_h(x-X_u)-\ES[\mathbb{K}_h(x-X_u)|{\cal F}_{k-1}]\right)\left(\mathbb{K}_h(x-X_s)-\ES[\mathbb{K}_h(x-X_s)|{\cal F}_{k-1}]\right)\right].$$
Conditioning with respect to ${\cal F}_s$ we have, for every $k-1\le s\le u\le k$,
$$|C_k(s,u)|\lesssim h^{-d}\ES\left|\ES[\mathbb{K}_h(x-X_u)|{\cal F}_s]-\ES[\mathbb{K}_h(x-X_u)|{\cal F}_{k-1}]\right|
.
$$
With similar arguments as above,
$$|\ES[\mathbb{K}_h(x-X_u)|{\cal F}_s]-\ES[\mathbb{K}_h(x-X_u)|{\cal F}_{k-1}]|\lesssim \min\left\{\TCR{(\prod_{l = 1}^d h_l)^{-1}}, \frac{\Xi_{s,u}+\Xi_{k-1,u}}{(u-s)^{dH}}\right\}.
$$
Cutting $\int_s^{k}$ into two parts (at time $(s+\TCR{(\prod_{l = 1}^d h_l)^{\frac{1}{dH}}})\wedge k$), we deduce that 
$$\int_{k-1}^k\int_{s}^k C_k(s,u) du\lesssim \TCR{(\prod_{l = 1}^d h_l)^{-2 + \frac{1}{dH}}}.$$
}
Combined with \eqref{eqa:freekkm1} and plugged into \eqref{eq:part24409}, this leads to:
$$\ES\left(\frac{1}{T}\left(\ES[F_T|{\cal F}_k]-\ES[F_T|{\cal F}_{k-1}]\right)\right)^2\lesssim 
\frac{\TCR{(\prod_{l = 1}^d h_l)^{-2 + \frac{1}{dH}}}}{T^2}+ \frac{\TCR{(\prod_{l = 1}^d h_l)^{-\frac{4}{5 - 2H}}}}{T^2}.$$
Thus, 
\begin{align*}
&\sum_{k=1}^T \ES\left(\frac{1}{T}\left(\ES[F_T|{\cal F}_k]-\ES[F_T|{\cal F}_{k-1}]\right)\right)^2\lesssim 
\frac{\TCR{(\prod_{l = 1}^d h_l)^{-2 + \frac{1}{dH}}}}{T}+ \frac{\TCR{(\prod_{l = 1}^d h_l)^{-\frac{4}{5 - 2H}}}}{T}\\
&\qquad \lesssim \frac{\TCR{(\prod_{l = 1}^d h_l)^{-2}}}{T}\max\left(\TCR{(\prod_{l = 1}^d h_l)^{\frac{1}{dH}}}, \TCR{(\prod_{l = 1}^d h_l)^{\frac{2(3 - 2H)}{5 - 2H}}}\right).
\end{align*}
Combining the above bound with \eqref{eq:bornehhTT} leads to the result.

\section{Proof of main results}{\label{s: proof stat}}

This section is devoted to the proof of our main results on the final convergence rate for our estimation problem, that is, Theorems \ref{th: rate start} and \ref{th: rate plus}.
\subsection{Proof of \cref{th: rate start}}\label{sec:prooftheo2}

In order to get the \tcr{announced} convergence rates we decompose the mean squared error through the bias-variance decomposition. It yields
$$\mathbb{E}[|\hat{\pi}_{h,T}(x) - \pi (x)|^2] = |\mathbb{E}[\hat{\pi}_{h,T}(x)] - \pi(x)|^2 + {\rm Var}(\hat{\pi}_{h,T}(x)).$$
Then, it is well-known that the bias is upper bounded by $c \TCR{ \sum_{l = 1}^d h_l^{\beta_l}}$ (see, for example, Proposition 2 of \cite{adaptive}). Let us \tcr{state} such a bound in our context, under our (more general) hypothesis on the drift coefficient. The proof of \cref{prop: bias} can be found in \cref{s: tech}. 
\begin{proposition}{\label{prop: bias}}
\ar{Assume that \(\pi\) belongs to the H\"older class \(\mathcal{H}_d(\beta, L)\).} Then, there exists $\Lambda>0$ such that if  \ref{hyp1} holds with $\lambda\le \Lambda$, then there exists a constant $c$ such that \tcr{for all $T>0$ and $x \in \R^d$,}
$$|\mathbb{E}[\hat{\pi}_{h,T}(x)] - \pi(x)| \le c \TCR{ \sum_{l = 1}^d h_l^{\beta_l}}.$$
\end{proposition}

From \cref{th: upper bound} and \cref{prop: bias} we deduce
$$\mathbb{E}[|\hat{\pi}_{h,T}(x) - \pi (x)|^2] \le c \TCR{ \sum_{l = 1}^d h_l^{\beta_l}} + \frac{c}{T \TCR{(\prod_{l = 1}^d h_l)^{2}}} 1_{H < \frac{1}{2}} + \frac{c_\varepsilon}{T^{2 - 2H - \varepsilon}\TCR{(\prod_{l = 1}^d h_l)^{2}}} 1_{H > \frac{1}{2}}.$$
We want to choose $h= (h_1, ... h_d)$ such that the quantity here above is as small as possible. We start looking for the trade-off in the case $H < \frac{1}{2}$. Thus, we set $\TCR{h_l = \left(\frac{1}{T}\right)^{a_l}}$ and we want to minimize 
$$\TCR{ \sum_{l = 1}^d \left(\frac{1}{T}\right)^{2 a_l \beta_l} + \left(\frac{1}{T}\right)^{1 - 2 \sum_{l = 1}^d a_l } }.$$
\TCR{The trade-off provides the condition $2 a_1 \beta_1 = ... = 2 a_d \beta_d = 1 - 2 \sum_{l = 1}^d a_l$. Observe that, as a consequence of the first $d-1$ equations, we can write $a_l = \frac{\beta_d}{\beta_l} a_d$ for any $l \in \{1, ... , d \}$. Therefore, the last identity becomes 
\begin{align*}
2 a_d \beta_d & = 1 - 2 \beta_d a_d  \sum_{l = 1}^d \frac{1}{\beta_l} = 1 - 2 \beta_d a_d \frac{d}{\bar{\beta}}. 
\end{align*}
This implies that $2 a_d \beta_d (1 + \frac{d}{\bar{\beta}} ) = 1$, providing $$a_d = \frac{\bar{\beta}}{2 \beta_d (\bar{\beta} + d)}, \qquad \mbox{ and so } a_l = \frac{\bar{\beta}}{2 \beta_l (\bar{\beta} + d)}, \, \forall l \in \{1, ... , d \}.$$}
Replacing it in the bias-variance decomposition it implies that
$$\mathbb{E}[|\hat{\pi}_{h,T}(x) - \pi (x)|^2] \le c \TCR{\left(\frac{1}{T}\right)^{\frac{\bar{\beta}}{\bar{\beta} + d}}}.$$
This concludes the proof of \cref{th: rate start} in the case $H < \frac{1}{2}$. 

The proof in the case $H > \frac{1}{2}$ follows the same argument. We want to minimize the quantity 
$$\TCR{ \sum_{l = 1}^d \left(\frac{1}{T}\right)^{2 a_l \beta_l} + \left(\frac{1}{T}\right)^{2 - 2H - \varepsilon -2\sum_{l = 1}^d a_l}}.$$
Acting as above, it leads us to \TCR{$a_l =  \frac{(1 - H) \bar{\beta}}{\beta_l(\bar{\beta} + d)} - \varepsilon $ for any $l \in \{1, ... , d \}$},
 that provides the wanted convergence rate \TCR{$\left(\frac{1}{T}\right)^{\frac{2 {\bar{\beta}}(1 - H)}{\bar{\beta} + d}- \varepsilon}$}.

\subsection{Proof of \cref{th: rate plus}}\label{sec:prooftheo44}

The proof of \cref{th: rate plus} follows along the same lines as that of \cref{th: rate start}. It relies on the bias–variance decomposition, together with \cref{prop: bias} and the control of the stochastic term described in \cref{th:bound2}. Fine-tuning the bandwidth \( h \) depends on which term dominates the variance bound.
Specifically, we adjust \TCR{$h_l(T) = \frac{1}{T^{a_l}}$ to minimize
$$\phi(a_1, ..., a_d):= \sum_{l = 1}^d \left(\frac{1}{T}\right)^{2 a_l \beta_l} + \left(\frac{1}{T}\right)^{M},$$
where
$$
M:=(1 + \sum_{l = 1}^d a_l(\frac{1}{dH} - 2) ) \land (1 + 2(\frac{3 - 2H}{5 - 2H} - 1 )\sum_{l = 1}^d a_l) \land (2 - 2H - 2 \sum_{l = 1}^d a_l  - \varepsilon).$$}

\TCR{In order to better understand this object, let us start by considering the isotropic case, for which $a_1 = ... = a_d =: a$. In this case, }
$$M = (1 - a \alpha_{d,H}) \land (2 - 2H - 2da - \varepsilon). $$
We note that \(a\mapsto 2a\beta\) is increasing, while both \(1 - a\alpha_{d,H}\) and \(2 - 2H - 2da - \varepsilon\) are decreasing functions of \(a\), given that \(\alpha_{d,H} > 0\). In such a case, a simple study shows that $\phi$ attains its minimum at \(\tilde{a}_0 := \min(\tilde{a}_1,\tilde{a}_2)\)
where 
$$\tilde{a}_1={\rm Argmin}_a 2a\beta\land (1 - a\alpha_{d,H})\quad\textnormal{and}\quad \tilde{a}_2= {\rm Argmin}_a 2a\beta\land
(2 - 2H - 2da - \varepsilon).$$
%\min\left(\frac{1}{2\beta + \alpha_{d,H}}, \frac{1 - H - \varepsilon}{\beta + d}\right)\)

This yields \(\tilde{a}_0 := \min\left(\frac{1}{2\beta + \alpha_{d,H}}, \frac{1 - H - \varepsilon}{\beta + d}\right)\) and leads to the following convergence rate:
$$\left(\frac{1}{T}\right)^{2 \tilde{a}_0 \beta} = \left(\frac{1}{T}\right)^{\frac{2\beta}{2\beta + \alpha_{d,H}} \land \frac{2 \beta (1 - H)}{\beta + d} - \varepsilon}.$$
\TCR{To conclude the proof, let us now turn to the anisotropic case. Recall that the first $d-1$ constraints can be written as  \begin{equation}\label{eq: constraint al}
  a_l = \frac{\beta_d}{\beta_l} a_d, \qquad \text{for any } l \in \{1, \ldots, d\},
\end{equation}
so that  
\[
\sum_{l = 1}^d a_l = a_d \beta_d \sum_{l = 1}^d \frac{1}{\beta_l} = a_d \beta_d \frac{d}{\bar{\beta}}.
\]
Similarly as before, we observe that
\[
M=(1 - a_d \tfrac{\beta_d}{\bar{\beta}} \alpha_{d,H}) \wedge (2 - 2H - 2d\, a_d \tfrac{\beta_d}{\bar{\beta}}).
\]
The mapping $a_d \mapsto 2 a_d \beta_d$ is increasing, while both 
$1 - a_d \tfrac{\beta_d}{\bar{\beta}} \alpha_{d,H}$ and 
$2 - 2H - 2d\, a_d \tfrac{\beta_d}{\bar{\beta}}$ are decreasing functions of $a_d$, since $\alpha_{d,H}$ and the regularity parameters are positive. 
Hence, the optimal choice corresponds to  
\[
\tilde{a}_{d,0} := \min(\tilde{a}_{d,1}, \tilde{a}_{d,2}),
\]
where  
\[
\tilde{a}_{d,1} = \operatorname*{argmin}_{a_d} \big( 2a_d\beta_d \wedge (1 - a_d \tfrac{\beta_d}{\bar{\beta}} \alpha_{d,H}) \big), 
\quad 
\tilde{a}_{d,2} = \operatorname*{argmin}_{a_d} \big( 2a_d\beta_d \wedge (2 - 2H - 2a_d \tfrac{\beta_d}{\bar{\beta}} - \varepsilon) \big).
\]
This yields  
\[
\tilde{a}_{d,0} = \tfrac{\bar{\beta}}{\beta_d} 
\min\!\left( \frac{1}{2\bar{\beta} + \alpha_{d,H}}, \frac{1 - H - \varepsilon}{\bar{\beta} + d} \right).
\]
Using \eqref{eq: constraint al}, we deduce that the minimum of $\phi$ is attained at $(\tilde{a}_{1,0}, \ldots, \tilde{a}_{d,0})$, where  
\[
\tilde{a}_{l,0} := \tfrac{\bar{\beta}}{\beta_l} 
\min\!\left( \frac{1}{2\bar{\beta} + \alpha_{d,H}}, \frac{1 - H - \varepsilon}{\bar{\beta} + d} \right).
\]
This completes the proof and yields the announced convergence rate.
}

\bibliographystyle{plain}
\bibliography{references}

\appendix
\section{Proof of Proposition \ref{prop: mixing}}{\label{sec:martproof}}
\fab{We start from the martingale decomposition \eqref{eq:martdecomp429} but in order to slightly simplify the purpose, we assume that $T$ is an integer\footnote{If not, it extends easily  by using that $\ES[(F_T-\ES[F_T|{\cal F}_{\lfloor T\rfloor}])^2]\le 4\|F\|_\infty^2$.}}. 
%\tcr{Fabien: An alternative could be to consider the sum until $k=-\infty$.}
Thus,  
\begin{equation}{\label{eq: var split}}
\ES[(F_T-\ES[F_T])^2]\le 2\ES\left[(\ES[F_T|{\cal F}_{0}]-\ES[F_T])^2\right]+2\sum_{k=1}^T \ES\left[\left(\ES[F_T|{\cal F}_k]-\ES[F_T|{\cal F}_{k-1}]\right)^2\right].
\end{equation}
On the one hand, \tcr{under the stationarity assumption, we have} (using the same notation as in \cref{prop:TVTV})
\begin{equation*}{\label{eq:int}} \begin{split}
&\ES[F_T|{\cal F}_0]-\ES[F_T]\\
&=\int_0^T \int\ES[F(\Phi_{\tau}(\ell(\tcr{X_0},W^{-})))|\TCR{X_0},{W^{-}}]-\ES[F(\Phi_{\tau}(\ell(y,w^{-})))]\tcr{\Pi}(dy,dw^{-}) d\tau.   
\end{split}
\end{equation*}
By \cref{prop:TVTV}, 
\begin{align}
&|\ES[F(\Phi_{\tau}(\ell(\tcr{X_0},W^{-})))|\TCR{X_0},{W^{-}}]-\ES[F(\Phi_{\tau}(\ell(y,w^{-})))]|\lesssim \nonumber\\
&\qquad \|F\|_\infty \left(e^{-c\tau}|\tcr{X_0}-y|+\TCR{\frC(W^{-}-{w^{-}},\epsilon)}(1\vee\tau)^{H-1+\frac{\varepsilon}{2}}\right).\label{TVlongtimes;;}
\end{align}
{Recall that the first and second marginals of $\Pi$ are respectively $\pi$ and $\PE_{W^{-}}$. Then, we deduce that 
$$\TCR{|}\ES[F_T|{\cal F}_{0}]-\ES[F_T]\TCR{|}\lesssim \|F\|_\infty \left(\int|X_0-y|\pi(dy)+\TCR{\int \frC(W^{-}-{w^{-}},\epsilon)\PE_{W^{-}}(dw^{-})T^{H+\frac{\varepsilon}{2}}}\right).
$$}
{It is classical that $\sup_{t\ge0} (1\wedge t)^{-\frac{1}{2}-\varepsilon}|B_t|$ has moments \TCR{of any order}. It easily follows that $\frC$ defined in \eqref{eq:mathfrakCesp} satisfies: 
\begin{equation}\label{eq:boundedfrC}
\int \frC^2(\tilde{w}^{-}-{w^{-}},\epsilon)\PE_{W^{-}}^{\otimes2}(d\tilde{w}^{-},dw^{-})<+\infty.
\end{equation}
Hence, for every $T\ge 1$,
$$\ES\left[(\ES[F_T|{\cal F}_{0}]-\ES[F_T])^2\right]\lesssim_{\varepsilon}  \|F\|_\infty ^2\left( \int |x-y|^2\pi^{\otimes 2}(dx,dy)+T^{2H+\varepsilon}\right)\lesssim_{\varepsilon} T^{2H+\varepsilon} \|F\|_\infty^2 .$$
On the other hand, { for all $k\in \left \{ 1, ... , T \right \}$,
\begin{align*}
\ES[F_T|{\cal F}_k]-\ES[F_T|{\cal F}_{k-1}]&= \int_{k-1}^T\ES[F(X_s)|{\cal F}_k]-\ES[F(X_s)|{\cal F}_{k-1}] ds.
\end{align*}
To simplify the notations, let us detail the case $k=1$ (the generalization easily follows from a Markov argument). We have for $s\ge1$}
\begin{align*}
&|\ES[F(X_s)|{\cal F}_1]-\ES[F(X_s)|{\cal F}_{0}]| \\
& \qquad \le\|F\|_\infty \int \|{\cal L}(X_s|{\cal F}_1)-{\cal L}(X_s|{\cal F}_1,(W_t)_{t\in[0,1]}=(w_t)_{t\in[0,1]})\|_{TV}\PE_{W}(dw).
\end{align*}
At the price of a translation from $[0,1]$ to $[-1,0]$, we are exactly in the setting of the second statement of \cref{prop:TVTV}. More precisely, the conditioning of the Wiener process on $[0,1]$ implies that we are considering two solutions for which the past before time $1$ differs only on $[0,1]$. We thus deduce from \cref{prop:TVTV} that for $s\ge1$
\begin{align*}
\|{\cal L}(X_s|{\cal F}_1)&-{\cal L}(X_s|{\cal F}_1,(W_t)_{t\in[0,1]}=(w_t)_{t\in[0,1]})\|_{TV}
\\
&\le  \left(e^{-c(s-1)}|X_1-X_1^{w}|+\TCR{\|w-W\|_{\infty,[0,1]}}((s-1)\vee 1)^{H-\frac{5}{2}}\right),
\end{align*}
where $X_1^{w}$ denotes the solution at time $1$ starting from $X_0$ but with Wiener path $w$ on $[0,1]$. Since $X_1$ and $X_1^{w}$ have the same starting point $X_0$, a Gronwall argument implies that
$$|X_1-X_1^{w}|\le \ar{\sigma} \int_0^1 e^{[b]_{\rm Lip}(s-t)}|\tilde{B}_s-\tilde{B}_s^{w}| ds\le \ar{\sigma} {\cal T}(1,W,w),$$
where $\tilde{B}$ (or $\tilde{B}^W$) denotes the Liouville process (defined by \eqref{def:liouville}),  $\tilde{B}_s^{w}$ the Liouville path driven by $w$ and
$${\cal T}(k,W,w)=\sup_{s\in[k-1,k]}|\tilde{B}_s^{k-1}-\tilde{B}_s^{k-1,w}| \textnormal{ with} \quad  \tilde{B}_s^{t}=\tilde{B}_s-\tilde{B}_t.$$
From what precedes and a Markov argument for extending to any $k$, we get for $s\ge k$
\begin{equation*}\label{eq:aveck??}
\begin{split}
 &|\ES[F(X_s)|{\cal F}_k]-\ES[F(X_s)|{\cal F}_{k-1}]| \\
&\ar{\lesssim}\|F\|_\infty \int\left(e^{-c(s-k)}{\cal T}(k,W,w)+ \TCR{\|w^{k-1}-W^{k-1}\|_{\infty,[k-1,k]}}(\TCR{(s-k)\vee 1})^{H-\frac{5}{2}}\right)\PE_{W}(dw),
\end{split}
\end{equation*}
where $w^{t}_s=w_s-w_t$ and $W^{t}_s=W_s-W_t$. The random variables involved in the above inequality do not depend on $s$ and have moments of any order. This implies that 
$$\ES\left[\left(\ES[F_T|{\cal F}_k]-\ES[F_T|{\cal F}_{k-1}]\right)^2\right]\lesssim \|F\|_\infty^2 \left(\int_{k-1}^T ((t-k)\vee 1)^{H-\frac{5}{2}} dt\right)^2 \lesssim \|F\|_\infty^2. $$}
%\Phi_{\tau}(\ell(x,w^{-})) - {\cal L}(\Phi_{\tau}(\ell(Y_1,\tilde{w})))\|_{{\rm TV}}

%To this end, let us denote by $Y_t^{w^+}=\Phi_t^{w^+}(\ell(X_0,W^{-}))$, the solution to the SDE when $(W_t)_{t\in[0,1]}=(w_t^+)_{t\in[0,1]}$. We have

%&\le \|F\|_\infty\int_0^{T-k-1} \int G_\tau(X_{k}, X_{k-1},(W_{k+t}-W_k)_{t\le 0},(w_t^+)_{t\in[0,1]})\PE_{W}(dw^+) d\tau\\
%&\textnormal{with,}\quad G(x,y,(w_t^{-})_{t\le 0},(w_t^+)_{t\in[0,1]})=\|{\cal L}(\Phi_{\tau}(\ell(x,w^{-})) - {\cal L}(\Phi_{\tau}(\ell(Y_1,\tilde{w})))\|_{{\rm TV}}

%X_{k},\Py,W^{(k),-,W^{(k-1)},-}\sqcup w)
%\ES[F(\Phi_{s-k-1}(\ell(X_{k-1}},w^{-})))]\tcr{\Pi}(dy,dw^{-}) d\tau.   
%\end{align*}
%}
%\int_{k-1}^T\int\ES[F(\Phi_{s-k}(\ell(\tcr{X_{k}},W^{-})))|{W^{-}}]-\ES[F(\Phi_{s-k-1}(\ell(X_{k-1}},w^{-})))]\tcr{\Pi}(dy,dw^{-}) d\tau.   

%Let us emphasize that we are using the same notation $\pi$ for both the density and the first marginal. Specifically, in \eqref{eq:int}, we integrate with respect to the density of the pair: $\pi(dy,dw^{-})$. Subsequently, we integrate with respect to $\pi(dx,d\tilde{w}^{-})$ as well, and we isolate the dependence on $\PE_{W^{-}}^{\otimes2}(d\tilde{w}^{-},dw^{-})$. This leaves us with the integral $\pi^{\otimes 2}(dx,dy)$, which now denotes the density of the first marginals.\\
\noindent {Thus,
$$\sum_{k=1}^T \ES\left[\left(\ES[F_T|{\cal F}_k]-\ES[F_T|{\cal F}_{k-1}]\right)^2\right]\lesssim T {\|F\|_\infty^2}.$$
Plugging into \eqref{eq: var split}, this yields
$$ {\rm Var}\lp\frac{1}{T} \int_0^T F(X_s) ds \rp  \le \frac{1}{T^2}\|F\|_\infty^2(c T + c_\epsilon T^{2 H + \epsilon}). $$
The result follows (taking, for example, $\varepsilon=\frac{1-2H}{2}$ when $H<1/2$). }
\noindent \hfill $\square$

\TCR{
\section{Concentration inequalities and discrete observations}{\label{s: CI}}
In this section, we focus on the proof of \cref{th: upper bounddiscrete}, addressing the case of discrete observations and concentration inequalities. To this end, we provide below a general concentration inequality that can be applied in both the discrete and continuous settings.}

\TCR{
We focus on 
\begin{equation}\label{eq:generalfuncrefmes}
F_{T,\mes} = \int_0^T F(X_s) \, \mes(ds),
\end{equation}
where \( \mes \) is a positive measure on \( \mathbb{R}_+ \). When \( \mes \) is the Lebesgue measure, we recover the standard integral of \( (F(X_s))_{0 \le s \le T} \), and when \( \mes = \sum_{i \ge 0} \delta_{i\Delta} \) with \( T = n \Delta \), this allows us to apply our results below to \( \sum_{i=0}^{n} F(X_{i\Delta}) \).
}

 \TCR{
We set
\begin{align*}
&a_{0,\mes}=\int_0^T e^{-c\tau}\mes(d\tau),\quad b_{0,\mes,\varepsilon}=\int_0^T(1+\tau)^{H-1+\frac{\varepsilon}{2}} \mes(d\tau)\quad\textnormal{and for $k\in\llbracket 1,\lfloor T\rfloor \rrbracket $,}\\
&a_{k,\mes}=\int_{0}^{T-k} e^{- c\tau}\mes(d\tau),\quad b_{k,\mes}=\int_0^{T-k}(1+\tau)^{H-\frac{5}{2}}\mes(d\tau).
    \end{align*}
    We finally set $a_{\lceil T\rceil,\mes}=b_{\lceil T\rceil,\mes}=1$.}

\TCR{
In the sequel, we usually write $b_{0,\mes}$ instead of $b_{0,\mes,\varepsilon}$. Furthermore, the constant $C_\varepsilon$ in the next proposition depends only on $\varepsilon$ but may change from line to line.
\begin{theorem}\label{prop:concentrationdisccontinu}
There exists $\Lambda>0$ such that if  \ref{hyp1} holds with $\lambda\le \Lambda$, then, for all $\varepsilon>0$ sufficiently small, there exist $C_\varepsilon>0$ such that 
$$\ES\left[e^{\lambda |F_{T,\mes}-\ES[F_{T,\mes}]|}\right]\le e^{C_\varepsilon\lambda^2 \|F\|_\infty^2\left(\sum_{k=0}^{\lceil T\rceil} a_{k,\mes}^2+b_{k,\mes}^2\right)}.
$$
In this case, for all $x>0$,
\begin{equation}\label{eq:generalconcentration}
\PE\left(|F_{T,\mes}-\ES[F_{T,\mes}]|>x\right)\le \exp\left({- \frac{x^2}{C_\varepsilon \|F\|_\infty^2\left(\sum_{k=0}^{\lceil T\rceil} a_{k,\mes}^2+b_{k,\mes}^2\right) }}\right).  
\end{equation}
As a consequence, for all $x>0$, 
$$\PE\left(\left|\frac{1}{T}\int_0^T F(X_s) ds-\pi(F)\right|>x\right)\le \exp\left(- \frac{x^2 T^{1\wedge (2-2H+\varepsilon)}}{C_\varepsilon \|F\|_\infty^2}\right),$$
and if $n\Delta\ge 1$ and $\Delta\le 1$,
$$\PE\left(\left|\frac{1}{n}\sum_{k=0}^{n-1} F(X_{i\Delta})-\pi(F)\right|>x\right)\le \exp\left(-\frac{x^2 (n\Delta)^{1\wedge (2-2H+\varepsilon)}}{C_\varepsilon \|F\|_\infty^2} \right).$$
\end{theorem}}
\TCR{
In order to slightly simplify the notations, we assume that $T$ is an integer in the proofs below (but the extension to the general case does not add specific difficulties).
\begin{proof}
 Set $M_k=\ES[F_{T,\mes}|{\cal F}_k]$ for $k\ge0$. Since $F_{T,\mes}-\ES[F_{T,\mes}]= \sum_{k\ge 1} M_{k}-M_{k-1}+ \ES[F_{T,\mes}|{\cal F}_0]-\ES[F_{T,\mes}],$ we deduce from \cref{lem:conc245} below and from an induction (see \cite{adaptive} for a similar approach) that 
\begin{align*}
\ES\left[e^{\lambda |F_{T,\mes}-\ES[F_{T,\mes}]|}\right]&\le \ES[e^{\lambda |M_T-M_{T-1}|}|{\cal F}_{T-1}] \ES[e^{\sum_{k=1}^{T-1}\lambda|M_k-M_{k-1}|+\ES[F_{T,\mes}|{\cal F}_0]-\ES[F_{T,\mes}]}]\\
&\le e^{C\lambda^2\|F\|_\infty^2(a_{T,\mes}^2+b_{T,\mes}^2)}.
\end{align*}
The bound \eqref{eq:generalconcentration} is then a classical consequence of the (exponential) Markov property (combined with an optimization of $\lambda$). For the third one, we remark that in this case: $a_{0,\mes}\lesssim 1$, $b_{0,\mes,\varepsilon}\lesssim T^{H+\varepsilon/2} $ and for $k\ge1$,
\begin{align*}
a_{k,\mes}\lesssim 1\quad \textnormal{and}\quad b_{k,\mes}\lesssim 1,
\end{align*}
which leads to the result (applying to $\tilde{x}=Tx$). For the last one, we apply the general bound \eqref{eq:generalconcentration} with 
$T=n\Delta$ and $\mes=\sum_{i=0}^{n-1} \delta_{i\Delta}$. In this case, 
$a_{0,\mes}=\sum_{i=0}^{n-1} e^{-c i\Delta}\lesssim \Delta^{-1}\vee 1\lesssim \Delta^{-1}$,
since we assume that $\Delta\le 1$.
One also checks that 
$$b_{0,\mes,\varepsilon}\lesssim \Delta^{-1}+\Delta^{-1}(n\Delta)^{H+\frac{\varepsilon}{2}},$$
and that 
\begin{align*}
a_{k,\mes}\lesssim \Delta^{-1}\quad \textnormal{and}\quad b_{k,\mes}\lesssim \Delta^{-1}.
\end{align*}
This yields 
$$\sum_{k=0}^{\lceil n\Delta\rceil} a_{k,\mes}^2+b_{k,\mes}^2\lesssim n\Delta^{-1}+\Delta^{-2}(n\Delta)^{2H+\varepsilon}.$$
Applying \eqref{eq:generalconcentration} to $\tilde{x}=nx$  leads to the announced bound.
\end{proof}}
\TCR{
\begin{lemma}\label{lem:conc245} Let $T$ be a positive integer. Let $M_k=\ES[F_{T,\mes}|{\cal F}_k]$ for $k\ge0$. 
\begin{itemize}
    \item There exists $C>0$ such that for every $k\in \llbracket 1,T\rrbracket$,
$$\ES[e^{\lambda |M_k-M_{k-1}|}|{\cal F}_{k-1}]\le e^{C\lambda^2\|F\|_\infty^2(a_{k,\mes}^2+b_{k,\mes}^2)}.$$
\item For all $\varepsilon>0$, there exists $C_\varepsilon>0$ such that
\begin{equation*}\label{eq:fOOO}
 \ES[e^{\lambda |\ES[F_{T,\mes}|{\cal F}_0]-\ES[F_{T,\mes}]|}]\le e^{C_\varepsilon\lambda^2\|F\|_\infty^2(a_{0,\mes}^2+b_{0,\mes,\varepsilon}^2)}.
 \end{equation*}
\end{itemize}
\end{lemma}}
\TCR{
\begin{proof} By \cref{lem:gaussianconc} below, it is enough to prove that there exists $\alpha>0$ such that 
for every $\eta$ such that if $\eta\|F\|_\infty^2(a_{k,\mes}^2+b_{k,\mes}^2)
\le \alpha$ for $k\in\rrbracket 0,T\llbracket $, then,
\begin{equation}\label{eq:gausscontrol1}
\ES[e^{\eta (M_k-M_{k-1})^2}|{\cal F}_{k-1}]\le C <+\infty,
    \end{equation}
 where $C$ is a deterministic constant   and, 
    \begin{equation}\label{eq:gausscontrol2}
        \ES[e^{\eta(\ES[F_{T,\mes}|{\cal F}_0]-\ES[F_{T,\mes}])^2}]<+\infty. 
    \end{equation}
For \eqref{eq:gausscontrol1}, we appeal to \eqref{eq:aveck??} which yields:
%\TCR{\int \frC(W^{-}-{w^{-}},\epsilon)\PE_{W^{-}}(dw^{-})T^{H+\frac{\varepsilon}{2}}}\right).
$$ |M_k-M_{k-1}|\ar{\lesssim}\|F\|_\infty \int \left( a_{k,\mes}{\cal T}(k,W,w)+b_{k,\mes}{\|w^{k-1}-W^{k-1}\|_{\infty,[k-1,k]}}\right)\PE_{W}(dw).
$$
Thus, 
$$\ES[e^{\eta (M_k-M_{k-1})^2}|{\cal F}_{k-1}]\lesssim \ES[e^{2\eta\|F\|_\infty^2\left(a_{k,\mes}^2 U_1^2+b_{k,\mes}^2 U_2^2\right)}]$$
with 
$$ U_1=\int \sup_{t\in[0,1]}|\int_0^{t} (t-s)^{H-\frac{1}{2}} d(W-w)_s|\PE_W(dw)$$
and
$$ U_2=\int \|w-W\|_{\infty,[0,1]}\PE_W(dw).$$
Note that we above used the stationarity of the increments.
Using Cauchy-Schwarz inequality, it remains to prove that $U_1$ and $U_2$ are subgaussian. For $U_2$, it is enough to show that there exists $C>0$ such that 
$$ \ES[e^{C (\|W\|_{\infty,[0,1]})^2}]<+\infty.$$
This property is classical (in one-dimension, it is a consequence of the fact that $\sup_{t\in[0,1]} W_t\sim |W_1|$ combined with the symmetry of the Brownian motion). For $U_2$, an integration by parts implies that
$$\int_0^{t} (t-s)^{H-\frac{1}{2}} dW_s=-t^{H-\frac{1}{2}}W_t+\left(H-\frac{1}{2}\right)\int_0^t (t-s)^{H-\frac{3}{2}}(W_s-W_t)ds.$$
Thus, for (small) $\varepsilon>0$
\begin{align*}
\sup_{t\in[0,1]}|\int_0^{t} (t-s)^{H-\frac{1}{2}} dW_s|&\le \sup_{t\in[0,1]} t^{H-\frac{1}{2}}|W_t|+
C_{H,\varepsilon}\sup_{0\le s<t<1} \frac{|W_s-W_t|}{(t-s)^{\frac{1}{2}-\varepsilon}}\\
&\le(1+C_{H,\varepsilon}) \underbrace{\sup_{0\le s<t<1} \frac{|W_s-W_t|}{(t-s)^{\frac{1}{2}-\varepsilon}}}_{\|W\|_{\frac{1}{2}-\varepsilon,[0,1]}},
\end{align*}
where in the second line we used that $t^{H-\frac{1}{2}}|W_t|\le t^{H-\varepsilon}\|W\|_{\frac{1}{2}-\varepsilon,[0,1]}$. The subgaussianity of $U_2$ then follows from that of $\|W\|_{\frac{1}{2}-\varepsilon,[0,1]}$. This property is a consequence of \cite[Theorem A.19]{frizvictoir}.
For \eqref{eq:gausscontrol2}, we use \eqref{TVlongtimes;;} which leads to: 
$$\TCR{|}\ES[F_{T,\mes}|{\cal F}_{0}]-\ES[F_{T,\mes}]\TCR{|}\lesssim \|F\|_\infty \left(a_{0,\mu}\int|X_0-y|\pi(dy)+b_{0,\mu,\varepsilon}U_3\right)$$
with 
$$U_3=\int \frC(W^{-}-{w^{-}},\epsilon)\PE_{W^{-}}(dw^{-}).$$
Once again, it is now enough to prove that $U_3$ is subgaussian. To this end, we remark that
$$\frC(W^{-},\epsilon)=\sup_{s\le0}\frac{|W_s^{-}|}{(1\vee |s|)^{\frac{1+\varepsilon}{2}}}\le \sup_{s\in[-1,0]} |W_s^{-}|+
\sup_{s\le -1}\frac{|W_s^{-}|}{|s|^{\frac{1+\varepsilon}{2}}}. $$
By symmetry (and the fact that $\|W\|_{\infty,[0,1] }$ is subgaussian), this implies that we need to prove the subgaussianity of $\sup_{s\ge 1}\frac{|W_s|}{s^{\frac{1+\varepsilon}{2}}}$. At the price of considering the coordinates, we can focus on the one-dimensional case. Setting $u=1/s$ and using that $(u W_{\frac{1}{u}})_{u\ge0}$ is a Brownian motion, we deduce that
$$\sup_{s\ge 1}\frac{|W_s|}{s^{\frac{1+\varepsilon}{2}}}\sim \sup_{u\in[0,1]} \frac{|W_u|}{u^{\frac{1-\varepsilon}{2}}}\le \|W\|_{\frac{1-\varepsilon}{2},[0,1]},$$
whose subgaussianity has already been recalled previously. This concludes the proof.
\end{proof}}

\TCR{
\begin{lemma}\label{lem:gaussianconc} Let $X$ be a non-negative random variable and $\eta>0$ such that $\ES[e^{\eta X^2}]<+\infty$. Then, for every $\lambda\ge0$,
$$\ES[e^{\lambda X}]\le 1+\frac{\lambda\sqrt{\pi}}{\sqrt{\eta}} e^{\frac{\lambda^2}{4\eta}} \le C  e^{\frac{\lambda^2}{2\eta}} $$
where $C$ is a universal constant.
\end{lemma}
\begin{proof} We have:
$$\ES[e^{\lambda X}]=1+\lambda\int_0^{+\infty} e^{\lambda u}\PE(X>u) du.$$
Since $\ES[e^{\eta X^2}]<+\infty$, the Markov inequality implies that
$$\PE(X>u)\le e^{-\eta u^2}.$$
Then, standard computations imply that 
$$ \int_0^{+\infty} e^{\lambda u}\PE(X>u) du \le \sqrt{\pi}{\eta} e^{\frac{\lambda^2}{4\eta}}.$$
This yields the first bound. The second follows thanks to the elementary inequality $xe^{\frac{x^2}{4}}\le e^{\frac{x^2}{2}}$ for every $x\ge0$.
\end{proof}
}
%\begin{proof}[Proof of \cref{th: upper bounddiscrete}]
%\textcolor{blue}{To be added}
%\end{proof}
\TCR{
\begin{proof}[Proof of \cref{th: rate start discrete}]
The proof of \cref{th: rate start discrete} follows the same line of reasoning as that of Theorem~\ref{th: rate start}, with $T = n\Delta$. It relies on a standard bias–variance decomposition, together with the bias bound in Proposition~\ref{prop: bias}. Note that, although the proof of this proposition was given for the estimator based on continuous observations, it remains valid in the discrete observation setting (see Proposition~2 of~\cite{adaptive}). The bound on the variance is instead provided by Proposition~\ref{th: upper bounddiscrete}, which replaces the variance bound in \cref{th:  upper bound}, used in the proof of Theorem~\ref{th: rate start}. The proof is therefore omitted.
\end{proof}}

\TCR{
\section{Proof adaptive procedure}{\label{s: proof adaptive}}}

\TCR{
This section is devoted to the proofs of the results stated in Section~\ref{s:adaptive}, concerning the adaptive procedure. In particular, we provide the proofs of Theorem~\ref{th: rate adaptive} and Corollary~\ref{cor: rate adaptive final}.}

\vskip 12pt

\begin{proof}[Proof of \cref{th: rate adaptive}]
\TCR{
The adaptive procedure crucially relies on the following bound on the expectation of $A_n(h,x)$, defined in \eqref{eq:Anhx}, which we state here and prove below, valid for any $h \in \mathcal{H}_n$:
\begin{equation}\label{eq: claim adaptive}
\E[A_n(h,x)] \lesssim B_n(h) + (n \Delta)^{c_1} e^{-c_2 (\log(n \Delta))^{c_3}},   
\end{equation}
for some positive constants $c_1, c_2$, and $c_3 > 1$. 
}

\TCR{
The fact that the bound \eqref{eq: claim adaptive} implies the result stated in \cref{th: rate adaptive} follows from standard algebraic arguments and is therefore omitted here. A detailed proof of this implication can be found in Section~6 of \cite{Chapitre4}, where the bound above appears as Proposition~5, and its connection to the adaptive procedure is established in the proof of Theorem~1 therein. This reasoning is now classical: even in the proof of the adaptive procedure in \cite{Minimax} (see Theorems~3 and~4), the verification that \eqref{eq: claim adaptive} yields the desired convergence rate for the data-driven estimator is omitted for brevity.
}

\TCR{
Let us now turn to the proof of~\eqref{eq: claim adaptive}.  
By adding and subtracting $\pi_{(h,\eta), n}$ and $\pi_{\eta, n}$, and using the triangle inequality, we obtain, for any $h \in \mathcal{H}_n$,  
\begin{align*}
A_n(h,x) 
&\le \sup_{\eta \in \mathcal{H}_n} \Big[ (|\check{\pi}_{(h,\eta), n}(x) - \pi_{(h,\eta), n}(x)|^2 - \tfrac{V_n(\eta)}{2})_+ 
+ |\pi_{(h,\eta), n}(x) - \pi_{\eta, n}(x)|^2  \\
& \hspace{3.5cm} + (|\check{\pi}_{\eta, n}(x) - \pi_{\eta, n}(x)|^2 - \tfrac{V_n(\eta)}{2})_+ \Big] \\
&=: \sup_{\eta \in \mathcal{H}_n} \big[ I_1^{h, \eta}(x) + I_2^{h, \eta}(x) + I_3^{\eta}(x) \big].
\end{align*}
We now analyze the three terms above separately.  
We begin with $I_2^{h, \eta}(x)$, which is the simplest one.  
By definition,
\[
|\pi_{(h,\eta), n}(x) - \pi_{\eta, n}(x)|
= |\mathbb{K}_\eta \star (\pi_h - \pi)(x)|
\le \|\mathbb{K}_\eta\|_{L^1} \, \|\pi_h - \pi\|_{\infty}
\lesssim B_n(h),
\]
where we used the boundedness of $\|\mathbb{K}_\eta\|_{L^1}$ and the definition of $B_n(h)$.}

\TCR{
To control both $I_1^{h, \eta}(x)$ and $I_3^{\eta}(x)$, the main tool will be the concentration inequality stated in Proposition~\ref{th: upper bounddiscrete}.  
Let us start with $I_3^{\eta}(x)$. We have
\begin{align*}
\sup_{\eta \in \mathcal{H}_n} I_3^{\eta}(x) 
&\le \sum_{\eta \in \mathcal{H}_n} \int_0^\infty 
\mathbb{P}\big( (|\check{\pi}_{\eta, n}(x) - \pi_{\eta, n}(x)|^2 - \tfrac{V_n(\eta)}{2})_+ \ge t \big) \, dt \\
&\le \sum_{\eta \in \mathcal{H}_n} \int_0^\infty 
\mathbb{P}\big( |\check{\pi}_{\eta, n}(x) - \pi_{\eta, n}(x)| \ge ( \tfrac{V_n(\eta)}{2} + t )^{1/2} \big) \, dt.
\end{align*}
Recall that 
\[
\check{\pi}_{\eta, n}(x) = \frac{1}{n} \sum_{i = 0}^{n - 1} \mathbb{K}_\eta(x - X_{i \Delta}), 
\qquad 
\pi_\eta = \mathbb{E}[\check{\pi}_{\eta, n}(x)],
\]
so that the concentration inequality in Proposition~\ref{th: upper bounddiscrete} applies, yielding
\begin{align*}
\sup_{\eta \in \mathcal{H}_n} I_3^{\eta}(x)
&\le \sum_{\eta \in \mathcal{H}_n} \int_0^\infty 
\exp\!\left( - \frac{(t + \tfrac{V_n(\eta)}{2})(n \Delta)^{1 \wedge (2 - 2H + \epsilon)}}{C_\epsilon \|\mathbb{K}_\eta\|_\infty^2} \right) dt \\
&= \sum_{\eta \in \mathcal{H}_n} 
\exp\!\left( - \frac{1}{2 C_\epsilon (\prod_{l = 1}^d \eta_l)^4} \right)
\int_0^\infty 
\exp\!\left( - \frac{t (n \Delta)^{1 \wedge (2 - 2H + \epsilon)}}{C_\epsilon \|\mathbb{K}_\eta\|_\infty^2} \right) dt,
\end{align*}
where we used the definition of $V_n(\eta)$ and the fact that 
$\|\mathbb{K}_\eta\|_\infty \lesssim \prod_{l = 1}^d \eta_l$. Computing the integral explicitly gives
\[
\sup_{\eta \in \mathcal{H}_n} I_3^{\eta}(x)
\le c \sum_{\eta \in \mathcal{H}_n}
\frac{\|\mathbb{K}_\eta\|_\infty^2}{(n \Delta)^{1 \wedge (2 - 2H + \epsilon)}}
\exp\!\left( - \frac{1}{2 C_\epsilon (\prod_{l = 1}^d \eta_l)^4} \right).
\]
Since for all $\eta \in \mathcal{H}_n$ we have 
$\prod_{l = 1}^d \eta_l \le (1 / \log(n \Delta))^{1/4 + a}$, it follows that
\[
\exp\!\left( - \frac{1}{2 C_\epsilon (\prod_{l = 1}^d \eta_l)^4} \right)
\le \exp\!\left( - \frac{(\log(n \Delta))^{1 + a/4}}{2 C_\epsilon} \right).
\]
Moreover, according to \eqref{eq: Hn pol}, $\mathcal{H}_n$ is assumed to be polynomial in $n \Delta$, which implies
\[
\sup_{\eta \in \mathcal{H}_n} I_3^{\eta}(x)
\lesssim (n \Delta)^{c_1} e^{-c_2 (\log(n \Delta))^{c_3}},
\]
for some constants $c_1, c_2 > 0$ and $c_3 > 1$, as desired.}

\TCR{
The bound on $\sup_{\eta \in \mathcal{H}_n} I_1^{h, \eta}(x)$ follows the same reasoning as that for $\sup_{\eta \in \mathcal{H}_n} I_3^{\eta}(x)$.  
The only difference lies in $\|\mathbb{K}_\eta\|_\infty$ being replaced by $\|\mathbb{K}_h \star \mathbb{K}_\eta\|_\infty$, which satisfies however
\[
\|\mathbb{K}_h \star \mathbb{K}_\eta\|_\infty 
\le c \|\mathbb{K}_\eta\|_\infty \|\mathbb{K}_h\|_{L^1} 
\lesssim \|\mathbb{K}_\eta\|_\infty,
\]
since the $L^1$-norm of the kernel function is uniformly bounded by a constant.  
It then follows immediately that the same bound established for $\sup_{\eta \in \mathcal{H}_n} I_3^{\eta}(x)$ also holds for $\sup_{\eta \in \mathcal{H}_n} I_1^{h, \eta}(x)$.  
This completes the proof of~\eqref{eq: claim adaptive} and, consequently, the proof of Theorem~\ref{th: rate adaptive}.}

\end{proof}

\begin{proof}[Proof of Corollary \ref{cor: rate adaptive final}]
\TCR{
Recall that, according to Theorem~\ref{th: rate start discrete}, the rate-optimal choice of bandwidth is given by  
\begin{equation}\label{eq: h disc opt}
h_l(n \Delta) = \left(\frac{1}{n \Delta}\right)^{-\frac{\bar{\beta}}{2 \beta_l (\bar{\beta} + d)} (1 \wedge (2 - 2H + \epsilon))}
\qquad \forall l \in \{1, \ldots, d\}.
\end{equation}
We want this bandwidth to belong to the set of candidate bandwidths $\mathcal{H}_n$ defined in~\eqref{eq: bandwidths pol}, ensuring that it achieves the optimal bias–variance trade-off appearing on the right-hand side of~\eqref{eq: result adap} in Theorem~\ref{th: rate adaptive}.  
To this end, we verify that the defining condition of $\mathcal{H}_n$ is satisfied, namely
\[
\prod_{l = 1}^d h_l(n \Delta) \le \left(\frac{1}{\log(n \Delta)}\right)^{\frac{1}{4} + a}.
\]
Observe that, for $h_l$ as in~\eqref{eq: h disc opt},
\[
\prod_{l = 1}^d h_l(n \Delta) 
= \left(\frac{1}{n \Delta}\right)^{-\frac{\bar{\beta}}{2 (\bar{\beta} + d)} (1 \wedge (2 - 2H + \epsilon)) \sum_{l = 1}^d \frac{1}{\beta_l}}
= \left(\frac{1}{n \Delta}\right)^{-\frac{d}{2 (\bar{\beta} + d)} (1 \wedge (2 - 2H + \epsilon))},
\]
where we have used the definition of $\bar{\beta}$.
It follows that, for $T = n \Delta$ sufficiently large, this product is indeed smaller than 
$\left(\frac{1}{\log(n \Delta)}\right)^{\frac{1}{4} + a}$ for any $a > 0$.}

\TCR{
However, to guarantee that the bandwidths defined in~\eqref{eq: h disc opt} actually belong to the discrete set $\mathcal{H}_n$, we require them to be of the form 
$h_l = \frac{1}{z_l}$ with $z_l \in \{1, \ldots, \lfloor n \Delta \rfloor\}$, which does not necessarily hold in general.  
To address this, we define instead the discretized bandwidths
\[
\tilde{h}_l(n \Delta) = \frac{1}{\lfloor (n \Delta)^{-\frac{\bar{\beta}}{2 \beta_l (\bar{\beta} + d)} (1 \wedge (2 - 2H + \epsilon))} \rfloor}, 
\qquad \forall l \in \{1, \ldots, d\},
\]
which are asymptotically equivalent to $h_l(n \Delta)$ (and hence yield the same rate), but satisfy 
$\tilde{h}(n \Delta) := (\tilde{h}_1(n \Delta), \ldots, \tilde{h}_d(n \Delta)) \in \mathcal{H}_n$ by construction.  
From Theorem~\ref{th: rate start discrete} we then deduce that
\[
\arg \inf_{h \in \mathcal{H}_n} \big(B_n(h) + V_n(h)\big) = \tilde{h},
\]
and substituting this into Theorem~\ref{th: rate adaptive} completes the proof of the corollary.}

\end{proof}

\section{Technical results}{\label{s: tech}}

\begin{lemma}\label{lem:stationary_density_bounded}
    Under \ref{hyp1} the Lebesgue density $\pi$ of the stationary measure satisfies
    \begin{equation*}
        \|\pi \|_\infty < \infty.
    \end{equation*}
\end{lemma}
\begin{proof}
    This is a consequence of \cite[Theorem 1.1]{LPS22} with $k=0$.
\end{proof}

\begin{lemma}\label{lem:appenA} Assume \ref{hyp1}. Let $\varpi>0$. Let $\varsigma $ be a ${\cal C}^2$-function on $\ER_+$. Let $(x_t)_{t\ge0}$ be a continuous function and consider the ode 
$$\dot{\rho}_t=b(x_t+\rho_t)-b(x_t)+\dot{\varsigma}_t-\varphi_t$$
with
$$\varphi_t= {\varpi}{|\rho_t|}^{-\frac{1}{2}}\rho_t+\lambda \rho_t+\dot{\varsigma}_t.$$
Then, for any starting point $\rho_0$, the ode admits a unique solution on $\ER_+$ and if $\varpi=2|\rho_0|^{\frac{1}{2}}$, then $\rho_1=0$,
\bqn\label{eq:phiinfty1}
\|\varphi_t\|_{\infty,[0,1]}\lesssim |\rho_0|+\|\dot{\varsigma}_t\|_{\infty,[0,1]}
\eqn
and 
\begin{equation*}\label{eq:phidotinfty}
\|\dot{\varphi}_t\|_{\infty,[0,1]}\lesssim |\rho_0|+\|\ddot{\varsigma}_t\|_{\infty,[0,1]}.
\end{equation*}
  %\begin{remark} The above lemma may be surprising since the ode does not depend on $\varsigma$ (since we add and substract $\dot{\varsigma}$).  Nevertheless, it must be understood as follows. To the ode $\dot{z}_t=b(x_t+z_t)-b(x_t)+\dot{\varsigma}_t$, we had a \textit{control} $\varphi$ which is such that $\rho_1$
%\end{remark} 
\end{lemma}
\begin{remark} The above lemma may be surprising since the ode does not depend on $\varsigma$ (as we add and substract $\dot{\varsigma}$).  Nevertheless, it must be understood as follows. To the ode $\dot{\rho}_t=b(x_t+\rho_t)-b(x_t)+\dot{\varsigma}_t$, we \ar{add} a \textit{control} $\varphi$ which is such that the modified $\rho$ satisfies $\rho_1=0$. 
\end{remark}
\begin{proof}
Set $z_t=|\rho_t|^2$. By construction and Assumption \ref{hyp1}, if $\rho$ is a solution, then
\begin{align*}
\dot{z}_t=2\langle b(x_t+\rho_t)-b(x_t)-{\varpi}{|\rho_t|}^{-\frac{1}{2}}\rho_t-\lambda \rho_t, \rho_t\rangle\le   -2{\varpi}{|\rho_t|}^{\frac{3}{2}}= -2{\varpi} z(t)^{\frac{3}{4}}.
\end{align*}
Set $t_0:=\inf\{t\ge0, \rho_t=0\}$. If $\rho$ is a solution, then one deduces from what precedes that $\partial_t (|\rho_t|^{\frac{1}{2}})\le - \frac{\varpi}{2}$ on $[0,t_0)$  so that on $\ER_+$,
\bqn\label{eq:controlrhott}
|\rho_t|^{\frac{1}{2}}\le \max(|\rho_0|^{\frac{1}{2}}-{\frac{\varpi t}{2} } ,0).
\eqn
Since the ode has the form $\dot{\rho}=F(\rho)$ with a function $F$ which is locally Lispchitz continuous on $\ER^d\backslash\{0\}$ and such that $F(0)=0$, we deduce from what precedes that the equation has a unique solution on $\ER_+$. Furthermore, it also follows from \eqref{eq:controlrhott} that if $\varpi=2|\rho_0|^{\frac{1}{2}}$, then $|\rho_1|=0$ and  the fact that $|\rho_t|\le |\rho_0|$ (by \eqref{eq:controlrhott}) implies \eqref{eq:phiinfty1}. 
Finally, still with $\varpi=2|\rho_0|^{\frac{1}{2}}$, one checks that
$$|\dot{\varphi}_t|\lesssim |\rho_0|+ |\rho_0|^{\frac{1}{2}} |\rho_t|^{-\frac{1}{2}}|\dot{\rho}_t|+|\dot{\rho}_t|+|\ddot{\varsigma}_t|.$$
But since $b$ is Lipschitz continuous, we have
$$|\dot{\rho}_t|=\tcr{|b(x_t+\rho_t)-b(x_t)+\dot{\varsigma}_t-{\varpi}{|\rho_t|}^{-\frac{1}{2}}\rho_t-\lambda \rho_t|}\lesssim |\rho_t|+|\rho_0|^{\frac{1}{2}}|\rho_t|^{\frac{1}{2}}.$$
%\textcolor{blue}{I am a bit lost, I think we are also using that, by definition of ${\rho}_t$ and the fact that $b$ is Lipschitz}
%\begin{align*}
%\textcolor{blue}{|\dot{\rho}_t|}& \textcolor{blue}{\lesssim |\rho_t| + |\dot{\varsigma}_t| + |\varphi_t| }\\
%& \textcolor{blue}{\lesssim |\rho_t| + |\rho_0| + \|\dot{\varsigma}_t\|_{\infty,[0,1]}, }
%\end{align*}
%\textcolor{blue}{where the last is a consequence of \eqref{eq:phiinfty1}. No? In this case we would have the extra term in red, in the statement.}
The second bound follows.
\end{proof}

\begin{proof}[Proof of \cref{prop:TVTV}]
%\subsection{Proof of \cref{prop:TVTV} }\label{sec:prooftvtv}
With the notations of \cref{cor:L2bound}, let 
$$
\begin{cases}
%\begin{equation}\label{eq:innovation_sde}
  \Phi_t(\ell(x,w))=\ell_t(x,w)+\int_0^t b\big(\Phi_s(\ell(x,w))\big)\,ds+\sigma\int_0^t (t-s)^{H-\frac{1}{2}} dW_s,\qquad\\
   \widetilde{\Phi}_t(\ell(y,\tilde{w}))=\ell_t(y,\tilde{w})+\int_0^t b\big(\widetilde{\Phi}_s(\ell(y,\tilde{w}))\big)\,ds+\sigma\int_0^t (t-s)^{H-\frac{1}{2}} d\widetilde{W}_s,\qquad t\geq 0,
   \end{cases}
   $$
   where $(W,\widetilde{W})$ denotes a coupling of Brownian motions on $\ER_+$. In order to shorten the notations, we will usually write $\Phi_t$ and $\widetilde{\Phi}_t$ instead of $\Phi_t(\ell(x,w))$ and  $ \widetilde{\Phi}_t(\ell(x,\tilde{w}))$.  Let $\tau\ge 1$. The idea of the proof is to bound 
$$\PE(\Phi_{\tau+1}\neq \widetilde{\Phi}_{\tau+1}|{\cal F}_{\tau})$$
by a coupling argument and then to use \cref{cor:L2bound} to deduce the announced result.
   %For any $\delta>0$,
  %\begin{equation}\label{eq:tvbounddecomp}
  %\PE(\Phi_{\tau+1}\neq \widetilde{\Phi}_{\tau+1})\le \PE\left(\Phi_{\tau+1}\neq \widetilde{\Phi}_{\tau+1}, |\Phi_{\tau}-\widetilde{\Phi}_{\tau}|\le \delta\right)+\PE( |\Phi_{\tau}-\widetilde{\Phi}_{\tau}|> \delta).
  %\eqn
%First, 
%if we now assume that 
%$$W_t=\widetilde{W}_t\quad \forall t\in[0,\tau],$$
%then, we deduce from \cref{cor:L2bound} and the Markov inequality that
%$$\PE( |\Phi_{\tau}-\widetilde{\Phi}_{\tau}|> \delta)\lesssim \frac{1}{\delta}\left(e^{-c \tau} |x-y|^2+\tau^{H-\frac{5}{2}}\right).$$

%To bound the first term in \eqref{eq:tvbounddecomp}, we will use a ``Girsanov'' coalescent coupling. 
Assume that on a subset of $\Omega$,
\bqn\label{eq:girsident}
\widetilde{W}_t-\widetilde{W}_\tau=W_t-W_\tau+\int_\tau^{.} \psi^\tau(s) ds\quad \forall t\in[\tau,\tau+1].
\eqn
{Under appropriate assumptions (see \cite[Lemma 4.2]{Hairer} for details), this implies that 
\begin{equation*}\label{eq:deriveW}
\int_\tau^t (t-s)^{H-\frac{1}{2}} d\widetilde{W}_s=\int_\tau^t (t-s)^{H-\frac{1}{2}} d{W}_s+\int_{\tau}^t \varphi^\tau(s) ds\quad \textnormal{on $[\tau,\tau+1]$},
\end{equation*}
with 
$$\varphi^\tau(t)=\frac{d}{dt}\int_\tau^t (t-s)^{H-\frac{1}{2}} \psi^\tau(s) ds,$$
and with the following reverse property: 
\begin{equation}\label{eq:corpsiphi}
\psi^\tau(t)= \gamma_H\frac{d}{dt}\int_{\tau}^t (t-u)^{\frac{1}{2}-H} \varphi^\tau(u) du,\quad t\in[\tau,\tau+1],\quad \gamma_H\in\ER.
\end{equation}}
%\textcolor{blue}{Shall we write something more about $\varphi^\tau$? \\}
With these notations, one remarks that for all $t\in[\tau,\tau+1]$,
$$\tilde{\Phi}_t-\Phi_t=\tilde{\Phi}_\tau-\Phi_\tau+\int_\tau^t b(\tilde{\Phi}_s)-b({\Phi}_s) ds+ \sigma \int_\tau^t \varphi^\tau(s) ds+ \varsigma_t^\tau $$
with 
\begin{align*}
\varsigma^\tau_t&= \ell_t(y,\tilde{w})-\ell_\tau(y,\tilde{w})-(\ell_t(x,w)-\ell_\tau(x,w))\\
&=\left(\frac{1}{2}-H\right)\int_{-a}^0 \left((t-s)^{H-\frac{3}{2}}-(\tau-s)^{H-\frac{3}{2}} \right)(\tilde{w}_s-w_s) ds,
\end{align*}
with $a=\infty$ in the general case or $a=1$ in the particular case where $w$ and $\tilde{w}$ are such that $w_t=\tilde{w}_t$ on $(-\infty,-1]$.
One checks that for all $t\in[\tau,\tau+1]$,
$$\dot{\varsigma}^\tau_t=\left(\frac{1}{2}-H\right)\left(\frac{3}{2}-H\right)\int_{-a}^0 (t-s)^{H-\frac{5}{2}}(\tilde{w}_s-w_s) ds$$
and hence, acting as in previous proposition in order to obtain \eqref{eq: 1} and \eqref{eq: 2}, respectively, we obtain that for every $\tau\le t\le \tau+1$,
\begin{align*}
|\dot{\varsigma}^\tau_t|\lesssim\begin{cases} \frC(w-\tilde{w},\epsilon)  \tau^{H-1+\frac{\varepsilon}{2}}&\textnormal{if $a=\infty$}\\
 \|\tilde{w}-w\|_{\infty,[-1,0]} \tau^{H-\frac{5}{2}}&\textnormal{if $a=1$.}
\end{cases}
\end{align*}
Similarly, 
\begin{align*}
|\ddot{\varsigma}^\tau_t|\lesssim\begin{cases} \frC(w-\tilde{w},\epsilon)  \tau^{H-2+\frac{\varepsilon}{2}}&\textnormal{if $a=\infty$}\\
 \|\tilde{w}-w\|_{\infty,[-1,0]} \tau^{H-\frac{7}{2}}&\textnormal{if $a=1$.}
\end{cases}
\end{align*}

%
%%& \lesssim\|\tilde{w}-w\|_{\infty,[-1,0]}\left(|\tau^{H-\frac{1}{2}}-t^{H-\frac{1}{2}}|+|(\tau+1)^{H-\frac{1}{2}}-\tau^{H-\frac{1}{2}}|\right)\\
%%&\lesssim\|\tild
%% one checks that
%\begin{align*}
%|\varsigma_t|&\lesssim \|\tilde{w}-w\|_{\infty,[-1,0]} \int_{-1}^0 (\tau-s)^{H-\frac{3}{2}}-(t-s)^{H-\frac{3}{2}} ds\\
%%& \lesssim\|\tilde{w}-w\|_{\infty,[-1,0]}\left(|\tau^{H-\frac{1}{2}}-t^{H-\frac{1}{2}}|+|(\tau+1)^{H-\frac{1}{2}}-\tau^{H-\frac{1}{2}}|\right)\\
%&\lesssim\|\tilde{w}-w\|_{\infty,[-1,0]} \tau^{H-\frac{5}{2}}
%\end{align*}
%and that
%\begin{align*}
%|\dot{\varsigma}_t|&\lesssim \|\tilde{w}-w\|_{\infty,[-1,0]} \int_{-1}^0 (\tau-s)^{H-\frac{5}{2}}-(t-s)^{H-\frac{5}{2}} ds\\
%%& \lesssim\|\tilde{w}-w\|_{\infty,[-1,0]}\left(|\tau^{H-\frac{1}{2}}-t^{H-\frac{1}{2}}|+|(\tau+1)^{H-\frac{1}{2}}-\tau^{H-\frac{1}{2}}|\right)\\
%&\lesssim\|\tilde{w}-w\|_{\infty,[-1,0]} \tau^{H-\frac{7}{2}}.
%\end{align*}
%%$$\|\varsigma\|_{\Omega_\beta}\lesssim $$
%Thus, for any $\beta<\frac{5}{2}-H$,
%$$\|\varsigma\|_{\Omega_\beta}\lesssim \|\tilde{w}-w\|_{\infty,[-1,0]} $$
%where
%\begin{equation}
%\|f\|_{\Omega_{\beta}}=\sup_{t\ge1} (t^\beta |f(t)|+t^{1+\beta} |\dot{f}(t)|.
%\end{equation}

By \cref{lem:appenA} (stated and proved in \cref{s: tech}) applied with $\rho_0=\tilde{\Phi}_\tau-{\Phi}_\tau$, ${\varsigma}={\varsigma}^\tau(\tau+.)$, $x(t)=\tilde{\Phi}_{\tau+t}-{\Phi}_{\tau+t}$, we get the existence of an adapted $\varphi^\tau$ such that if \eqref{eq:girsident} holds, then 
$\Phi_{\tau+1}\,\tcr{=} \,\widetilde{\Phi}_{\tau+1}$ and $\varphi^\tau$ satisfies the following bounds:
\begin{equation*}\label{eq:controlphitau}
\|\varphi^\tau\|_{\infty,[\tau,\tau+1]}\lesssim |\tilde{\Phi}_\tau-{\Phi}_\tau|+\begin{cases} \frC(w-\tilde{w},\epsilon)  \tau^{H-1+\frac{\varepsilon}{2}}&\textnormal{if $a=\infty$}\\
 \|\tilde{w}-w\|_{\infty,[-1,0]} \tau^{H-\frac{5}{2}}&\textnormal{if $a=1$}
\end{cases}
\end{equation*}
and
\begin{equation*}{\label{eq: varphi dot}}
\|\dot{\varphi}^\tau\|_{\infty,[\tau,\tau+1]}\lesssim |\tilde{\Phi}_\tau-{\Phi}_\tau|+\begin{cases} \frC(w-\tilde{w},\epsilon)  \tau^{H-2+\frac{\varepsilon}{2}} &\textnormal{if $a=\infty$}\\
 \|\tilde{w}-w\|_{\infty,[-1,0]} \tau^{H-\frac{7}{2}} &\textnormal{if $a=1$.}
\end{cases}
\end{equation*}
%
%(a slight adaptation of) \cite[Lemma 3.21]{LS22a}, one is able to calibrate an $({\cal F}_t)_{t\in[\tau,\tau+1]}$-adapted $\varphi^\tau$ in such a way that if  \eqref{eq:deriveW} holds, then
%$\Phi_{\tau+1}=\tilde{\Phi}_{\tau+1}$ and such that for every $\beta<\frac{5}{2}-H$
%\begin{equation}\label{eq:controlphitau}
%\|\varphi^\tau\|_{\infty,[\tau,\tau+1]}\lesssim |\tilde{\Phi}_\tau-{\Phi}_\tau|+\frac{\|\varsigma\|_{\Omega_\beta}}{\tau^{\beta}}\quad \textnormal{and} 
%\|\dot{\varphi}^\tau\|_{\infty,[\tau,\tau+1]}\lesssim |\tilde{\Phi}_\tau-{\Phi}_\tau|+|\tilde{\Phi}_\tau-{\Phi}_\tau|^{\frac{1}{2}}+\frac{\|\varsigma\|_{\Omega_\beta}}{\tau^{1+\beta}}.
%\end{equation}
By \eqref{eq:corpsiphi}, one checks that
$$ \|\psi^\tau\|_{\infty,[\tau,\tau+1]}\lesssim  \|\varphi^\tau\|_{\infty,[\tau,\tau+1]}+\|\dot{\varphi}^\tau\|_{\infty,[\tau,\tau+1]}\ind{H>1/2}.$$
Hence, 
$$ \|\psi^\tau\|_{\infty,[\tau,\tau+1]}\lesssim   |\tilde{\Phi}_\tau-{\Phi}_\tau|+\begin{cases} \frC(w-\tilde{w},\epsilon)  \tau^{H-1+\frac{\varepsilon}{2}}&\textnormal{if $a=\infty$}\\
 \|\tilde{w}-w\|_{\infty,[-1,0]} \tau^{H-\frac{5}{2}}&\textnormal{if $a=1$.}
\end{cases}
$$

%
%\begin{cases}|\tilde{\Phi}_\tau-{\Phi}_\tau|+\tau^{-\beta}{\|\tilde{w}-w\|_{\infty,[-1,0]}}&\textnormal{if $H<1/2$}\\
%|\tilde{\Phi}_\tau-{\Phi}_\tau|+|\tilde{\Phi}_\tau-{\Phi}_\tau|^{\frac{1}{2}}+\tau^{-\beta}{\|\tilde{w}-w\|_{\infty,[-1,0]}}&\textnormal{if $H>1/2$}.
%\end{cases}
%$$
%we deduce that for any $\beta<\frac{5}{2}-H$,
%$$ \|\psi^\tau\|_{\infty,[\tau,\tau+1]}\lesssim   \begin{cases}|\tilde{\Phi}_\tau-{\Phi}_\tau|+\tau^{-\beta}{\|\tilde{w}-w\|_{\infty,[-1,0]}}&\textnormal{if $H<1/2$}\\
%|\tilde{\Phi}_\tau-{\Phi}_\tau|+|\tilde{\Phi}_\tau-{\Phi}_\tau|^{\frac{1}{2}}+\tau^{-\beta}{\|\tilde{w}-w\|_{\infty,[-1,0]}}&\textnormal{if $H>1/2$}.
%\end{cases}
%$$
By construction, when \eqref{eq:girsident} holds true, then we have $\Phi_{\tau+1}= \widetilde{\Phi}_{\tau+1}$. Hence,
$$\PE\left(\Phi_{\tau+1}\neq \widetilde{\Phi}_{\tau+1}|{\cal F}_{\tau}\right)\le 1-\PE(\widetilde{W}_t-\widetilde{W}_\tau=W_t-W_\tau+\int_\tau^{t} \psi^\tau(s) ds, t\in[\tau,\tau+1]|{\cal F}_{\tau}).$$
Due to the independence of the increments of the Brownian motion, it follows that almost surely,
 $$\inf_{(W,\tilde{W})}\PE\left(\Phi_{\tau+1}\neq \widetilde{\Phi}_{\tau+1}|{\cal F}_{\tau}\right)\le \frac{1}{2}\|\PE_W- \Upsilon^*\PE_W\|_{{\rm TV}}$$
 where $\Upsilon$  is defined by $\Upsilon(w)=w+\int_0^. \psi^\tau(\tau+s) ds$ and $\PE_W$ denotes the Wiener distribution on ${\cal C}([0,1],\ER^d)$. \\

\noindent By Girsanov's Theorem, we know that $\Upsilon^*\PE_W$ is absolutely continuous with respect to
$\PE_W$ with density $D_1$, where $(D_t)_{t\ge0}$ is the true martingale defined by:
$$D_t(w)=\exp\left(\int_0^t \psi^\tau(\tau+s)  dw(s)-\frac{1}{2}\int_0^t |\psi^\tau(\tau+s) |^2 ds\right), \quad t\in[0,1].$$
%for $t\in [0,1],\,\PE_W-a.s.$ \\ 
Thus, by Pinsker inequality, almost surely,
\begin{align*}
\|\PE_W- \Upsilon^*\PE_W\|_{{\rm TV}}&\le
 \sqrt{\frac{1}{2}H(\Upsilon^*\PE_W|\PE_W)}= \left(\frac{1}{2}\int \log( D_1(w)^{-1})\PE_{W}(dw)\right)^{\frac{1}{2}}\\
&\le \frac{1}{2}\left({\int} \int_0^{1} |\psi^\tau(\tau+s)|^2 ds ~{\PE_{W}(dw)}\right)^{\frac{1}{2}}\\
&\lesssim  |\tilde{\Phi}_\tau-{\Phi}_\tau|+\begin{cases} \frC(w-\tilde{w},\epsilon)  \tau^{H-1+\frac{\varepsilon}{2}}&\textnormal{if $a=\infty$}\\
 \|\tilde{w}-w\|_{\infty,[-1,0]} \tau^{H-\frac{5}{2}}&\textnormal{if $a=1$.}
\end{cases}
\end{align*}
%
%$$\PE\left(\Phi_{\tau+1}\neq \widetilde{\Phi}_{\tau+1}, |\Phi_{\tau}-\widetilde{\Phi}_{\tau}|\le \delta\right)\le 1-\PE(\widetilde{W}_t=W_t+\int_\tau^{t} \psi^\tau(s) ds, t\in[\tau,\tau+1],  |\Phi_{\tau}-\widetilde{\Phi}_{\tau}|\le \delta).$$
%Let us denote by ${\cal A}(\tau-1,\tau)$ the set of couplings $(W_t,\tilde{W}_t)_{t\in[\tau-1,\tau]}$ of Brownian motions. By construction,
%\begin{align*}
%\PE(x(\tau)\neq y(\tau)|{\cal F}_{\tau-1})&\le 1-\sup_{(W,\tilde{W})\in {\cal A}(\tau-1,\tau)}\PE((\tilde{W}_t)_{t\in(\tau-1,\tau)}\\
%&=(W_t+\int_{\tau-1}^t g_w(s) ds)_{t\in(\tau-1,\tau)}|{\cal F}_{\tau-1})\\&=\frac{1}{2}\|\PE_W- \Upsilon^*\PE_W\|_{{\rm TV}},
%\end{align*}
%where $\Upsilon$ is defined by $\Upsilon(w)=w+\int_0^. g_w(\tau-1+s) ds$ and $\PE_W$ denotes the Wiener distribution on ${\cal C}([0,1],\ER^d)$. 
%By Girsanov's Theorem, we know that $\Upsilon^*\PE_W$ is absolutely continuous with respect to
%$\PE_W$ with density $D_1$ where $(D_t)_{t\ge0}$ is the true martingale defined by:
%$$D_t(w)=\exp\left(\int_0^t g_w(\tau-1+s) dw(s)-\frac{1}{2}\int_0^t |g_w(\tau-1+s)|^2 ds\right)$$
%for $t\in [0,1],\,\PE_W-a.s.$ \\ 
%Thus, by Pinsker inequality,
%\begin{align*}
%\|\PE_W- \Upsilon^*\PE_W\|_{{\rm TV}}&\le \sqrt{\frac{1}{2}}\left(\int \log( D(w)^{-1})\PE_{W}(dw)\right)^{\frac{1}{2}}\\
%&\le \frac{1}{2}\left({\int} \int_0^{1} |g_w(\tau-1+s)|^2 ds ~{\PE_{W}(dw)}\right)^{\frac{1}{2}} \\
%
%\end{align*}
As a consequence,
\begin{align*}
 \|{\cal L}(\Phi_{\tau+1}-{\cal L}(\widetilde{\Phi}_{\tau+1})\|_{TV}\le \ES[|\tilde{\Phi}_\tau-{\Phi}_\tau|]+\begin{cases} \frC(w-\tilde{w},\epsilon)  \tau^{H-1+\frac{\varepsilon}{2}}&\textnormal{if $a=\infty$}\\
 \|\tilde{w}-w\|_{\infty,[-1,0]} \tau^{H-\frac{5}{2}}&\textnormal{if $a=1$.}
\end{cases}
\end{align*}
To conclude, it is now enough to use \cref{cor:L2bound} \TCR{to control $\ES[|\tilde{\Phi}_\tau-{\Phi}_\tau|]$ and to deduce the result}.
%which implies that
%$$  \ES[|\tilde{\Phi}_\tau-{\Phi}_\tau|]\lesssim  \begin{cases} \frC(w-\tilde{w},\epsilon)  \left(e^{-\frac{c}{2}\tau} |x-y|+ \tau^{H-1+\frac{\varepsilon}{2}}\right)&\textnormal{if $a=\infty$}\\
 %\|\tilde{w}-w\|_{\infty,[-1,0]}\left(e^{-\frac{c}{2}\tau} |x-y|+ \tau^{H-\frac{5}{2}}\right)&\textnormal{if $a=1$.}
%\end{cases}
%$$
\end{proof}
%\noindent \hfill $\square$

\begin{proof}[Proof of \cref{prop: bias}]
Notice that the concept outlined below follows a classical approach, akin to Proposition 2 in \cite{adaptive}, sharing a similar objective. %Because neither the bandwidth $h$ nor the smoothness parameter $\beta$ vary with direction in our scenario, the proof can be significantly streamlined.
Given the definition of the kernel estimator we have introduced and the stationarity of the process, we arrive at the following
\begin{align*}
|\mathbb{E}[\hat{\pi}_{h,T}(x)] - \pi(x)| & = |\int_{\R^d} \mathbb{K}_h(x - y) \pi(y) dy - \pi(x)| \\
& = |\frac{1}{\TCR{\prod_{l = 1}^d h_l}}\int_{\R^d} \prod_{l = 1}^d {K}(\frac{x_l - y_l}{\TCR{h_l}}) \pi(y) dy - \pi(x)| \\
& = |\int_{\R^d} \prod_{l = 1}^d {K}(z_l)[ \pi(x_1 - \TCR{h_1 z_1, ... , x_d - h_d z_d}) - \pi(x)] dz|,
\end{align*}
having applied in the last equality the change of variable $\frac{x_l - y_l}{\TCR{h_l}}= z_l$, together with the fact that $\int_{\R^d} \prod_{l = 1}^d {K}(z_l) dz = 1$, by definition of the kernel function we introduced. 
Then, we can use Taylor formula applied iteratively to the functions $t \mapsto \pi(x_1 - \TCR{h_1} z_1, ... , x_{d-1} - \TCR{h_d} z_{d-1}, t)$, ... $t \mapsto (t, x_2 - \TCR{h_2} z_2, ... , x_d - \TCR{h_d} z_d)$. It implies we can write
\begin{align*}
& \pi(x_1 - \TCR{h_1} z_1, ... , x_d - \TCR{h_d} z_d) 
 = \pi(x) + \sum_{j = 1}^d (\sum_{k = 1}^{\lfloor \TCR{\beta_j} \rfloor - 1}\frac{D_j \pi(\tilde{x}_j)}{k !}(\TCR{h_j} z_j)^k \\
 &\qquad + \frac{D_j^{\lfloor \TCR{\beta_j} \rfloor}\pi(x_1 - \TCR{h_1} z_1, ... , x_j - u_j \TCR{h_j} z_j, x_{j + 1}, ... , x_d )}{\lfloor \TCR{\beta_j} \rfloor !} (\TCR{h_j} z_j)^{\lfloor \TCR{\beta_j} \rfloor}),
\end{align*}
where we introduced $u \in [0, 1]$ and $\tilde{x}_j := (x_1 - \TCR{h_1} z_1, ... , x_{j-1} - \TCR{h_{j - 1}} z_{j-1}, x_{j}, ... , x_d )$. The purpose of the last is to monitor the component we are focusing on during the iterative application of the Taylor formula. Using then that the kernel function is of order $M \ge \TCR{\max_i \beta_i}$ we obtain, for any $j \in \{1, ... , d \}$
$$\int_{\R} K(z_j) \sum_{k = 1}^{\lfloor \TCR{\beta_j} \rfloor - 1} \frac{D_j \pi(\tilde{x}_j)}{k !}(\TCR{h_j} z_j)^k dz_j = 0. $$
Here it is crucial the fact that $\tilde{x}_j$ depends only on $z_l$, for $l < j$. Observe moreover that the fact that the kernel function is of order $M \ge \TCR{\max_i \beta_i}$ also implies
$$\int_{\R^d} \prod_{l = 1}^d K(z_l) \sum_{j = 1}^{d} \frac{(\TCR{h_j} z_j)^{\lfloor \TCR{\beta_j} \rfloor}}{\lfloor \TCR{\beta_j} \rfloor !} D_j^{\lfloor \TCR{\beta_j} \rfloor} \pi(x) dz = 0. $$
Then, we obtain
\begin{align*}
&|\mathbb{E}[\hat{\pi}_{h,T}(x)] - \pi(x)|  \le \int_{\R^d} |\prod_{l = 1}^d {K}(z_l)| \\
&\qquad \times \Big( \sum_{j = 1}^{d} \frac{(\TCR{h_j} z_j)^{\lfloor \TCR{\beta_j} \rfloor}}{\lfloor \TCR{\beta_j} \rfloor !} |D_j^{\lfloor \TCR{\beta_j} \rfloor} \pi(x_1 - \TCR{h_1} z_1, ... , x_j - u_j \TCR{h_j} z_j, x_{j + 1}, ... , x_d ) - \pi(x) | \Big)dz \\
& \le \int_{\R^d} |\prod_{l = 1}^d {K}(z_l)| \sum_{j = 1}^{d} \frac{(\TCR{h_j} z_j)^{\lfloor \TCR{\beta_j} \rfloor}}{\lfloor \TCR{\beta_j} \rfloor !} L|u_j \TCR{h_j} z_j|^{\TCR{\beta_j} - \lfloor \TCR{\beta_j} \rfloor} dz,
\end{align*}
having used that $\pi \in \mathcal{H}_d(\beta, L)$. The proof is concluded once one remarks that 
$$\int_{\R^d} |\prod_{l = 1}^d {K}(z_l)| \sum_{j = 1}^{d} \frac{(\TCR{h_j} z_j)^{\lfloor \TCR{\beta_j} \rfloor}}{\lfloor \TCR{\beta_j} \rfloor !} L|u_j \TCR{h_j} z_j|^{\TCR{\beta_j - \lfloor \beta_j \rfloor}} dz \le c \TCR{h_j^{\beta_j}},$$
with a constant $c$ independent of $x$. 

\end{proof}

\end{document}